\documentclass[a4paper,11pt]{amsart}

\usepackage{amsmath,amsthm,amssymb,stmaryrd} 
\usepackage[a4paper,margin=2.5cm]{geometry}
\usepackage{xcolor}
\usepackage[colorlinks,linkcolor=blue,citecolor=blue]{hyperref}
\usepackage{accents}

\newcommand\ub[1]{\underaccent{\bar}{#1}}

\theoremstyle{plain} 
\newtheorem{theorem}{Theorem}[section]
\newtheorem{lemma}[theorem]{Lemma}
\newtheorem{proposition}[theorem]{Proposition}
\newtheorem{corollary}[theorem]{Corollary} 

\newtheorem{introtheorem}{Theorem}
\newtheorem{introcorollary}[introtheorem]{Corollary}

\theoremstyle{definition} 

\newtheorem{notation}[theorem]{Notation}
\newtheorem{remark}[theorem]{Remark}

\numberwithin{equation}{section}

\usepackage{enumitem}
\setlist[itemize]{itemsep=0ex,label=--}
\setlist[enumerate]{itemsep=0ex,topsep=0ex}

\newcommand{\NN}{\mathbb{N}} 
\newcommand{\ZZ}{\mathbb{Z}} 
\newcommand{\CC}{\mathbb{C}}

\newcommand{\Aa}{\mathcal{A}}

\newcommand{\Ll}{\mathcal{L}}
\newcommand{\FO}{\mathbb{F}O} 
\newcommand{\FU}{\mathbb{F}U} 
\newcommand{\W}{W} 

\DeclareSymbolFont{bbold}{U}{bbold}{m}{n}
\DeclareSymbolFontAlphabet{\mathbbold}{bbold}
\newcommand{\GGamma}{\mathbbold{\Gamma}}

\renewcommand{\dim}{\operatorname{dim}}
\renewcommand{\Re}{\operatorname{Re}}

\newcommand{\Tr}{\operatorname{Tr}}
\newcommand{\tr}{\operatorname{tr}}
\newcommand{\Span}{\operatorname{Span}}
\newcommand{\Corep}{\operatorname{Corep}}
\newcommand{\Hom}{\operatorname{Hom}}
\newcommand{\id}{\mathord{\operatorname{id}}} 
\newcommand{\I}{\operatorname{I}} 
\newcommand{\ts}{\textstyle}
\newcommand{\ccirc}{{\circ\circ}}

\newcommand\itv[5]{\left #1#2,#4\right#5}
\newcommand\x[2]{#1{\scriptstyle #2}}

\title[Maximal amenability in free quantum group factors]{Maximal amenability of the radial subalgebra \\ in free quantum group factors}

\author{Roland Vergnioux}
\address{Roland Vergnioux, Normandie Univ, UNICAEN, CNRS, LMNO, 14000 Caen, France} 
\email{roland.vergnioux@unicaen.fr}

\author{Xumin Wang}
\address{Institute for Advanced Study in Mathematics, Harbin Institute of Technology, Harbin 150001, China; and  Department Of Mathematical Sciences, Seoul National University, Seoul 08826, Republic of Korea} 
\email{xumin.wang1124@gmail.com}

\makeatletter
\@namedef{subjclassname@2020}{%
  \textup{2020} Mathematics Subject Classification}
\makeatother


\begin{document}

\begin{abstract}
  We show that the radial MASA in the orthogonal free quantum group algebra
  $\Ll(\FO_N)$ is maximal amenable if $N$ is large enough, using the Asymptotic
  Orthogonality Property. This relies on a detailed study of the corresponding
  bimodule, for which we construct in particular a quantum analogue of R\u
  adulescu's basis. As a byproduct we also obtain the value of the Puk\'anszky
  invariant for this MASA.
\end{abstract}

\maketitle

\tableofcontents

\section*{Introduction}

The orthogonal free quantum groups $\FO_N$, for $N\in\NN^*$, are discrete
quantum groups which were introduced by Wang \cite{Wang_FreeProd} via their
universal $C^*$-algebra defined by generators and relations:
\begin{displaymath}
  C^*_u(\FO_N) = A_o(N) = C^*(u_{i,j}, 1\leq i,j\leq N \mid u = \bar u, uu^* = u^*u = 1).
\end{displaymath}
Here $u = (u_{i,j})_{i,j}$ is the matrix of generators, $u^*$ is the usual
adjoint in $M_N(C^*_u(\FO_N))$, and $\bar u = (u_{i,j}^*)_{i,j}$. There is a
natural coproduct
$\Delta : C^*_u(\FO_N) \to C^*_u(\FO_N)\otimes C^*_u(\FO_N)$ which encodes the
quantum group structure, and which turns $C^*_u(\FO_N)$ into a Woronowicz
$C^*$-algebra \cite{Woronowicz_CQG}. In particular $C^*_u(\FO_N)$ is equipped
with a canonical $\Delta$-invariant tracial state $h$. In this article we are
interested in the von Neumann algebra
$\Ll(\FO_N) = \lambda(C^*_u(\FO_N))'' \subset B(H)$ generated by the image of
$C^*_u(\FO_N)$ in the GNS representation $\lambda$ associated with $h$. We still
denote $u_{i,j} \in \Ll(\FO_N)$ the images of the generators.

The von Neumann algebras $\Ll(\FO_N)$, and their unitary variants $\Ll(\FU_N)$,
can be seen as quantum, or matricial, analogues of the free group factors
$\Ll(F_N)$. More precisely if we denote $FO_N = (\ZZ/2)^{*N}$, with canonical
generators $a_i$, $1\leq i\leq N$, we have a surjective $*$-homomorphism
$\pi : C^*_u(\FO_N) \to C^*_u(FO_N)$, $u_{i,j} \mapsto \delta_{i,j} a_i$
compatible with coproducts. It turns out that this analogy is fruitful also at
an analytical level: one can show that $\Ll(\FO_N)$ shares many properties with
$\Ll(FO_N)$ and $\Ll(F_N)$, although the existence of $\pi$, which has a huge
kernel, is useless to prove such properties. For instance, $\Ll(\FO_N)$ is non
amenable for $N\geq 3$ \cite{Banica_LibreUnitaire}, and in fact it is a full and
prime II$_1$ factor \cite{VaesVergnioux} without Cartan subalgebras
\cite{Isono_NoCartan}. On the other hand it is not isomorphic to a free group
factor \cite{BrannanVergnioux}.

\bigskip

The II$_1$ factor $M = \Ll(\FO_N)$ has a natural ``radial'' abelian subalgebra,
$A = \chi_1'' \cap M$ where $\chi_1 = \chi_1^* = \sum_1^N u_{i,i}$ is the sum of
the diagonal generators. It was shown, already in \cite{Banica_LibreUnitaire},
that $\chi_1/2$ is a semicircular variable with respect to $h$, in particular
$\|\chi_1\| = 2$ in $\Ll(\FO_N)$. Since $\epsilon(\chi_1) = N$ in
$C^*_u(\FO_N)$, this implies the non-amenability of $\FO_N$ for $N\geq 3$. The
subalgebra $A\subset M$ is the quantum analogue of the radial subalgebra of
$\Ll(F_N)$, generated by the sum $\chi_1 = \sum_1^N (a_i + a_i^*)$ of the
generators $a_i\in F_N$ and their adjoints, which is known to be a maximal
abelian subalgebra (MASA) since \cite{Pytlik_Radial}.

The position of $A$ in $M$ was already investigated in \cite{FreslonVergnioux},
where it was shown, for $N\geq 3$, to be a strongly mixing MASA.  Note that
$\FO_N$ admits deformations $\FO_Q$, where $Q\in M_N(\CC)$ is an invertible
matrix such that $Q\bar Q =\pm I_N$. When $Q$ is not unitary, the corresponding
von Neumann algebra $M = \Ll(\FO_Q)$ is a type III factor, at least for small
deformations \cite{VaesVergnioux}. One can still consider the abelian subalgebra
$A = \chi_1''\cap M$, but if $Q$ is not unitary it is not maximal abelian
anymore, as shown in \cite{KrajczokWasilewsi}. More precisely, in this case the
inclusion $A\subset M$ is quasi-split in the sense of \cite{DoplicherLongo}.

\bigskip

The aim of the present article is to pursue the study of \cite{FreslonVergnioux} in the non-deformed case.  Our main result is the following theorem, proved at the end of
Section~\ref{sec_support_loc}. Here, and in the rest of the article, we fix a free ultrafilter $\omega$ on $\NN$, but the result also holds for the Fr\'echet filter $\omega = \infty$.

\begin{introtheorem} \label{thm_main} There exists $N_0\in\NN$ such that if $N\geq N_0$ the radial subalgebra $A\subset M = \Ll(\FO_N)$ satisfies the Asymptotic
  Orthogonality Property: for every $y\in A^\bot\cap M$ and for every bounded sequence of elements $z_r \in A^\bot\cap M$ such that
  $\forall a\in A ~ \|[a,z_r]\|_2 \to_\omega 0$, we have $(yz_r\mid z_ry)\to_\omega 0$.
\end{introtheorem}

The Asymptotic Orthogonality Property (AOP) originates from Popa's seminal
article \cite{Popa_MaxInjective} where it was established for
$A = a_1''\subset \Ll(F_N)$, the generator MASA in free group factors, and
proved to imply maximal amenability. It is often stated in a non-symmetric
way, for scalar products of the form $(yz_r \mid z'_ry')$, but the version above
is sufficient for our purposes. We can indeed formulate the following corollary,
which is a quantum analogue of the result of \cite{CFRW_RadialMaxAmenable} about
the radial MASA in free group factors.

\begin{introcorollary}
  There exists $N_0\in\NN$ such that if $N\geq N_0$ the radial subalgebra $A\subset M = \Ll(\FO_N)$ is maximal amenable: for any amenable
  subalgebra $P\subset M$ such that $A\subset P$, we have $A = P$.
\end{introcorollary}

\begin{proof}
  Since $A$ is already known to be a singular MASA by
  \cite[Corollary~5.8]{FreslonVergnioux}, this follows directly from
  \cite[Corollary~2.3]{CFRW_RadialMaxAmenable}, whose proof uses only
  ``symmetric'' scalar products $(yz_r\mid z_ry)$.
\end{proof}

\bigskip

The proof of Theorem~\ref{thm_main} follows a strategy which can also be traced
back to Popa's work on the generator MASA of free
group factors. One can identify the following ingredients:
\begin{enumerate}
\item a good description of the $A,A$-bimodule $H = \ell^2(F_N)$ ;
\item a decreasing sequence of subspaces $V_m \subset H$ such that, for
  $y\in A^\bot\cap M$ fixed and $m$ big enough, $y V_m \bot V_m y$ ;
\item the fact that elements $z\in A^\bot\cap M$ almost commuting to $A$ are almost supported in $V_m$.
\end{enumerate}

\bigskip

In the classical case the arguments for each of the above steps rely on the
combinatorics of reduced words in the free group. In the quantum case the
techniques are completely different and consist in performing analysis in the
Temperley-Lieb category, which is naturally associated with $\FO_N$ as we recall
in the preliminary Section~\ref{sec_prelim}. We give below more details about
the strategy used for each of the three steps, in the classical and quantum
cases, and present the organization of the article.

The more precise goal for (1) is to exhibit an orthonormal basis $\W$ of the
$A,A$-bimodule $A^\bot\cap H$ with good combinatorial properties, which will
allow to carry out computations. In the case of the generator MASA
$a_1''\subset\Ll(F_N)$, this basis is just given by the set of reduced words in
$F_N$ which do not start nor end with $a_1$ nor $a_1^{-1}$. In the case of the
radial MASA in $\Ll(F_N)$, a convenient basis was constructed by R\u adulescu
\cite{Radulescu_Radial} to show that the radial MASA is singular with
Puk\'anszky invariant $\{\infty\}$.

In Section~\ref{sec_bimod} we construct an analogue
$\W = \bigsqcup_{k\geq 1} \W_k$ of the R\u adulescu basis for our free quantum
groups. Surprisingly one has to take into account additional symmetries of $H$
given by the rotation maps $\rho_k$ which already played a (minor) role in
\cite{FreslonVergnioux}. Using this construction and a result from
\cite{FreslonVergnioux}, we can already deduce (Corollary~\ref{crl_pukanszky})
that the Puk\'anszky invariant of the radial MASA in $\Ll(\FO_N)$ is
$\{\infty\}$, a result that was missing in \cite{FreslonVergnioux}.

From $x\in \W$ one can generate a natural $\CC$-linear basis $(x_{i,j})_{i,j\in\NN}$ of the cyclic submodule $AxA$. In R\u adulescu's case, $(x_{i,j})$ is
orthogonal as soon as $x\in \W_k$ with $k\geq 2$, and for $k=1$ it is nevertheless a Riesz basis. In our case, $(x_{i,j})$ is never orthogonal and we have to
show that it is a Riesz basis, uniformly over $x\in \W$. This is accomplished in Section~\ref{sec_gram}, which is the most challenging technically, and we manage
to reach this conclusion only if $N$ is large enough.

The core of the strategy then lies in ingredient (2). In the case of the generator MASA in the free group factor $\Ll(F_N)$, $V_m$ is simply the subspace of $H$ generated by the
reduced words of $F_N$ that begin and end with a ``large'' power $a_1^k$ of the generator, $|k|\geq m$, without being themselves a power of $a_1$. We have then
clearly $V_my\bot yV_m$ if $y\in A^\bot\cap M$ is supported on reduced words of length at most $m$.

In the case of the radial MASA in $\Ll(F_N)$, $V_m$ is defined in terms of the R\u adulescu basis as the subspace generated by the elements $x_{i,j}$,
$x\in \W$, $i$, $j\geq m$. We adopt the same definition in the quantum case, using our analogue of the R\u adulescu basis, and we show in
Section~\ref{sec_orthogonality} that the orthogonality property $V_m y\bot y V_m$ holds in an approximate sense as $m\to\infty$. Note that we use one of the two
main technical tools from \cite{FreslonVergnioux}, in an improved version (Lemma~\ref{lem_trace_proj}).

In the case of the classical generator MASA, the step (3) follows by observing that if $z\in M$ almost commutes to the powers $a_1^k$ of the generator, then its
components supported on a subset $S\subset F_N$ and on the subset $a_1^kSa_1^{-k}$ have approximately the same norm. If $S = S_m$ is the set of words starting
with a power at most $m$ (in absolute value) of $a_1$, for many values of $k$ the subsets $a_1^kS_ma_1^{-k}$ will be pairwise disjoint, so that the norms of the
corresponding components of $z$ will be small. One can then show that $z$ is ``almost'' contained in $V_m$, in a quantitative way.

In our case, we similarly relate various components of $z$ using the commutator $[\chi_1,z]$, see Proposition~\ref{prop_commut_relations} in
Section~\ref{sec_support_loc}. This requires to determine the structure constants for the left and right action of $\chi_1$ on the basis $(x_{i,j})$ for a given
$x\in W$. Then the components of $z$ that we are able to relate in this way are not as simply ``localized'' as in the classical cases, and moreover the
coefficients in these relations are only recursively specified and require a quite delicate analysis to reach the conclusion. For all this it is naturally
necessary to know that the families $(x_{i,j})$ are Riesz bases, uniformly with respect to $x\in W$.

Assembling the results obtained in Sections~\ref{sec_orthogonality} and~\ref{sec_support_loc} it is then easy to prove Theorem~\ref{thm_main}.

\section{Preliminaries}
\label{sec_prelim}

We denote by $\NN$ the set of non-negative integers. Unless otherwise stated, all indices used in the statements belong to $\NN$.

\bigskip

In this article, a discrete quantum group $\GGamma$ is given by a Woronowicz $C^*$-algebra $C^*(\GGamma)$ \cite{Woronowicz_CQG}, i.e.\ a unital $C^*$-algebra
equipped with a unital $*$-homomorphism $\Delta : C^*(\GGamma) \to C^*(\GGamma)\otimes C^*(\GGamma)$ satisfying the following two axioms: i)
$(\Delta\otimes\id)\Delta = (\id\otimes\Delta)\Delta$ (co-associativity); ii) $\Delta(C^*(\GGamma))(1\otimes C^*(\GGamma))$ and
$\Delta(C^*(\GGamma))(C^*(\GGamma)\otimes 1)$ span dense subspaces of $C^*(\GGamma)\otimes C^*(\GGamma)$ (bicancellation). This encompasses classical discrete
groups, as well as duals of classical compacts groups $G$, given by $C^*(\GGamma) = C(G)$.

In this setting, the existence and uniqueness of a bi-invariant state $h\in C^*(\GGamma)^*$, i.e.\ satisfying the relations
$(h\otimes\id)\Delta = 1h = (\id\otimes h)\Delta$, were proved by Woronowicz \cite{Woronowicz_CQG} when $C^*(\GGamma)$ is separable, and by Van Daele
\cite{VanDaele_Haar} in general. We can consider the GNS representation $\lambda$ associated with $h$ and we shall mainly work with the corresponding von
Neumann algebra $M = \Ll(\GGamma) = \lambda(C^*(\GGamma))''$ represented on the Hilbert space $H = \ell^2(\GGamma)$. We still denote $h$ the factorization of
the invariant state to $M$. As the notation suggests, in the classical case $\Ll(\Gamma)$ is the usual group von Neumann algebra with its canonical trace,
whereas for the dual of a compact group $G$ we have $\Ll(\GGamma) = L^\infty(G)$ with the Haar integral.

\bigskip

A corepresentation of $\GGamma$ is an element $u\in M\bar\otimes B(H_u)$ such
that $u_{13}u_{23} = (\Delta\otimes\id)(u)$. We will work exclusively with
unitary and finite-dimensional corepresentations. We denote
$\Hom(u,v)\subset B(H_u,H_v)$ the space of intertwiners from $u$ to $v$,
i.e.\ maps $T$ such that $(1\otimes T)u = v(1\otimes T)$. A corepresentation $u$
is irreducible if $\Hom(u,u) = \CC \id$; two corepresentations $u$, $v$ are
equivalent if $\Hom(u,v)$ contains a bijection. The tensor product of $u$ and
$v$ is $u\otimes v := u_{12}v_{13}$, with $H_{u\otimes v} = H_u \otimes H_v$. We
have defined in this way a tensor $C^*$-category denoted $\Corep(\GGamma)$ with
a fiber functor to Hilbert spaces.

Let $u \in M\otimes B(H_u)$ be a corepresentation of $\GGamma$. For $\zeta$,
$\xi\in H_u$ we can consider the corresponding coefficient
$u_{\zeta,\xi} = (\id\otimes\zeta^*) u (\id\otimes\xi) = (\id\otimes\Tr)
(u(1\otimes\xi\zeta^*))\in M$. More generally for $X\in B(H_u)$ we denote
$u(X) = (\id\otimes\Tr)(u(1\otimes X))$ --- although it would perhaps be more
natural to denote this element $u(\varphi)$ where
$\varphi = \Tr(\,\cdot\, X) \in B(H_u)^*$.

In the present article we will work only with unimodular discrete quantum groups, equivalently, the canonical state $h$ will be a trace. In this case the
Peter-Weyl-Woronowicz orthogonality relations read, for $u$ irreducible:
\begin{equation}
  \label{eq_peter_weyl}
  (u(X) \mid u(Y)) = (\dim u)^{-1} (X\mid Y),
\end{equation}
where we use on the left the scalar product associated with $h$, $(x\mid y) = h(x^*y)$, and on the right the Hilbert-Schmidt scalar product
$(X\mid Y) = \Tr(X^*Y)$. On the other hand we have $(u(X)\mid v(Y)) = 0$ if $u$, $v$ are irreducible and not equivalent.

The product in $M$ can be computed according to the evident formula
$u(X) v(Y) = {(u\otimes v)}$ ${(X\otimes Y)}$. We have moreover $u(TX) = v(XT)$ for
$X\in B(H_u, H_v)$ and $T\in \Hom(v,u)$. As a result, if we choose intertwiners
$T_i \in \Hom(w_i, u\otimes v)$ such that $T_i^*T_i = \id$ and
$\sum_i T_i T_i^* = \id$, we obtain the formula
$u(X)v(Y) = \sum_i w_i(T_i^*(X\otimes Y)T_i)$, which we can use to compute the
product of coefficients of irreducible corepresentations as a linear combination
of coefficients of irreducible corepresentations.

\bigskip

In this article we consider the orthogonal free quantum groups $\GGamma = \FO_N$
defined in the Introduction, and assume $N\geq 3$. Associated to $N$ is the
parameter $q\in\itv]0,1[$ such that $q+q^{-1} = N$, which plays an important
role in the computations. We have $q\to 0$ as $N\to\infty$. Banica
\cite{Banica_LibreOrtho} showed that the $C^*$-tensor category $\Corep(\FO_N)$
is equivalent, as an abstract tensor category, to the Temperley-Lieb category
$TL_\delta$ at parameter $\delta = N$, and that $\FO_N$ is realized via
Tannaka-Krein duality by the fiber functor $F : TL_N \to \mathrm{Hilb}$ which
sends the generating object to $H_1 := \CC^N$, with corepresentation
$u = (u_{i,j})_{i,j}$ given by the canonical generators of $\Ll(\FO_N)$, and the
generating morphism to
$F(\cap) = t := \sum_i e_i\otimes e_i \in H_1\otimes H_1$, where $(e_i)_i$ is
the canonical basis of $\CC^N$. See \cite[Section~2.5]{NeshveyevTuset_Book} for details about this category.

This means that we have a pictorial representation of elements
$A\in \Hom(H_1^{\otimes k},H_1^{\otimes l})$. More precisely, denote $NC_2(k,l)$
the set of non-crossing pair partitions of $k+l$ points. For each partition
$\pi\in NC_2(k,l)$ there is a morphism
$T_\pi \in \Hom(H_1^{\otimes k},H_1^{\otimes l})$ whose matrix coefficients
$(e_{i_1}\otimes\cdots\otimes e_{i_l} \mid T_\pi(e_{j_1}\otimes\cdots\otimes
e_{j_k}))$ are equal to $1$ if ``the indices $i_s$, $j_t$ agree in each block of
$\pi$'', and to $0$ otherwise. Then, for $N\geq 3$ the maps $T_\pi$ with
$\pi\in NC_2(k,l)$ form a linear basis of
$\Hom(H_1^{\otimes k},H_1^{\otimes l})$. Elements $\pi\in NC_2(k,l)$, and the
corresponding morphisms $T_\pi$, are usually depicted inside a rectangle with
$k$ numbered points on the upper edge and $l$ numbered points on the bottom edge
by drawing non-crossing strings joining the two elements in each block of $\pi$.

More generally, the collection of spaces $B(H_1^{\otimes k},H_1^{\otimes l})$ is an (even) planar algebra, meaning that linear maps obtained by composing and
tensoring given maps $X_i \in B(H_1^{\otimes k_i},H_1^{\otimes l_i})$ with maps $T_\pi$ can be represented by means of a rectangular Temperley-Lieb diagram as
above with $p$ internal boxes representing the maps $X_i$. For instance, if $X$, $Y \in B(H_1^{\otimes 2})$ we have, drawing dashed internal and external boxes,
and solid Temperley-Lieb strings:
\begin{displaymath}
  (t^*\otimes t^*\otimes\id)(\id\otimes X\otimes Y)(t\otimes t\otimes\id) =
  ~\vcenter{\hbox{
\setlength{\unitlength}{2072sp}%
\begingroup\makeatletter\ifx\SetFigFont\undefined%
\gdef\SetFigFont#1#2#3#4#5{%
  \reset@font\fontsize{#1}{#2pt}%
  \fontfamily{#3}\fontseries{#4}\fontshape{#5}%
  \selectfont}%
\fi\endgroup%
\begin{picture}(2184,1824)(3499,-4528)
\thinlines
{\color[rgb]{0,0,0}\put(4771,-3256){\oval(540,540)[tr]}
\put(4771,-3256){\oval(540,540)[tl]}
}%
{\color[rgb]{0,0,0}\put(5041,-3346){\line( 0, 1){ 90}}
}%
{\color[rgb]{0,0,0}\put(4501,-3256){\line( 0,-1){ 90}}
}%
{\color[rgb]{0,0,0}\put(4051,-3256){\oval(540,540)[tr]}
\put(4051,-3256){\oval(540,540)[tl]}
}%
{\color[rgb]{0,0,0}\put(4321,-3346){\line( 0, 1){ 90}}
}%
{\color[rgb]{0,0,0}\put(3781,-3256){\line( 0,-1){ 90}}
}%
{\color[rgb]{0,0,0}\put(4051,-3976){\oval(540,540)[bl]}
\put(4051,-3976){\oval(540,540)[br]}
}%
{\color[rgb]{0,0,0}\put(4321,-3886){\line( 0,-1){ 90}}
}%
{\color[rgb]{0,0,0}\put(3781,-3976){\line( 0, 1){ 90}}
}%
{\color[rgb]{0,0,0}\put(4771,-3976){\oval(540,540)[bl]}
\put(4771,-3976){\oval(540,540)[br]}
}%
{\color[rgb]{0,0,0}\put(5221,-3346){\line( 0, 1){630}}
}%
{\color[rgb]{0,0,0}\put(5041,-3886){\line( 0,-1){ 90}}
}%
{\color[rgb]{0,0,0}\put(4501,-3886){\line( 0,-1){ 90}}
}%
{\color[rgb]{0,0,0}\put(5221,-3886){\line( 0,-1){630}}
}%
{\color[rgb]{0,0,0}\put(4141,-3886){\dashbox{86}(540,540){}}
}%
{\color[rgb]{0,0,0}\put(3781,-3346){\line( 0,-1){540}}
}%
{\color[rgb]{0,0,0}\put(3511,-4516){\dashbox{86}(2160,1800){}}
}%
{\color[rgb]{0,0,0}\put(4861,-3886){\dashbox{86}(540,540){}}
}%
\put(4411,-3751){\makebox(0,0)[b]{\smash{{\SetFigFont{11}{13.2}{\rmdefault}{\mddefault}{\updefault}{\color[rgb]{0,0,0}$X$}%
}}}}
\put(5131,-3751){\makebox(0,0)[b]{\smash{{\SetFigFont{11}{13.2}{\rmdefault}{\mddefault}{\updefault}{\color[rgb]{0,0,0}$Y$}%
}}}}
\end{picture}%
    }}~ \in B(H_1).
\end{displaymath}

\bigskip

The irreducible objects of the Temperley-Lieb category, and hence the irreducible corepresentations of $\FO_N$, can be labeled by integers $k\in\NN$ up
to equivalence, in such a way that $u_0 = 1\otimes \id_\CC$ is the trivial corepresentation, $u_1 = u$ is the generating object, and the following fusion
rules are satisfied:
\begin{displaymath}
  u_k\otimes u_l \simeq u_{|k-l|} \oplus u_{|k-l|+2} \oplus \cdots \oplus u_{k+l}.
\end{displaymath}

We denote $H_k$ the Hilbert space associated with $u_k$ and $d_k = \dim H_k$. We write $\Tr_k$, $\tr_k$ the standard and normalized traces on
$B(H_k)$. Note that $d_0 = 1$ and $d_1 = N$. The remaining dimensions can be computed using
the fusion rules and are given by $q$-numbers:
\begin{equation}
  \label{eq_dimensions}
  d_k = [k+1]_q := \frac{q^{k+1} - q^{-(k+1)}}{q-q^{-1}}.
\end{equation}

The irreducible characters are $\chi_k = (\id\otimes\Tr_k)(u_k) \in M$. It follows from the fusion rules and the Peter-Weyl-Woronowicz formula that they form an
orthonormal basis of the $*$-subalgebra $\Aa$ generated by $\chi_1 = \sum u_{i,i}$, which is weakly-$*$ dense in $A = \chi_1''$. 

According to the fusion rules, $u_k$ appears with multiplicity $1$ as a
subobject of $u_1^{\otimes k}$. We agree to take for $H_k$ the corresponding
subspace of $H_1^{\otimes k}$, and we denote $P_k\in B(H_1^{\otimes k})$ the
orthogonal projection onto $H_k$: this is the $k$th Jones-Wenzl projection. We
have $P_k (P_a\otimes P_b) = P_k$, i.e.\ $H_k$ is a subspace of
$H_a\otimes H_b$, as soon as $k=a+b$. We shall use the notation $\id_k$ for the
identity map {\em both on $H_k$ or on $H_1^{\otimes k}$} ; the space it is
acting on should be clear from the context. We will also use the embeddings
$H_k\subset H_a\otimes H_b \subset H_1^{\otimes k}$, when $a+b=k$, to identify
an element $X\in B(H_k)$ with the corresponding elements of $B(H_a\otimes H_b)$
and $B(H_1^{\otimes k})$. This is especially used to take partial traces of $X$
such as $(\Tr_a\otimes\id)(X)$ or $(\Tr_1\otimes\id)(X)$, where $\Tr_k$ always
stands for the trace of $B(H_k)$ as indicated above.

As another consequence of the fusion rules, there is a unique line of fixed vectors in $H_k\otimes H_k$. We already know the generator $t = t_1$ of $\Hom(H_0,H_1\otimes H_1)$. This
map satisfies the {\em conjugate equations} $(\id_1\otimes t^*)(t\otimes \id_1) = \id_1 = (t^*\otimes\id_1)(\id_1\otimes t)$. We slightly abuse notation by defining recursively $t_1^1 = t_1$,
$t_1^k = (\id_1^{\otimes k-1}\otimes t_1\otimes \id_1^{\otimes k-1})t_1^{k-1} \in \Hom(H_0, H_1^{\otimes 2k})$, so that $\Hom(H_0,H_k\otimes H_k)$ is generated
by $t_k := (P_k\otimes P_k)t_1^k = (\id_k\otimes P_k)t_1^k = (P_k\otimes \id_k)t_1^k$. Note that we have then $t_k^*(X\otimes\id_k)t_k = \Tr_k(X)$ for $X\in B(H_k)$, in particular $\|t_k\| = \sqrt{d_k}$.

Using the intertwiner $t$ one can also investigate more precisely the position of $H_n$ in $H_{n-1}\otimes H_1$, and this gives rise for instance to the Wenzl
recursion relation \cite[Prop.~1]{Wenzl_Projections}, see also \cite[Equation~(3.8)]{FrenkelKhovanov_Graphical}  and \cite[Notation~7.7]{VaesVergnioux}:
\begin{equation}\label{eq_wenzl_recursion}
  P_{n} = (P_{n-1}\otimes \id_{1}) + \sum_{l=1}^{n-1} (-1)^{n-l} \frac{d_{l-1}}{d_{n-1}} \left(\id_{1}^{\otimes (l-1)}\otimes t\otimes \id_{1}^{\otimes (n-l-1)}\otimes t^*\right)(P_{n-1}\otimes \id_{1}).
\end{equation}

One can go further and define the {\em basic intertwiner} $V_m^{k,l} = (P_k\otimes P_l)(\id_{k-a}\otimes t_a\otimes\id_{l-a}) P_m$ which spans
$\Hom(H_m,H_k\otimes H_l)$, where $m = k+l-2a$. It is not isometric but its norm can be computed explicitly, see \cite[Lemma~4.8]{Vergnioux_RD}. Following
\cite{FreslonVergnioux}, we denote $\kappa_{m}^{k,l} = \|V_m^{k,l}\|^{-1}$. This yields the following explicit formula to compute the product of coefficients of
irreducible corepresentations:
\begin{equation}
  \label{eq_prod_FO}
  u_k(X)u_l(Y) = \sum_{a=0}^{\min(k,l)} \left(\kappa_m^{k,l}\right)^2 u_m\left(V_m^{k,l *}(X\otimes Y)V_m^{k,l}\right),
\end{equation}
where we still agree to write $m = k+l-2a$. This motivates the following notation (which is indeed connected with the convolution product in $c_c(\FO_N)$ up to
constants).

\begin{notation}
  \label{not_convol}
  For $X\in B(H_k)$, $Y\in B(H_l)$, $m = k+l-2a$ we consider the following element of $B(H_m)$:
  \begin{displaymath}
    X*_mY = V_m^{k,l *}(X\otimes Y)V_m^{k,l} = P_m(\id_{k-a}\otimes t_a^*\otimes\id_{l-a})(X\otimes Y)(\id_{k-a}\otimes t_a\otimes\id_{l-a})P_m.
  \end{displaymath}
\end{notation}

\bigskip

One can perform analysis in the tensor category $\Corep (\FO_N)$. Recall for instance Lemma~\ref{lem_highest_weight} from \cite{VaesVergnioux} below, with some
more precise information about constants.

\begin{lemma}
  \label{lem_dimensions}
  For any $k\in\NN$ we have $q^{-k}\leq d_k\leq q^{-k}/(1-q^2)$.
\end{lemma}

\begin{proof}
  Clear from~\eqref{eq_dimensions}.
\end{proof}

\begin{lemma}
  \label{lem_highest_weight}
  Fix $q_0\in\itv]0,1[$ and assume that $q\in\itv]0,{q_0}]$.  Then there exists a constant $C$ depending only on $q_0$ such that
  $\|(P_{a+b}\otimes\id_c)(\id_a\otimes P_{b+c}) - P_{a+b+c}\|\leq Cq^b$ for all $a$, $b$, $c\in\NN$.
\end{lemma}

\begin{proof}
  This is \cite[Lemma~A.4]{VaesVergnioux}, we only have to check that the
  constant $C$ remains bounded as $q\to 0$. The proof of
  \cite{VaesVergnioux} explicitly gives the following upper bound:
  \begin{displaymath}
    \|(P_{a+b}\otimes\id_c)(\id_a\otimes P_{b+c}) - P_{a+b+c}\|\leq  q^b
    \Big( \prod_0^\infty(1+Dq^k)\Big) \Big(\sum_0^\infty Cq^k\Big),
  \end{displaymath}
  where $C$ and $D$ a priori depend on $q$. Let us show that one can choose $C$
  and $D$ uniformly over $\itv]0,{q_0}]$. Using Lemma~\ref{lem_dimensions} we have
  \begin{displaymath}
    q^{-b-c} \frac{[2]_q[a]_q}{[a+b+c+1]_q} \leq
    q^{-b-c} \frac{q^{-1}q^{-a+1}}{q^{-a-b-c}(1-q^2)^2} \leq \frac 1{(1-q_0^2)^2}.
  \end{displaymath}
  Similarly:
  \begin{align*}
    q^{-b-c} \Big|\frac{[2]_q[a+b]_q}{[a+b+c+1]_q} - \frac{[2]_q[b]_q}{[b+c+1]_q} \Big|
    &= q^{-b-c} \frac{[2]_q[a]_q[c+1]_q}{[a+b+c+1]_q[b+c+1]_q} \\
    &\leq q^{-b-c}\frac{q^{-1}q^{-a+1}q^{-c}}{q^{-a-b-c}q^{-b-c}(1-q^2)^3}
      \leq \frac{q_0^{b}}{(1-q_0^2)^3} \leq \frac 1{(1-q_0^2)^3}.
  \end{align*}
  In \cite{VaesVergnioux}, the only constraint on $C$ is to be an upper bound for these two
  quantities, hence it can indeed be chosen to depend only on $q_0$. On the other
  hand, $D$ should be an upper bound for
  \begin{displaymath}
    q^{-c}\frac{[2]_q[b]_q}{[b+c+1]_q} \leq q^{-c}\frac{q^{-1}q^{-b+1}}{q^{-b-c}(1-q^2)^2}
    \leq \frac 1{(1-q_0^2)^2},
  \end{displaymath}
  hence it can also be chosen to depend only on $q_0$.
\end{proof}

We also have estimates on the constants $\kappa$, already proved in \cite{Vergnioux_RD}. The formulae for $\kappa_m^{k,l}$ show that, again, the constant $C$ is
uniform for $q$ varying in an interval $\itv]0,{q_0}]$ with $q_0<1$, but we will not need this fact.

\begin{lemma}
  \label{lem_kappa}
  There exists a constant $C$, depending only on $q$, such that we have $1\leq \sqrt{d_a} \kappa_m^{k,l} \leq C$ for all $k$, $l$ and $m = k+l-2a$.
\end{lemma}

\begin{proof}
  See the proof of \cite[Lemma~4.8]{Vergnioux_RD}, \cite[p.~1583]{BrannaVergniouxYoun}, \cite[Equation (6) and Proposition 3.1]{BrannanCollins_Entangled}.
\end{proof}

\noindent The following estimate appeared also in connection with Property RD
\cite{Vergnioux_RD}. Recall that $\|\cdot\|_2$ denotes the Hilbert-Schmidt norm on matrix spaces.

\begin{lemma}
  \label{lem_RD}
  Consider integers such that $m = k+l-2a$. Then for any $X\in B(H_1^{\otimes k})$, $Y\in B(H_1^{\otimes l})$ we have
  $\|(\id_{k-a}\otimes t_a^*\otimes \id_{l-a})(X\otimes Y)(\id_{k-a}\otimes t_a\otimes \id_{l-a})\|_2 \leq \|X\|_2 \|Y\|_2$ and
  ${\|(\id\otimes\Tr_a)(X)\|_2} \leq \sqrt{d_a} \|X\|_2$.
\end{lemma}

\begin{proof}
  The proof of \cite[Theorem~4.9]{Vergnioux_RD} applies, although it was there used only for $X\in B(H_k)$, $Y\in B(H_l)$. Let us repeat it. Consider an
  orthonormal basis $(f_i)_i$ of $H_a$, then the basis $(\bar f_i)_i$ defined by putting $t_a = \sum_i f_i\otimes \bar f_i$ is orthonormal as well: indeed its
  Gram matrix is Woronowicz modular matrix $F_a$, which is equal to the identity in our unimodular case. Put $E_I = f_i f_j^* $ and
  $\bar E_I = \bar f_i\bar f_j^*\in B(H_a)$ for $I = (i,j)$, these are orthonormal bases of $B(H_a)$ for the Hilbert-Schmidt structure and we have
  $t_a^*(E_I\otimes\bar E_J)t_a = \delta_{I,J}$. Decompose $(\id_{k-a}\otimes P_a)X(\id_{k-a}\otimes P_a) = \sum X_I\otimes E_I$ with
  $X_I\in B(H_1^{\otimes k-a})$ and similarly $(P_a\otimes\id_{l-a})Y(P_a\otimes\id_{l-a}) = \sum \bar E_J\otimes Y_J$. We have then
  $\sum \|X_I\|_2^2 = \|(\id_{k-a}\otimes P_a)X(\id_{k-a}\otimes P_a)\|_2^2\leq \|X\|_2^2$ and similarly $\sum \|Y_J\|_2^2 \leq \|Y\|_2^2$. Finally we have by
  the triangle inequality and Cauchy-Schwarz :
  \begin{align*}
    \|(\id\otimes t_a^*\otimes \id)(X\otimes Y)(\id\otimes t_a\otimes \id)\|_2^2
    &= \|{\ts\sum_{I,J}} t_a^*(E_I \otimes \bar E_J)t_a \times (X_I\otimes Y_J)\|_2^2 \\
    &\leq \big({\ts\sum_I} \|X_I\|_2\|Y_I\|_2\big)^2 \\
    &\leq ({\ts\sum_I} \|X_I\|_2^2) ({\ts\sum_I} \|Y_I\|_2^2)
      \leq \|X\|_2^2 \|Y\|_2^2.
  \end{align*}

  The second inequality of this lemma follows by taking $l=a$ and $Y = \id_a$, but can also be proved more directly by noticing that in the canonical isometric
  isomorphism $B(K\otimes L) \simeq K\otimes L\otimes \bar L\otimes \bar K$, the partial trace $\id\otimes \Tr_L$ corresponds to the map
  $\id\otimes t_L^*\otimes \id$, where $t_L :  \CC\to L\otimes\bar L$ is the canonical duality vector whose norm is $\sqrt{\dim L}$.
\end{proof}

We will use again one of the two main estimates from \cite{FreslonVergnioux} about $\Corep(\FO_N)$. For $a$, $b$, $c\in\NN$ consider
$\Pi_{a,b,c} = (\id_a\otimes\tr_b\otimes\id_c)(P_{a+b+c}) \in B(H_a\otimes H_c)$ --- this time the analysis deals with $\Corep(\FO_N)$ together with its
canonical fiber functor. Proposition~3.2 of \cite{FreslonVergnioux} shows that $\Pi_{a,b,c}$ is almost scalar as $b\to\infty$. We give below an improvement of
the corresponding constants.

\begin{lemma}
  \label{lem_trace_proj}
  For every $q_0 \in\itv]0,1[$ there exist constants $C > 0$, $\alpha\in\itv]0,1[$ such that, for all $a$, $b$, $c\in\NN$ and $q\in\itv]0,{q_0}]$ we have
  $\|\Pi_{a,b,c} - \lambda(\id_a\otimes\id_c)\|\leq C q^{\lfloor\alpha b\rfloor}$ for some scalar $\lambda \in\CC$.
\end{lemma}

\begin{proof}
  Let us note first that in the case $c = 0$ the map $d_b \Pi_{a,b,c} = (\id_a\otimes t_b^*)(P_{a+b}\otimes\id_b)(\id_a\otimes t_b)$ is an intertwiner of the
  irreducible space $H_a$, hence it is a multiple of the identity. On the other hand, for $c\geq 1$ Proposition~3.2 of \cite{FreslonVergnioux} uses the scalar
  $\lambda=\lambda_{a,c}$ explicitly given by $\lambda_{a,c} = q^{-a-c}/d_a d_c$. Consider $\Pi'_{a,b,c} = d_b\Pi_{a,b,c} - d_b\lambda_{a,c}(\id_a\otimes\id_c)$. A
  direct computation shows that
  \begin{align*}
    \Tr(\Pi'_{a,b,c}) = d_{a+b+c} - q^{-a-c}d_b &= q^{b+2}\frac{q^{-a-c} - q^{a+c}}{1-q^2} \\
    &\leq \frac 1{1-q^2} \sqrt{d_{a+c}d_ad_c} = \frac{\sqrt{d_{a+c}}}{1-q^2} (\Tr \id_a\otimes \id_c)^{1/2}.
  \end{align*}
  Now, \cite{FreslonVergnioux} shows the existence of constants $D_{a,c}$ such that $|\Tr(\Pi'_{a,b,c}f)|\leq D_{a,c} (\Tr f^*f)^{1/2}$ for
  $f\in B(H_a)\otimes B(H_c)$ with $\Tr(f) = 0$. This implies
  \begin{displaymath}
    |\Tr(\Pi'_{a,b,c}f)| \leq (d_{a+c}/(1-q^2)^2+D_{a,c}^2)^{1/2} (\Tr f^*f)^{1/2}
  \end{displaymath}
  for any $f \in B(H_a)\otimes B(H_c)$, hence $\|\Pi_{a,b,c} - \lambda_{a,c}\id\|_2 \leq (d_{a+c}/(1-q^2)^2 + D_{a,c}^2)^{1/2}d_b^{-1}$. Here we use the
  Hilbert-Schmidt norm in $B(H_a\otimes H_c)$, which is bigger than the operator norm.


  Moreover, it is explicitly stated in the proof of \cite[Prop.~3.2]{FreslonVergnioux} that one can take the constants $D_{a,c}$ defined by induction over $c$
  as follows: $D_{a,0} = 0$ and, for $c\geq 1$:
  \begin{displaymath}
    D_{a,c} = K_c \max(d_1^{1/2} D_{a,c-1} + d_1^{3/2} d_{a-1}, d_{a+c}^{1/2}),
  \end{displaymath}
  where $1\leq K_c = 1/(1-q^c) \leq K := 1/(1-q)$. In particular $d_{a+c}\leq D_{a,c}^2$ if $c\geq 1$. Putting $C_1 = \sqrt{2}/(1-q_0^2)$ we have thus, for all
  $a$, $b$, $c\in\NN$, the existence of $\lambda\in\CC$ such that $\|\Pi_{a,b,c}-\lambda\,\id\| \leq C_1 D_{a,c}q^b$.

  One can then show by induction that the constants $D_{a,c}$ satisfy the estimate $D_{a,c} \leq (2NK)^{a+c}$, where $N=d_1=q+q^{-1}$. Indeed
  $K_c d_{a+c}^{1/2} \leq K N^{(a+c)/2} \leq (2KN)^{a+c}$, and for $c\in\NN^*$ we have by induction
  \begin{align*}
    K_c(d_1^{1/2} D_{a,c-1} + d_1^{3/2} d_{a-1}) \leq KN^{1/2}(2NK)^{a+c-1} + KN^{3/2}N^{a-1} \leq (2NK)^{a+c}.
  \end{align*}
  Of course this estimate is quite bad, but one can improve it using \cite[Lemma~A.4]{VaesVergnioux}.

  More precisely, let $\alpha > 0$ be such that $(2KN)^{2\alpha} q = q^\alpha$. Take $a$, $c \geq \alpha b$. Denote $C_0$ the constant given by
  Lemma~\ref{lem_highest_weight}. Then we have
  \begin{displaymath}
    P_{a+b+c} \simeq (P_a\otimes\id_b\otimes P_c)(\id_{a-\lfloor\alpha b\rfloor}\otimes
    P_{b+2\lfloor\alpha b\rfloor}\otimes \id_{c-\lfloor\alpha b\rfloor})
  \end{displaymath}
  up to $2C_0 q^{\lfloor\alpha b\rfloor}$ in operator norm. Applying $\id\otimes\tr_b\otimes\id$, which is contracting, to this estimate we obtain
  \begin{displaymath}
    \Pi_{a,b,c} \simeq (P_a\otimes P_c) (\id_{a-\lfloor\alpha b\rfloor}\otimes
    \Pi_{\lfloor\alpha b\rfloor,b,\lfloor\alpha b\rfloor}\otimes \id_{c-\lfloor\alpha b\rfloor})
    \simeq \lambda (\id_a\otimes \id_c)
  \end{displaymath}
  up to
  $2C_0 q^{\lfloor\alpha b\rfloor} + C_1D_{\lfloor\alpha b\rfloor,\lfloor\alpha b\rfloor}q^b \leq 2C_0 q^{\lfloor\alpha b\rfloor} + C_1(2NK)^{2\lfloor\alpha
    b\rfloor}q^b$ in operator norm, for some $\lambda\in\CC$. Since $q\leq 1\leq 2NK$ we have moreover
  $(2NK)^{2\lfloor\alpha b\rfloor}q^b \leq (2NK)^{2\alpha b}q^b = q^{\alpha b} \leq q^{\lfloor\alpha b\rfloor}$ by definition of $\alpha$. This yields
  $\|\Pi_{a,b,c} - \lambda (\id_a\otimes\id_c)\| \leq (2C_0+C_1) q^{\lfloor \alpha b\rfloor}$. This estimate is also valid if $a$, $c < \alpha b$ because in this
  case $D_{a,c}q^b \leq (2KN)^{2\alpha b}q^b = q^{\alpha b}$. It holds also in the remaining cases by using Lemma~\ref{lem_highest_weight} only on one side.

  Finally we have shown the existence of $D_0>0$, depending only on $q_0$, and $\alpha > 0$ such that for all $a$, $b$, $c$ there exists a constant $\lambda$
  such that $\|\Pi_{a,b,c}-\lambda (\id_a\otimes \id_c)\| \leq D_0 q^{\lfloor\alpha b\rfloor}$. One should be careful that $\alpha$ depends on $q$. In fact it
  can be computed explicitly from the defining relation $(2KN)^{2\alpha} q = q^\alpha$, with $K = 1/(1-q)$ and $N = q+q^{-1}$: one gets
  \begin{displaymath}
    \alpha = \frac 13 \left[1-\frac{2\ln 2}{3\ln q}-\frac 2{3\ln q}\ln\left(\frac{1+q^2}{1-q}\right)\right]^{-1}.
  \end{displaymath}
  From this it follows that $\alpha$ is decreasing from $1/3$ to $0$ as $q$ varies from $0$ to $1$, and the result follows.
\end{proof}

\begin{remark}
  For instance one can take $\alpha = 1/4$ for $q_0 \approx 0.15  $ (or $N_0=7$). We also have $q^\alpha \sim L q^{1/3}$ as $q\to 0$, where $L = \exp(2\ln(2)/9)$.
\end{remark}

\bigskip

We will need in the next section one last tool about the representation category of $\FO_N$. The Wenzl recursion relation~\eqref{eq_wenzl_recursion}, applied
twice, yields the following bilateral version.

\begin{lemma}\label{lem_bilateral_recursion}
  For $n\geq 4$ we have the bilateral Wenzl recursion relation:
  \begin{align*}
    P_n &= (\id_{1}\otimes P_{n-2}\otimes \id_1) + \\
        & \qquad -\frac{d_{n-2}}{d_{n-1}}(\id_{1}\otimes P_{n-2}\otimes\id_1)(t t^*\otimes \id_{1}^{\otimes (n-2)})(\id_{1}\otimes P_{n-2}\otimes\id_1) \\
        & \qquad -\frac{d_{n-2}}{d_{n-1}}(\id_{1}\otimes P_{n-2}\otimes \id_1) (\id_1^{\otimes (n-2)}\otimes t t^*)(\id_{1}\otimes P_{n-2}\otimes \id_1)    \\
        & \qquad + \frac{(-1)^{n-1}}{d_{n-1}} (\id_{1}\otimes P_{n-2}\otimes \id_1)(t\otimes\id_1^{\otimes( n-2)}\otimes t^*)(\id_{1}\otimes P_{n-2}\otimes \id_1) \\
        & \qquad + \frac{(-1)^{n-1}}{d_{n-1}} (\id_{1}\otimes P_{n-2}\otimes \id_1)(t^*\otimes\id_1^{\otimes (n-2)}\otimes t)(\id_{1}\otimes P_{n-2}\otimes \id_1) \\
        & \qquad + \frac{d_1+d_{n-3}d_{n-2}}{d_{n-1}d_{n-2}}(\id_{1}\otimes P_{n-2}\otimes \id_1)(t t^*\otimes\id_1^{\otimes (n-4)}\otimes t t^*)(\id_{1}\otimes P_{n-2}\otimes \id_1).
  \end{align*}
  For $n=3$ the formula still holds, without the last term.
\end{lemma}

\begin{proof}
  We assume for this proof that $n\geq 4$. A similar calculation gives the result for $n=3$.
  
  We first multiply the relation~\eqref{eq_wenzl_recursion} on the left by $(\id_1\otimes P_{n-2}\otimes\id_1)$. All terms except $l=1$ and $l=n-1$ vanish
  because they involve $P_{n-2}(\id_i\otimes t\otimes\id_j)$, and we are left with :
  \begin{align*}
    P_{n} &= (P_{n-1}\otimes \id_{1}) + \frac{(-1)^{n-1}}{d_{n-1}} (\id_1\otimes P_{n-2}\otimes\id_1)(t\otimes \id_{1}^{\otimes (n-2)}\otimes t^*)(P_{n-1}\otimes \id_{1}) \\
          &\qquad - \frac{d_{n-2}}{d_{n-1}} (\id_1\otimes P_{n-2}\otimes\id_1)(\id_{1}^{\otimes( n-2)}\otimes t t^*)(P_{n-1}\otimes \id_{1}).
  \end{align*}
  Let us denote $A$, $B$, $C$ the three terms on the right-hand side above, without the numeric coefficients.
  We apply the left version Wenzl's recursion to the
  projections $P_{n-1}$:
  \begin{displaymath}
    P_{n-1} = (\id_{1}\otimes P_{n-2}) + \sum_{k=1}^{n-2} (-1)^{n-1-k} \frac{d_{k-1}}{d_{n-2}} \left(t^*\otimes\id_{1}^{\otimes (n-k-2)} \otimes t\otimes \id_{1}^{\otimes (k-1)}\right)(\id_{1}\otimes P_{n-2}).
  \end{displaymath}
  Multiplying on the left by $(\id_1\otimes P_{n-2})$ this yields
  \begin{align*}
    A = (\id_{1}\otimes P_{n-2}\otimes \id_1) - \frac{d_{n-3}}{d_{n-2}} (\id_{1}\otimes P_{n-2}\otimes\id_1)(tt^*\otimes \id_{1}^{\otimes (n-2)})(\id_{1}\otimes P_{n-2}\otimes\id_1).
  \end{align*}
  We proceed similarly with $B$: only the terms $k=1$, $k=2$ have a non-vanishing contribution and we obtain, applying the conjugate equation:
  \begin{align*}
    B &= (\id_{1}\otimes P_{n-2}\otimes \id_1)(t\otimes\id_1^{\otimes (n-2)}\otimes t^*)(\id_{1}\otimes P_{n-2}\otimes \id_1) \\
      & \qquad + \frac{(-1)^n}{d_{n-2}} (\id_{1}\otimes P_{n-2}\otimes \id_1)(t t^*\otimes\id_1^{\otimes (n-2)})(\id_{1}\otimes P_{n-2}\otimes \id_1) \\
      & \qquad + \frac{(-1)^{n-1}d_1}{d_{n-2}} (\id_{1}\otimes P_{n-2}\otimes \id_1)(t t^*\otimes\id_1^{\otimes (n-4)}\otimes tt^*)(\id_{1}\otimes P_{n-2}\otimes \id_1).
  \end{align*}
  Finally for $C$ only the terms $k=1$, $k=n-2$ survive, yielding:
  \begin{align*}
    C &= (\id_{1}\otimes P_{n-2}\otimes \id_1) (\id_1^{\otimes( n-2)}\otimes tt^*)(\id_{1}\otimes P_{n-2}\otimes \id_1) \\
      & \qquad + \frac{(-1)^n}{d_{n-2}}(\id_{1}\otimes P_{n-2}\otimes \id_1)(t^*\otimes\id_1^{\otimes (n-2)}\otimes t)(\id_{1}\otimes P_{n-2}\otimes \id_1) \\
      & \qquad - \frac{d_{n-3}}{d_{n-2}}(\id_{1}\otimes P_{n-2}\otimes \id_1)(t t^*\otimes\id_1^{\otimes (n-4)}\otimes t t^*)(\id_{1}\otimes P_{n-2}\otimes \id_1).
  \end{align*}
  The result follows by gathering $A$, $B$ and $C$ with their coefficients and using the relation $d_{n-3}d_{n-1}+1 = d_{n-2}^2$.
\end{proof}

\section{Decomposition of the Bimodule}
\label{sec_bimod}

In this section we consider the GNS space $H = \ell^2(\GGamma)$ of $M = \Ll(\GGamma)$ with respect to the Haar trace $h$. We identify $M$ with a dense subspace
of $H$. We shall study $H$ as an $A{,}A$-bimodule for $A = \chi_1''\cap M$. We will more specifically consider the orthogonal $H^\circ\subset H$ of the trivial
bimodule $A \subset H$, and we shall decompose it into simpler, pairwise orthogonal submodules generated by natural elements, see
Proposition~\ref{prp_decomp_bimod}. Moreover we will exhibit for each of these cyclic submodules $\Aa x\Aa$ a linear basis $(x_{i,j})_{i,j}$, see
Proposition~\ref{prp_simple_bimod} and Corollary~\ref{crl_independent}. Recall that $\Aa$ is the unital canonical dense sub-$*$-algebra of $A$ generated by $\chi_1$.

We denote $p_k\in B(H)$ the orthogonal projection onto the subspace $p_kH = u_k(B(H_k))$ spanned by coefficients of $u_k$. Note that $p_k$ belongs in fact to
the dual algebra $\ell^\infty(\GGamma)$, and that the projection $P_k \in B(H_1^{\otimes k})$ introduced in the preceding section is the image of $p_k$ under
the natural representation of $\ell^\infty(\GGamma)$ on the corepresentation space $H_1^{\otimes k}$.

The space $H^\circ$ is spanned by its subspaces $p_k H^\circ$ and we have $p_k H^\circ = H^\circ\cap p_k H = u_k(B(H_k)^\circ)$ where
$B(H_k)^\circ = \{X\in B(H_k) \mid \Tr(X) = 0\}$.  In the case of the classical generator MASA $a_1'' \subset \Ll(F_N)$, the subspace analogous to
$p_k H^\circ$ is spanned by reduced words of length $k$, different from $a_1^{\pm k}$. We introduce below a subspace $H^\ccirc\subset H^\circ$ which is the
quantum replacement for the set of words $g\in F_N$ that do not start nor end with $a_1$.

\begin{notation}
  \label{def_ccirc}
  For $n\geq 1$ we denote
  \begin{displaymath}
    B(H_n)^\ccirc = \{X\in B(H_n) \mid (\Tr_1\otimes\id)(X) = 0 = (\id\otimes\Tr_1)(X)\}.
  \end{displaymath}
  We denote $H^\ccirc$ the closed linear span of the subspaces $u_n(B(H_n)^\ccirc)$ in $H^\circ$.
\end{notation}

\begin{remark}
  \label{rk_ccirc}
  It is well-known that $H_n \subset H_1^{\otimes n}$ is the subspace of vectors $\zeta \in H_1^{\otimes n}$ such that
  $(\id_i\otimes t^*\otimes\id_{n-i-2})(\zeta) = 0$, for all $i = 0, \ldots, n-2$. This follows by induction from the fact that $H_n$ is the kernel of
  $t^*\otimes\id_{n-2} : H_1\otimes H_{n-1} \to H_{n-2}$, according to the fusion rules. As a consequence, an element $X\in B(H_1^{\otimes n})$ arises from an
  element of $B(H_n)$ {\bf iff} we have $(\id_i\otimes t^*\otimes\id_{n-i-2})X = 0$ and $X(\id_i\otimes t\otimes\id_{n-i-2}) = 0$ for all $i$. Graphically this
  means we have $X\in B(H_n)$ {\bf iff} we obtain $0$ by applying to $X$ any planar tangle which connects two consecutive points on the lower or upper edge of
  the internal box corresponding to $X$:
  \begin{displaymath}
\setlength{\unitlength}{4144sp}%
\begingroup\makeatletter\ifx\SetFigFont\undefined%
\gdef\SetFigFont#1#2#3#4#5{%
  \reset@font\fontsize{#1}{#2pt}%
  \fontfamily{#3}\fontseries{#4}\fontshape{#5}%
  \selectfont}%
\fi\endgroup%
\begin{picture}(2094,429)(2464,-73)
{\color[rgb]{0,0,0}\thinlines
\put(2701,-16){\circle*{10}}
}%
{\color[rgb]{0,0,0}\put(2656,-16){\circle*{10}}
}%
{\color[rgb]{0,0,0}\put(2611,-16){\circle*{10}}
}%
{\color[rgb]{0,0,0}\put(3061,-16){\circle*{10}}
}%
{\color[rgb]{0,0,0}\put(3016,-16){\circle*{10}}
}%
{\color[rgb]{0,0,0}\put(2971,-16){\circle*{10}}
}%
{\color[rgb]{0,0,0}\put(2881,299){\circle*{10}}
}%
{\color[rgb]{0,0,0}\put(2836,299){\circle*{10}}
}%
{\color[rgb]{0,0,0}\put(2791,299){\circle*{10}}
}%
{\color[rgb]{0,0,0}\put(2836, 29){\oval(180,180)[bl]}
\put(2836, 29){\oval(180,180)[br]}
}%
{\color[rgb]{0,0,0}\put(2566,-61){\line( 0, 1){ 90}}
}%
{\color[rgb]{0,0,0}\put(3106,-61){\line( 0, 1){ 90}}
}%
{\color[rgb]{0,0,0}\put(2566,254){\line( 0, 1){ 90}}
}%
{\color[rgb]{0,0,0}\put(2656,254){\line( 0, 1){ 90}}
}%
{\color[rgb]{0,0,0}\put(3016,254){\line( 0, 1){ 90}}
}%
{\color[rgb]{0,0,0}\put(3106,254){\line( 0, 1){ 90}}
}%
{\color[rgb]{0,0,0}\put(2476, 29){\framebox(720,225){}}
}%
{\color[rgb]{0,0,0}\put(4231,-16){\circle*{10}}
}%
{\color[rgb]{0,0,0}\put(4186,-16){\circle*{10}}
}%
{\color[rgb]{0,0,0}\put(4141,-16){\circle*{10}}
}%
{\color[rgb]{0,0,0}\put(4051,299){\circle*{10}}
}%
{\color[rgb]{0,0,0}\put(4006,299){\circle*{10}}
}%
{\color[rgb]{0,0,0}\put(3961,299){\circle*{10}}
}%
{\color[rgb]{0,0,0}\put(4411,299){\circle*{10}}
}%
{\color[rgb]{0,0,0}\put(4366,299){\circle*{10}}
}%
{\color[rgb]{0,0,0}\put(4321,299){\circle*{10}}
}%
{\color[rgb]{0,0,0}\put(4186,254){\oval(180,180)[tr]}
\put(4186,254){\oval(180,180)[tl]}
}%
{\color[rgb]{0,0,0}\put(3916,-61){\line( 0, 1){ 90}}
}%
{\color[rgb]{0,0,0}\put(4456,-61){\line( 0, 1){ 90}}
}%
{\color[rgb]{0,0,0}\put(4366,-61){\line( 0, 1){ 90}}
}%
{\color[rgb]{0,0,0}\put(4006,-61){\line( 0, 1){ 90}}
}%
{\color[rgb]{0,0,0}\put(3826, 29){\framebox(720,225){}}
}%
{\color[rgb]{0,0,0}\put(3916,254){\line( 0, 1){ 90}}
}%
{\color[rgb]{0,0,0}\put(4456,254){\line( 0, 1){ 90}}
}%
\put(2836, 74){\makebox(0,0)[b]{\smash{{\SetFigFont{11}{13.2}{\rmdefault}{\mddefault}{\updefault}{\color[rgb]{0,0,0}$X$}%
}}}}
\put(4186, 74){\makebox(0,0)[b]{\smash{{\SetFigFont{11}{13.2}{\rmdefault}{\mddefault}{\updefault}{\color[rgb]{0,0,0}$X$}%
}}}}
\put(3511, 74){\makebox(0,0)[b]{\smash{{\SetFigFont{12}{14.4}{\rmdefault}{\mddefault}{\updefault}{\color[rgb]{0,0,0}$= 0 =$}%
}}}}
\end{picture}%
  \end{displaymath}
  Since $(\Tr_1\otimes\id)(X) \in B(H_1^{\otimes n-1})$ (resp $(\id\otimes\Tr_1)(X)$) is obtained from $X$ by applying the planar tangle connecting the upper
  left and lower left (resp. upper right and lower right) points of the internal box, we conclude that $X\in B(H_1^{\otimes n})$ belongs to $B(H_n)^\ccirc$ {\bf
    iff} we obtain $0$ by applying to $X$ any planar tangle which connects any two consecutive points of the internal box corresponding to $X$. Diagrammatically
  this is represented by the additional constraints:
  \begin{displaymath}
\setlength{\unitlength}{4144sp}%
\begingroup\makeatletter\ifx\SetFigFont\undefined%
\gdef\SetFigFont#1#2#3#4#5{%
  \reset@font\fontsize{#1}{#2pt}%
  \fontfamily{#3}\fontseries{#4}\fontshape{#5}%
  \selectfont}%
\fi\endgroup%
\begin{picture}(2274,429)(2374,-73)
{\color[rgb]{0,0,0}\thinlines
\put(2881,299){\circle*{10}}
}%
{\color[rgb]{0,0,0}\put(2836,299){\circle*{10}}
}%
{\color[rgb]{0,0,0}\put(2791,299){\circle*{10}}
}%
{\color[rgb]{0,0,0}\put(2881,-16){\circle*{10}}
}%
{\color[rgb]{0,0,0}\put(2836,-16){\circle*{10}}
}%
{\color[rgb]{0,0,0}\put(2791,-16){\circle*{10}}
}%
{\color[rgb]{0,0,0}\put(2476,254){\oval(180,180)[tr]}
\put(2476,254){\oval(180,180)[tl]}
}%
{\color[rgb]{0,0,0}\put(2476, 29){\oval(180,180)[bl]}
\put(2476, 29){\oval(180,180)[br]}
}%
{\color[rgb]{0,0,0}\put(3106,-61){\line( 0, 1){ 90}}
}%
{\color[rgb]{0,0,0}\put(2656,254){\line( 0, 1){ 90}}
}%
{\color[rgb]{0,0,0}\put(3016,254){\line( 0, 1){ 90}}
}%
{\color[rgb]{0,0,0}\put(3106,254){\line( 0, 1){ 90}}
}%
{\color[rgb]{0,0,0}\put(2476, 29){\framebox(720,225){}}
}%
{\color[rgb]{0,0,0}\put(2656,-61){\line( 0, 1){ 90}}
}%
{\color[rgb]{0,0,0}\put(3016,-61){\line( 0, 1){ 90}}
}%
{\color[rgb]{0,0,0}\put(2386, 29){\line( 0, 1){225}}
}%
{\color[rgb]{0,0,0}\put(4231,-16){\circle*{10}}
}%
{\color[rgb]{0,0,0}\put(4186,-16){\circle*{10}}
}%
{\color[rgb]{0,0,0}\put(4141,-16){\circle*{10}}
}%
{\color[rgb]{0,0,0}\put(4231,299){\circle*{10}}
}%
{\color[rgb]{0,0,0}\put(4186,299){\circle*{10}}
}%
{\color[rgb]{0,0,0}\put(4141,299){\circle*{10}}
}%
{\color[rgb]{0,0,0}\put(4546,254){\oval(180,180)[tr]}
\put(4546,254){\oval(180,180)[tl]}
}%
{\color[rgb]{0,0,0}\put(4546, 29){\oval(180,180)[bl]}
\put(4546, 29){\oval(180,180)[br]}
}%
{\color[rgb]{0,0,0}\put(3916,-61){\line( 0, 1){ 90}}
}%
{\color[rgb]{0,0,0}\put(4366,-61){\line( 0, 1){ 90}}
}%
{\color[rgb]{0,0,0}\put(4006,-61){\line( 0, 1){ 90}}
}%
{\color[rgb]{0,0,0}\put(3826, 29){\framebox(720,225){}}
}%
{\color[rgb]{0,0,0}\put(3916,254){\line( 0, 1){ 90}}
}%
{\color[rgb]{0,0,0}\put(4006,254){\line( 0, 1){ 90}}
}%
{\color[rgb]{0,0,0}\put(4366,254){\line( 0, 1){ 90}}
}%
{\color[rgb]{0,0,0}\put(4636, 29){\line( 0, 1){225}}
}%
\put(2836, 74){\makebox(0,0)[b]{\smash{{\SetFigFont{11}{13.2}{\rmdefault}{\mddefault}{\updefault}{\color[rgb]{0,0,0}$X$}%
}}}}
\put(4186, 74){\makebox(0,0)[b]{\smash{{\SetFigFont{11}{13.2}{\rmdefault}{\mddefault}{\updefault}{\color[rgb]{0,0,0}$X$}%
}}}}
\put(3511, 74){\makebox(0,0)[b]{\smash{{\SetFigFont{12}{14.4}{\rmdefault}{\mddefault}{\updefault}{\color[rgb]{0,0,0}$= 0 =$}%
}}}}
\end{picture}%
  \end{displaymath}
\end{remark}

\smallskip

Now we compute the dimension of $B(H_n)^\ccirc$, see Proposition~\ref{prp_dim_ccirc}. This will be useful to prove that the families $(x_{i,j})_{i,j}$ are
linearly independent at Corollary~\ref{crl_independent}. The latter also follows from the stronger results of Section~\ref{sec_gram}, but there we will have to
assume that $N$ is large enough and the proofs are much more involved. Note however that the proof below is not optimal either, in the sense that the underlying
technical result established at Lemma~\ref{lem_dim_inequality} does not hold if $q+q^{-1}\in\itv]2,{2.41}[$, which can occur for the non unimodular groups
$\FO_Q$. We believe that Lemma~\ref{lem_contract_ends} and Proposition~\ref{prp_dim_ccirc} hold true for any group $\FO_Q$ with $q+q^{-1}>2$, i.e.\ excluding the
duals of $SU(2)$ and $SU_{-1}(2)$.

In the statement below we use the leg numbering notation: $t_{1,n}^* = \sum_i e_i^*\otimes\id_1\otimes\cdots\otimes\id_1\otimes e_i^*$, for $n\geq 2$. This
application maps $H_n$ to $H_{n-2}$, as can be seen when $n\geq 4$ by checking the condition $(\id_i\otimes t^*\otimes\id_{n-i-4})t_{1,n}^*(\zeta) = 0$, for any
$\zeta\in H_n$.

\begin{lemma}\label{lem_contract_ends}
  Assume $N\geq 3$. For $n\geq 3$ the map $t_{1,n}^* : H_n \to H_{n-2}$ is surjective.
\end{lemma}

\begin{proof}
  We apply $t_{1,n}^* \cdot\, t_{1,n}$ to the bilateral Wenzl recursion formula from Lemma~\ref{lem_bilateral_recursion}. Using the conjugate equations we have
  in $B(H_1^{\otimes (n-2)})$:
  \begin{align*}
    & t_{1,n}^*(t t^*\otimes \id_{1}^{\otimes (n-2)}) t_{1,n} = \id_1^{\otimes (n-2)} = t_{1,n}^*(\id_{1}^{\otimes (n-2)}\otimes t t^*) t_{1,n},  \\
    & t_{1,n}^*(t\otimes\id_1^{\otimes (n-2)}\otimes t^*) t_{1,n} = C_{n-2}^{-2}, \\
    & t_{1,n}^*(t^*\otimes\id_1^{\otimes (n-2)}\otimes t) t_{1,n} = C_{n-2}^2, \\
    & t_{1,n}^*(t t^*\otimes\id_1^{\otimes (n-4)}\otimes t t^*)  t_{1,n} = t_{1,n-2}t_{1,n-2}^* \qquad\text{($n\geq 4$)},
  \end{align*}
  where $C_{n-2} : \xi\otimes\zeta \mapsto \zeta\otimes\xi$ for $\xi\in H_1$, $\zeta\in H_1^{\otimes (n-3)}$. Thus for $n\geq 4$ we obtain
  \begin{align*}
    t_{1,n}^* P_n t_{1,n} &= \ts \left(d_1 - \frac{2d_{n-2}}{d_{n-1}}\right) P_{n-2} + \\
                          & \qquad + \frac{(-1)^{n-1}}{d_{n-1}} P_{n-2}(C_{n-2}^2 + C_{n-2}^{-2}) P_{n-2}
                            + \frac{d_1+d_{n-3}d_{n-2}}{d_{n-1}d_{n-2}} P_{n-2}t_{1,n-2}t_{1,n-2}^*P_{n-2}.
  \end{align*}
  This formula also holds for $n=3$, without the last term. Observe moreover that $\|P_{n-2}C_{n-2}^{\pm 2}P_{n-2}\|$ $\leq 1$ and
  $\|P_{n-2}t_{1,n-2}t_{1,n-2}^*P_{n-2}\| \leq d_1$ by composition. Now, the inequality established in the next Lemma shows that
  $t_{1,n}^* P_n t_{1,n}\geq \epsilon P_{n-2}$ for some $\epsilon >0$. As a result $t_{1,n}^* P_n t_{1,n} \in B(H_{n-2})$ is invertible and the result follows.
\end{proof}

\begin{lemma}\label{lem_dim_inequality}
  Still assuming $N\geq 3$, we have for any $n\geq 3$:
  \begin{displaymath}
    d_1 - \frac{2d_{n-2}}{d_{n-1}} > \frac{2}{d_{n-1}} + d_1 \frac{d_1+d_{n-3}d_{n-2}}{d_{n-1}d_{n-2}}.
  \end{displaymath}
\end{lemma}

\begin{proof}
  Denote $e_n = d_{n-1} - d_{n-2}$, $f_n = d_{n-5}+1+d_1$, with the convention
  $d_k = 0$ if $k<0$. For $n=3$ we have $e_3 = N^2-N-1$, $f_3 = 1+N$ and since
  $N\geq 3 > 1+\sqrt 3$ we have $e_3> f_3$.

  On the other hand we have, using the identity $Nd_{n-1}=d_n+d_{n-2}$ valid for $n\in\ZZ^*$:
  \begin{align*}
    e_{n+1}-e_n = d_n - 2d_{n-1}+d_{n-2} = (N-2)d_{n-1}
    \geq d_{n-1}\geq d_{n-4}-d_{n-5} = f_{n+1}-f_n.
  \end{align*}
  An easy induction then shows that we have $e_n > f_n$ for every $n\geq 3$.

  Multiplying this inequality by $d_1$ we find
  \begin{align*}
    & d_1 d_{n-1}- d_1 d_{n-2} > d_1 d_{n-5}+d_1+d_1^2 \\
    \Longleftrightarrow\quad & d_1 d_{n-1} - (d_1-1)d_{n-2} > d_1d_{n-5}+d_{n-2}+d_1+d_1^2 \\
    \Longrightarrow\quad & d_1 d_{n-1} - 2 d_{n-2} > d_1d_{n-3} + 2 + {d_1^2}/{d_{n-2}},
  \end{align*}
  using the facts $d_1 \geq 3$, $d_1 d_{n-5} + d_{n-2}\geq d_{n-4}+d_{n-2} \geq d_1d_{n-3} - 1$, and $d_{n-2} \geq 1$. Note that the inequality
  $d_1 d_{n-5} \geq d_{n-4}$, resulting from the fusion rules, does not hold for $n=4$, but one can check directly that in this case
  $d_1 d_{n-5} + d_{n-2} = d_1d_{n-3} - 1$.
\end{proof}

\begin{proposition}
  \label{prp_dim_ccirc}
  Still assume $N\geq 3$. For $n\geq 2$ we have $\dim B(H_n)^\ccirc = \dim p_n H^\ccirc = d_{2n} - d_{2n-2}$. For $n=1$ we have
  $\dim B(H_1)^\ccirc = \dim p_1 H^\ccirc = d_2$. 
\end{proposition}

\begin{proof}
  Recall the identification $B(H_n) \simeq H_n\otimes H_n$ via $X \mapsto x = (X\otimes\id) t_n$. In this identification the condition $(\id\otimes\Tr_1)(X) = 0$
  reads $(\id_{n-1}\otimes t^*\otimes\id_{n-1})(x) = 0$ and the corresponding kernel is $H_{2n} \subset H_n\otimes H_n$. This holds as well if $N=2$. Then the
  condition $(\Tr_1\otimes\id)(X) = 0$ reads $t_{1,2n}^*(x) = 0$, so that the result follows from the rank theorem and Lemma~\ref{lem_contract_ends}. For $n=1$
  both conditions coincide and we have $B(H_1)^\ccirc = B(H_1)^\circ \simeq H_2$. 
\end{proof}

On the other hand in the case $N=2$ one can check that $t_{1,2n}^*$ vanishes on $H_{2n}$ for all $n$, and thus $\dim B(H_n)^\ccirc = \dim p_n H^\ccirc = d_{2n}$
for all $n \geq 1$.

\bigskip

Recall then the ``rotation operators'' $\rho : B(H_k) \to B(H_k)$ already considered in \cite{FreslonVergnioux} and defined as follows:
$\rho(X) = (P_k\otimes t^*)(\id_1\otimes X\otimes\id_1)(t\otimes P_k)$. It follows from \cite[Lemma~3.1]{FreslonVergnioux} that $\rho$ stabilizes the subspace
$B(H_n)^\circ$ and contracts the Hilbert-Schmidt norm. On $B(H_n)^\ccirc$ it behaves even better: as the next lemma shows, it is a finite order unitary --- in
particular, it is diagonalizable.

\begin{lemma}
  \label{lem_rho}
  The map $\rho$ is a bijection from $B(H_n)^\ccirc$ to itself. Moreover we have $\rho^{2n} = \id$ and $\rho^* = \rho^{-1}$ on $B(H_n)^\ccirc$.
\end{lemma}

\begin{proof}
  We first note that for $X\in B(H_n)^\ccirc$ the element $Y = (\id_1\otimes\id_{n-1}\otimes t^*)(\id_1\otimes X\otimes\id_1) (t\otimes\id_{n-1}\otimes\id_1)$
  of $B(H_{n-1}\otimes H_1, H_1\otimes H_{n-1})$ is directly equal to $\rho(X)$. This is clear if $n=1$ since then $\id_1\otimes\id_{n-1} = \id_1 = P_1$. Assume
  $n\geq 2$. Since $t^*(\id\otimes A)t = \Tr_1(A)$ for any $A\in B(H_1)$ we have
  $(t^*\otimes\id_{n-2})Y = (\id\otimes t^*)[(\Tr_1\otimes\id)(X)\otimes\id_1)] = 0$, and similarly $Y(\id_{n-2}\otimes t) = 0$, so that using Remark~\ref{rk_ccirc} we have
  $Y = P_nYP_n = \rho(X)$. Then we compute $(\Tr_1\otimes\id)(Y)$ using again the morphism $t$. Thank to the conjugate equation we have
  \begin{align*}
    (\Tr_1\otimes\id)(Y) &=
                           (t^*\otimes\id_{n-1})(\id_1\otimes\id_1\otimes\id_{n-1}\otimes t^*)
                           (\id_1\otimes\id_1\otimes X\otimes\id_1) \\
                         & \hspace{4em} (\id_1\otimes t\otimes\id_{n-1}\otimes\id_1)
                           (t\otimes \id_{n-2}\otimes\id_1) \\
                         &= (\id_{n-1}\otimes t^*)(X\otimes\id_1)(t\otimes \id_{n-2}\otimes\id_1) = 0,
  \end{align*}
  since $X\in B(H_n)$. Similarly $(\id\otimes\Tr_1)(Y) = 0$ and this proves $\rho(B(H_n)^\ccirc) \subset B(H_n)^\ccirc$. The conjugate equation also implies
  that $(t^*\otimes\id_n)(\id_1\otimes Y\otimes\id_1)(\id_n\otimes t) = X$ so that $\rho$ is a bijection with
  $\rho^{-1}(X) = (t^*\otimes P_n)(\id_1\otimes X\otimes\id_1)(P_n\otimes t)$. This holds as well for $n=1$.

  Let us check that $\rho^{-1}$ is the adjoint of $\rho$ with respect to the Hilbert-Schmidt scalar product. Using twice the conjugate equation we have, for
  $X$, $Y\in B(H_n)^\ccirc$:
  \begin{align*}
    \Tr_n(\rho^{-1}(X)^* Y) &= (\Tr_1\otimes\Tr_{n-1}) [
                              (\id_n\otimes t^*)(\id_1\otimes X^*\otimes\id_1)(t\otimes \id_n) Y] \\
                            &= \Tr_{n-1}[(t^*\otimes \id_{n-1}\otimes t^*)(\id_1\otimes \id_1\otimes X^*\otimes\id_1) \\
                            & \hspace{4.1em} (\id_1\otimes t\otimes\id_n)(\id_1\otimes Y)(t\otimes \id_{n-1})] \\
                            &= \Tr_{n-1}[(\id_{n-1}\otimes t^*)(X^*\otimes\id_1)(\id_1\otimes Y)(t\otimes \id_{n-1})] \\
                            &= \Tr_{n-1}[(\id_{n-1}\otimes t^*)(X^*\otimes\id_1)(\id_1\otimes\id_{n-1}\otimes t^*\otimes\id_1) \\
                            & \hspace{4.1em} (\id_1\otimes Y\otimes\id_1\otimes\id_1)(t\otimes \id_{n-1}\otimes t)] \\
                            &= (\Tr_{n-1}\otimes\Tr_1)[X^*(\id_1\otimes\id_{n-1}\otimes t^*)
                              (\id_1\otimes Y\otimes\id_1)(t\otimes \id_{n-1}\otimes \id_1)] \\
                            &= \Tr_n(X^*\rho(Y)).
  \end{align*}

  Recall the notation $t_1^n \in \Hom(\CC,H_1^{\otimes n}\otimes H_1^{\otimes n})$ from the Preliminaries and consider the associated antilinear map
  $j^n : H_1^{\otimes n} \to H_1^{\otimes n}$ given by $j^n(\zeta) = (\zeta^*\otimes\id_n)t_1^n$. If $(e_i)_i$ is the canonical basis of $H_1 = \CC^N$ we have
  $j^n(e_{i_1}\otimes\cdots\otimes e_{i_n}) = e_{i_n}\otimes\cdots\otimes e_{i_1}$ so that $j^n\circ j^n = \id$ and $j^n(\zeta) = (\id_n\otimes\zeta^*)t^n$. Using
  the fact that $\rho(X) = (\id_1\otimes\id_{n-1}\otimes t^*)(\id_1\otimes X\otimes\id_1) (t\otimes\id_{n-1}\otimes\id_1)$ for $X \in B(H_n)^\ccirc$ we have
  easily $\rho^n(X) = (\id_n\otimes t^{n*})(\id_n\otimes X\otimes\id)(t^n\otimes\id_n)$, which yields ${(\zeta\mid \rho^n(X)\xi)} = (j^n\xi\mid Xj^n\zeta)$ for all
  $\zeta$, $\xi\in H_n$. Applying this identity a second time we get $\rho^{2n}(X) = X$.
\end{proof}

We shall now analyze the submodule $AxA$ when $x$ belongs to $H^\ccirc$. In the analogy with the generator MASA $a_1''\subset\Ll(F_N)$ in a free group factor,
the vectors $x_{i,j}$ below play the role of the words $a_1^iga_1^j\in F_N$, where $g\in F_N$ does not start nor end with $a_1$.

\begin{notation} 
  For $x \in H$ and $i$, $j\in\NN$ we denote $x_{i,j} = \sum_n p_{i+n+j}(\chi_i p_n(x) \chi_j)$. For $X\in B(H_n)$ we denote
  $X_{i,j} = P_{i+n+j} (\id_i\otimes X\otimes \id_j)P_{i+n+j} \in B(H_{i+n+j})$.
\end{notation}

\begin{remark}
  \label{rk_shift_tannaka}
  The sum in the definition of $x_{i,j}$ indeed converges in $H$, since its terms are pairwise orthogonal an satisfy the inequality
  $\|\chi_i p_n(x) \chi_j\| \leq \|\chi_i\| \|\chi_j\| \|p_n(x)\|$. This yields a map $(x\mapsto x_{i,j})$ which is linear and bounded from $H$ to $H$. We will
  mostly use the notation $x_{i,j}$ in the case when $x$ belongs to one of the subspaces $p_nH$.

  Note also that we have by construction $u_n(X)_{i,j} = u_{i+n+j}(X_{i,j})$ for $X\in B(H_n)$. Indeed, denote $x = u_n(X)$ and recall that
  $\chi_i = u_i(\id_i)$, $\chi_j = u_j(\id_j)$.  To compute the component $p_{i+n+j}(\chi_i x\chi_j)$ one has to use an orthonormal basis of isometric
  intertwiners $T : H_{i+n+j} \to H_i\otimes H_n\otimes H_j$. But according to the fusion rules there is only one such intertwiner up to a phase, and by
  construction of the spaces $H_k$ we can take for it the canonical inclusion of $H_{i+n+j}$ into $H_i\otimes H_n\otimes H_j\subset H_1^{\otimes (i+n+j)}$, whose
  adjoint is given by $P_{i+n+j}$.

  Finally, we record the fact that $X_{i,j}$ is the orthogonal projection of $\id_i\otimes X\otimes\id_j \in B(H_1^{\otimes (i+n+j)})$ onto $B(H_{i+n+j})$, with
  respect to the Hilbert-Schmidt scalar product --- indeed for any $Y$, $Z\in B(H_1^{\otimes (i+n+j)})$ we have
  $\Tr(Y^*P_{i+n+j}ZP_{i+n+j}) = \Tr((P_{i+n+j}YP_{i+n+j})^*Z)$.
\end{remark}

\begin{proposition}
  \label{prp_simple_bimod}
  Fix $k\in\NN^*$, $X \in B(H_k)^\ccirc$ an eigenvector of $\rho$ and $x = u_k(X)\in H^\ccirc$. Then we have $\Aa x\Aa = \Span\{x_{i,j}\mid i,j\in\NN\}$.
\end{proposition}

\begin{proof}
  Let us prove by induction over $i+j = n-k$ that $x_{i,j}\in \Aa x\Aa$. Assume that $x_{p,q}\in \Aa x\Aa$ if $p+q\leq i+j$ and compute $\chi_1x_{i,j}$. We have
  $p_{n-1}(\chi_1 x_{i,j}) = (\kappa_{n-1}^{1,n})^2u_{n-1}(\id_1*_{n-1}X_{i,j})$ and
  \begin{align} \label{eq_simple_bimod}
    \id_1*_{n-1}X_{i,j} &= (t^*\otimes P_{n-1})(\id_1\otimes X_{i,j})(t\otimes P_{n-1}) \\ \nonumber
                        &= P_{n-1}(t^*\otimes\id_{n-1})(\id_1\otimes P_n)
                          (\id_{i+1}\otimes X\otimes \id_j)(\id_1\otimes P_n)(t\otimes\id_{n-1}) P_{n-1}.
  \end{align}
  Since the Jones-Wenzl projections $P_n$ are intertwiners, we can expand them  into linear combinations of Temperley-Lieb diagrams, so that $\id_1*_{n-1}X_{i,j}$ is a linear
  combination of maps of the form $P_{n-1} T_\pi(X) P_{n-1}$, where $\pi$ is a Temperley-Lieb diagram with $n-1$ upper and lower points and an internal box with
  $2k$ points, and $T_\pi : B(H_1^{\otimes k}) \to B(H_1^{\otimes (n-1)})$ is the associated map. Since we multiply on the left and on the right by $P_{n-1}$ and
  $X = P_kXP_k$ belongs to $B(H_k)^\ccirc$, the term associated with $\pi$ vanishes as soon as a string of $\pi$ connects two upper points, or two lower points, or
  two internal points.

  Now consider the string originating from the first top left external point in a diagram $\pi$ such that $P_{n-1} T_\pi(X) P_{n-1}\neq 0$. If it is not
  connected to the internal box, it has to connect the top left point to the first bottom left external point, otherwise some other string would have to connect
  two upper or two lower external point, because of the non-crossing constraint. We can re-apply this reasoning to the following top left external points, until
  we find an external point connected to $X$, say with index $p+1$ on the top external edge. Moreover up to replacing $X$ by its image $\rho^l(X)$ under some
  iterated rotation we can assume that this external point is connected to the first top left point of the internal box by a vertical edge. Thus our
  diagram has the following form:
  \begin{displaymath}
    P_{n-1} T_\pi(X) P_{n-1} =
    ~\vcenter{\hbox{
\setlength{\unitlength}{2072sp}%
\begingroup\makeatletter\ifx\SetFigFont\undefined%
\gdef\SetFigFont#1#2#3#4#5{%
  \reset@font\fontsize{#1}{#2pt}%
  \fontfamily{#3}\fontseries{#4}\fontshape{#5}%
  \selectfont}%
\fi\endgroup%
\begin{picture}(2274,1966)(3139,-4670)
\thinlines
{\color[rgb]{0,0,0}\put(4411,-3256){\line( 0,-1){ 90}}
}%
{\color[rgb]{0,0,0}\put(4411,-3886){\line( 0,-1){ 90}}
}%
{\color[rgb]{0,0,0}\put(4411,-4156){\line( 0,-1){ 90}}
}%
{\color[rgb]{0,0,0}\put(4951,-4156){\line( 0,-1){ 90}}
}%
{\color[rgb]{0,0,0}\put(3871,-2986){\line( 0,-1){360}}
}%
{\color[rgb]{0,0,0}\put(3871,-3886){\line( 0,-1){ 90}}
}%
{\color[rgb]{0,0,0}\put(3871,-4156){\line( 0,-1){ 90}}
}%
{\color[rgb]{0,0,0}\put(3331,-2986){\line( 0,-1){1260}}
}%
{\color[rgb]{0,0,0}\put(4411,-3076){\line( 0, 1){ 90}}
}%
{\color[rgb]{0,0,0}\put(4951,-2986){\line( 0,-1){ 90}}
}%
{\color[rgb]{0,0,0}\put(4411,-3076){\line( 0, 1){ 90}}
}%
{\color[rgb]{0,0,0}\put(4951,-2986){\line( 0,-1){ 90}}
}%
{\color[rgb]{0,0,0}\put(3601,-2986){\line( 0,-1){1260}}
}%
{\color[rgb]{0,0,0}\put(4141,-3076){\line( 0, 1){ 90}}
}%
{\color[rgb]{0,0,0}\put(4141,-3346){\line( 0, 1){ 90}}
}%
{\color[rgb]{0,0,0}\put(4681,-2986){\line( 0,-1){ 90}}
}%
{\color[rgb]{0,0,0}\put(4681,-3256){\line( 0,-1){ 90}}
}%
{\color[rgb]{0,0,0}\put(3781,-3886){\framebox(990,540){}}
}%
{\color[rgb]{0,0,0}\put(4141,-3886){\line( 0,-1){ 90}}
}%
{\color[rgb]{0,0,0}\put(4681,-3886){\line( 0,-1){ 90}}
}%
{\color[rgb]{0,0,0}\put(5221,-2986){\line( 0,-1){ 90}}
}%
{\color[rgb]{0,0,0}\put(4141,-4156){\line( 0,-1){ 90}}
}%
{\color[rgb]{0,0,0}\put(4681,-4156){\line( 0,-1){ 90}}
}%
{\color[rgb]{0,0,0}\put(5221,-4156){\line( 0,-1){ 90}}
}%
{\color[rgb]{0,0,0}\put(3151,-4246){\framebox(2250,1260){}}
}%
\put(4276,-3751){\makebox(0,0)[b]{\smash{{\SetFigFont{11}{13.2}{\rmdefault}{\mddefault}{\updefault}{\color[rgb]{0,0,0}$\rho^l(X)$}%
}}}}
\put(4276,-2851){\makebox(0,0)[b]{\smash{{\SetFigFont{11}{13.2}{\rmdefault}{\mddefault}{\updefault}{\color[rgb]{0,0,0}$P_{n-1}$}%
}}}}
\put(4276,-4606){\makebox(0,0)[b]{\smash{{\SetFigFont{11}{13.2}{\rmdefault}{\mddefault}{\updefault}{\color[rgb]{0,0,0}$P_{n-1}$}%
}}}}
\end{picture}%
      }}~.
  \end{displaymath}

  By the same reasoning we see that the $(p+2)^{\text{th}}$ external point on the top edge connects to the second top left point of the internal box $\rho^l(X)$
  by a vertical string (if $k\geq 2$). Indeed if it is connected to another point of the internal box, there will be a string joining two  internal
  points, and if it is connected to a bottom external point, there will be either a string connecting two upper external points, or a string connecting two internal points.
  Continuing like this, we see that the only possibility for a non-vanishing diagram is one composed entirely of vertical lines, i.e.\
  $P_{n-1}(\id_p\otimes\rho^l(X)\otimes\id_q)P_{n-1}$, with $l\in\ZZ$ and $p+q = n-k-1$. Since $X$ is an eigenvector of $\rho$, this shows that
  $p_{n-1}(\chi_1 x_{i,j})$ is a linear combination of vectors $x_{p,q}$ with $p+q < i+j$, which belong to $\Aa x\Aa$ by the induction hypothesis. Note that if
  $i=j=0$ we have $p_{n-1}(\chi_1 x_{i,j}) = 0$ ; this can also be checked directly because~\eqref{eq_simple_bimod} then equals $(\Tr_1\otimes\id)(X)$.
  
  We have $p_{n+1}(\chi_1x_{i,j}) = p_{n+1}(\chi_1\chi_i x\chi_j) = x_{i+1,j}$ because $p_{n+1}(\chi_1 y) = 0$ if $y\in p_{k'}H$ with $k'<n$. We have thus
  $x_{i+1,j} = \chi_1x_{i,j} - p_{n-1}(\chi_1x_{i,j})$ and it follows that $x_{i+1,j}$ belongs to $\Aa x\Aa$. One can proceed in the same way on the right to
  show that $x_{i,j+1}$ belongs to $\Aa x\Aa$. By induction we have proved $x_{i,j}\in \Aa x\Aa$ for all $i$, $j$. Moreover from the identities
  $\chi_1x_{i,j} = x_{i+1,j} + p_{n-1}(\chi_1x_{i,j})$, $x_{i,j}\chi_1 = x_{i,j+1} + p_{n-1}(x_{i,j}\chi_1)$ and the fact that $p_{n-1}(\chi_1x_{i,j})$,
  $p_{n-1}(x_{i,j}\chi_1)$ are linear combinations of vectors $x_{p,q}$ it also follows that $\Span\{x_{i,j}\}$ is stable under the left and right actions of
  $\Aa$.
\end{proof}

\begin{remark}
  Using the Jones-Wenzl recursion relations, one can prove more precisely that $p_{n-1}(x_{i,j}\chi_1)$ is a linear combination of $x_{i-1,j}$,
  $\rho^{\pm 1}(x)_{i,j-1}$ and $x_{i+1,j-2}$, where we abusively write $\rho(u_k(X)) := u_k(\rho(X))$.
\end{remark}

\begin{notation}\label{not_basis}
  Choose for all $k\geq 1$ a basis $(X_r)_r \subset B(H_k)^\ccirc$ of eigenvectors of $\rho$, normalized in such a way that $\|u_k(X_r)\|_2=1$.
  Denote $\W_k=(u_k(X_r))_r$ its image in $p_kH^\ccirc$. Put as well $\W = \bigcup_k \W_k$, which is a linearly independent family consisting of unital vectors
  in $H^\ccirc$. For $x\in \W$ we denote $H(x) = \overline{AxA}$, and for $k\in\NN^*$, $H(k) = \overline{A\W_kA}$, using the left and right actions of $A$. The
  previous lemma shows that the vectors $x_{i,j}$ span a dense subspace of $H(x)$.
\end{notation}

\begin{proposition}
  \label{prp_decomp_bimod}
  The family $\W$ spans $H^\circ$ as a closed $A{,}A$-bimodule. Moreover, for $x\neq y \in \W$ we have $H(x) \bot H(y)$.
\end{proposition}

\begin{proof}
  Denote $L_n = \Span\{X_{i,j} \mid X\in \W_k, k\leq n, i+j+k=n\} \subset B(H_n)^\circ$, and let us show by induction over $n\geq 1$ that $L_n =
  B(H_n)^\circ$. For $n = 1$ we have by definition $L_1 = \Span \W_1 = B(H_1)^\ccirc = B(H_1)^\circ$. Assume that $L_n = B(H_n)^\circ$ and take
  $Y \in L_{n+1}^\bot\cap B(H_{n+1})^\circ$. We want to show that $Y = 0$.  We consider first $(\Tr_1\otimes\id)(Y)$. For any generator $X_{i,j}$ of $L_n$ we
  have
  \begin{align*}
    \Tr_n(X_{i,j}^*(\Tr_1\otimes\id)(Y))
    = (\Tr_1\otimes\Tr_n)(P_{n+1}(\id_1\otimes X_{i,j}^*)P_{n+1}Y ) 
    = \Tr_{n+1}(X_{i+1,j}^*Y) = 0,
  \end{align*}
  by assumption on $Y$. Since $L_n = B(H_n)^\circ$, this implies $(\Tr_1\otimes\id)(Y) = 0$. Similarly, $(\id\otimes\Tr_1)(Y) = 0$. As a result,
  $Y\in B(H_{n+1})^\ccirc$. But $B(H_{n+1})^\ccirc \subset L_{n+1}$, and $Y \bot L_{n+1}$, so that we have indeed proved $Y = 0$. Taking into account
  Proposition~\ref{prp_simple_bimod}, this proves that $p_n H^\circ \subset \Span\Aa \W \Aa$ for every $n$ and the first result follows.

  For the second part of the statement, take $x\in \W_k$, $y\in \W_l$ distinct, with $k\leq l$. The subspaces $H(x)$, resp. $H(y)$ are spanned by vectors
  $\chi_1^px\chi_1^q$, resp. $\chi_1^r y\chi_1^s$. We have
  \begin{displaymath}
    (\chi_1^px\chi_1^q \mid \chi_1^r y\chi_1^s) = (x \mid \chi_1^{p+r}y\chi_1^{q+s}) = (x \mid p_k(\chi_1^{p+r}y\chi_1^{q+s})).
  \end{displaymath}
  But $\chi_1^{p+r}y\chi_1^{q+s} \in \Span\{y_{i,j}\}$ and $y_{i,j} \in p_{i+l+j}H$. Since $k\leq l$ this implies that
  $p_k(\chi_1^{p+r}y\chi_1^{q+s}) \in \CC y$ and the result follows since $x\bot y$.
\end{proof}

Denote $B\subset B(H^\circ)$ the commutant of the left and right actions of $A$. Being the commutant of an abelian algebra, it is a type I von Neumann algebra,
which can be decomposed into type I$_n$ algebras. The numbers $n\in\NN^*\cup\{\infty\}$ appearing in this way form the Puk\'anszky invariant of the maximal
abelian subalgebra $A$.

\begin{corollary}\label{crl_pukanszky}
  The bimodule $H^\circ$ is isomorphic to $L^2(A)\otimes \ell^2(W)\otimes L^2(A)$. In particular the Puk\'anszky invariant of $A\subset M$ is $\{\infty\}$.
\end{corollary}

\begin{proof}
  Indeed the proof of in \cite[Theorem~5.10]{FreslonVergnioux} shows that the measure on $\itv[{-2},2]\times \itv[{-2},2]$ induced by a given
  $\zeta\in H^\circ\cap \Aa$ and the action of $A\otimes A$ on $H^\circ$ is equivalent to the Lebesgue measure --- in fact it has a non-zero analytic
  density. As a result, the corresponding cyclic bimodule $H(\zeta)$ is isomorphic to the coarse bimodule $L^2(A)\otimes L^2(A)$. This applies to
  $\zeta = x \in \W$. Now, Proposition~\ref{prp_decomp_bimod} shows that we have an isomorphism of $A,A$-bimodules
  \begin{displaymath}
    H^\circ  \simeq \bigoplus_{x\in \W} L^2(A)\otimes L^2(A) \simeq L^2(A)\otimes\ell^2(\W)\otimes L^2(A).
  \end{displaymath}
  As a result $(A\otimes A)'\cap B(H^\circ) \simeq A\bar\otimes B(\ell^2(\W))\bar\otimes A$ and the value of the Puk\'anszky invariant follows since $\W$ is
  infinite.
\end{proof}

\begin{corollary}\label{crl_independent}
  For $x\in \W$ the vectors $x_{i,j}$ are linearly independent.
\end{corollary}

\begin{proof}
  Since the subspaces $p_n H^\circ$ are pairwise orthogonal, it suffices to consider a subfamily $(x_{i,j})$ with $i+k+j = n$ fixed. Note that
  $\#\{x_{i,j} \mid i+k+j = n\} = n-k+1$.  According to Proposition~\ref{prp_decomp_bimod} we have
  \begin{displaymath}
    p_n H^\circ = \bigoplus_{k\leq n}\bigoplus_{x\in \W_k} \Span \{x_{i,j} \mid i+k+j = n\}, 
  \end{displaymath}
  so that $\dim p_nH^\circ \leq \sum_{k=1}^n (n-k+1)\# \W_k$. We will prove that this estimate is an equality, so that
  $\dim\Span \{x_{i,j} \mid i+k+j = n\} = n-k+1$ for all $x\in \W_k$, $k\leq n$, which implies the linear independence.

  Recall from Proposition~\ref{prp_dim_ccirc} that $\# \W_k = \dim B(H_k)^\ccirc = d_{2k}-d_{2k-2}$ for $k\geq 2$, and $\# \W_1 = d_2$. We have then
  \begin{align*}
    {\ts\sum_{k=1}^n} (n-k+1)\# \W_k &= {\ts\sum_{k=1}^n} (n-k+1)d_{2k} - {\ts\sum_{k=2}^n} (n-k+1) d_{2k-2} \\
                                    &={\ts\sum_{k=1}^n} (n-k+1)d_{2k} - {\ts\sum_{k=1}^{n-1}} (n-k) d_{2k} \\
                                    &={\ts\sum_{k=1}^{n}}d_{2k} = d_n^2 - 1 = \dim B(H_n)^\circ = \dim p_nH^\circ.
  \end{align*}
  The computation of the sum in the last line follows from the decomposition of $u_n\otimes u_n$ given by the fusion rules.
\end{proof}

\begin{remark} \label{rk_isomorphism}
  As a result, the map $\Phi : c_c(\NN)\otimes H^\ccirc\otimes c_c(\NN)\to H^\circ$, $\delta_i\otimes x\otimes\delta_j\mapsto x_{i,j}$ is injective with
  dense image. It is however not an isometry. We will see in the next section that, at least for ``large $N$'', it extends to an isomorphism from
  $\ell^2(\NN)\otimes H^\ccirc\otimes\ell^2(\NN)$ to $H^\circ$.
\end{remark}

We end this section with one further property of elements of $H^\ccirc$ which
is established using the action of planar tangles and will be used in
Section~\ref{sec_orthogonality}.

\begin{proposition} \label{prp_propagation} For any $\zeta\in H(k)$, $\zeta'\in H(k')$ and $y\in p_nH$ with $n<|k-k'|$ we have $y\zeta \bot \zeta'$. If $k>n$ we
  have $y\zeta \in H^\circ$.
\end{proposition}

\begin{proof}
  By bilinearity one can assume $\zeta = x_{i,j} = u_{i+k+j}(X_{i,j})$, $\zeta' = x'_{i',j'} = u_{i'+k'+j'}(X^{\prime *}_{i',j'})$, with $X\in B(H_k)^\ccirc$,
  $X' \in B(H_{k'})^\ccirc$. Denote also $y = u_n(Y)\in M$ with $Y\in B(H_n)$. Then the product $y\zeta$ is a linear combination of elements $u_m(Y*_m X_{i,j})$ with
  $m = n+i+k+j-2a$. Using the Peter-Weyl relations~\eqref{eq_peter_weyl} it thus suffices to prove that $\Tr (X_{i',j'}'(Y*_mX_{i,j})) = 0$, with
  $m = i'+k'+j' = i+k+j-2a$.

  By definition, the element in the trace is computed by the following formula:
  \begin{align*}
    &P_m(\id_{i'}\otimes X'\otimes\id_{j'})P_m(\id_{n-a}\otimes t_a^*\otimes\id_{i+j+k-a})(\id_n\otimes P_{i+k+j}) \\
    &\hspace{2cm} (Y\otimes \id_i\otimes X\otimes\id_j)(\id_n\otimes P_{i+k+j})(\id_{n-a}\otimes t_a\otimes\id_{i+j+k-a})P_m.
  \end{align*}
  Since $P_m$, $P_{i+k+j}$, $t_a$ are morphisms, this element is a linear combination of planar tangles on $m$ lower and upper points, with $3$ inside boxes,
  applied to $X$, $X'$, $Y$. Since $\Tr(Z) = t_m^*(Z\otimes\id_m)t_m$, the scalar $\Tr (X_{i',j'}'(Y*_mX_{i,j}))$ is itself a linear combination of such planar
  tangles $T$, without external points, applied to $X$, $X'$, $Y$.

  Fix one of these tangles and consider the strings starting at one of the $2k$ points on the internal box corresponding to $X$. These strings can have their
  second ends on $X$, $X'$ or $Y$.  If $2k > 2k' + 2n$, the first possibility must happen at least once, i.e.\ there is a string connecting two points of $X$. Since
  the strings are non crossing, this implies that there is even a string connecting two consecutive points of the internal box corresponding to $X$. But then the value of
  the tangle applied to $X$, $X'$, $Y$ is $0$ since $X\in B(H_k)^\ccirc$: see Remark~\ref{rk_ccirc}.

  If $k<k'-n$ we proceed in the same way by considering strings starting on the internal box corresponding to $X'$. The last assertion of the statement amounts
  to considering the trace $\Tr (Y*_mX_{i,j})$ which is again a linear combination of planar tangles without external points applied to $Y$ and $X$, and if
  $k>n$ the same argument as above applies.
\end{proof}

\section{Invertibility of the Gram Matrix}
\label{sec_gram}

In this section we fix $k\in\NN^*$, $x = u_k(X)\in p_kH^\ccirc$ with
$X\in B(H_k)^\ccirc$ an eigenvector of $\rho$ with associated eigenvalue $\mu$,
$|\mu| = 1$. Recall the notation $x_{i,j} = p_{i+k+j}(\chi_i x\chi_j)$,
$X_{i,j} = P_{i+k+j}(\id_i\otimes X\otimes\id_j)P_{i+k+j}$. We know from the
previous section that $(x_{i,j})$ spans a dense subspace of the bimodule
$AxA$. Our aim is now to show that it is a Riesz basis, i.e.\ it implements an
isomorphism between $H(x) = \overline{AxA}$ in $H$ and
$\ell^2(\NN\times\NN)$. We will only achieve this for small $q$, i.e.\ large
$N$. We thus consider the associated Gram matrix, which is block diagonal since
$p_m H \bot p_n H$ for $m\neq n$. Let us formalize this as follows:

\begin{notation} \label{not_gram} We fix $k\in \NN^*$ and a unital vector $x = u_k(X)\in \W_k \subset p_k H^{\circ\circ}$. We denote $G = G(x)$ the Gram matrix
  of the family $(x_{i,j})_{i,j}\subset H$, and $G_n = G_n(x)$ its diagonal block corresponding to indices $(i,j)$ such that $x_{i,j}\in p_n H$, i.e.\
  $i+k+j = n$. Since $k$ is fixed we drop the second index $j$ and denote $x_{n;i} = x_{i,j}$, $X_{n;i} = X_{i,j}$.  For $i$, $p\in \{0,\ldots,n-k\}$ we denote
  accordingly
  \begin{displaymath}
    G_{n;i,p} = (x_{n;i} \mid x_{n;p}) = d_n^{-1}(X_{n;i} \mid X_{n;p}).
  \end{displaymath}
\end{notation}

The second equality follows from the Peter-Weyl-Woronowicz orthogonality relations, using the Hilbert-Schmidt scalar product in $B(H_n)$. Let us record the
following symmetry properties of $G$:

\begin{lemma} \label{lem_symmetries}
  For any $n = i+k+j = p+k+q$ we have
  \begin{displaymath}
    G_{n;i,p}(x) = \overline{G_{n;p,i}(x)} = G_{n;q,j}(x^*) = G_{n;j,q}(S(x)).
  \end{displaymath}
\end{lemma}

\begin{proof}
  As a Gram matrix, $G_n$ is self-adjoint, which corresponds to the first equality.  Define maps $J$, $U : H\to H$ by $J(x) = x^*$, $U(x) = S(x)$ where $S$ is
  the antipode. The maps are surjective isometries because we are in the Kac case, and since $u_n$ is orthogonal they stabilize $p_nH$ and send $\chi_n$ to itself. We have then
  \begin{align*}
      G_{n;i,p}(x) &= (p_n(\chi_i x \chi_j)\mid p_n(\chi_p x\chi_q)) = (J p_n(\chi_p x\chi_q)\mid J p_n(\chi_i x \chi_j)) \\
                   &= ( p_n(\chi_q x^*\chi_p)\mid p_n(\chi_j x^* \chi_i))=G_{n;q,j}(x^*) \\
                   &= (U p_n(\chi_i x\chi_j)\mid U p_n(\chi_p x \chi_q)) =  (p_n(\chi_j S(x)\chi_i) \mid p_n(\chi_q S(x) \chi_p)) = G_{n;j,q}(S(x)). \qedhere
  \end{align*}
\end{proof}

Our main aim is then to show the existence of a constant $C$ such that $\|G_n\|$, $\|G_n^{-1}\|\leq C$ for all $n$. In fact we even want the constant $C$ to be
uniform over $k$ and $x\in \W_k$, so that the map $\Phi$ from Remark~\ref{rk_isomorphism} will indeed be an isomorphism.

\bigskip

We shall first show that the Gram matrix $G = G(x)$ is bounded as an operator on $\ell^2(\NN\times\NN)$. We start with an easy estimate, which is not
sufficient for this purpose but will be useful later. We then prove an off-diagonal decay estimate for the coefficients of the Gram matrix, see
Lemma~\ref{lem_gram_decay}, using the improvement of the main estimate of \cite{FreslonVergnioux} established at Lemma~\ref{lem_trace_proj}. These two results
easily imply the boundedness of $G(x)$ on $\ell^2(\NN\times\NN)$, which we record at Proposition~\ref{prop_gram_bounded_0}. Note however that the constant $C$
obtained in this way depends on $k$, so that one cannot deduce the boundedness of the whole Gram matrix. This will be improved later.

\begin{lemma}
  \label{lem_gram_entry_bound}
  We have $\|x_{n;i}\|_2\leq (1-q^2)^{-3/2}\|x\|_2$, hence $|G_{n;i,p}| \leq (1-q^2)^{-3} \|x\|_2^2$, for all $n$, $0\leq i,p \leq n-k$. 
\end{lemma}

\begin{proof}
  We have
  \begin{align*}
    \|X_{n;i}\|_2^2 &= \Tr(P_n (\id_i\otimes X^*\otimes \id_j) P_n (\id_i\otimes X\otimes \id_j) P_n) \\
    &\leq \Tr(P_n (\id_i\otimes X^*X\otimes \id_j) P_n) \leq
  \Tr(P_i\otimes X^*X\otimes P_j) = d_i d_j \|X\|_2^2,
  \end{align*}
  hence $\|x_{n;i}\|_2^2 \leq (d_id_jd_k/d_n)\|x\|_2^2$. The result then follows from
  Lemma~\ref{lem_dimensions}.
\end{proof}

\begin{lemma}
  \label{lem_gram_decay}
  For every $q_0\in\itv]0,1[$ there exists $\alpha\in\itv]0,1[$ and $C>0$ depending only on $q_0$ such that $|G_{n;i,p}|\leq Cq^{\alpha(|p-i|-k)-2-k}\|x\|_2^2$
  for all $n$, $i$, $p$ such that $|p-i| \geq k$, as soon as $q\in\itv]0,{q_0}]$.
\end{lemma}

\begin{proof}
  The reader will find after the proof a graphical ``explanation'' of the computations.  Write
  $n = i+k+j = p+k+q$.  We have
  $(x_{n;i}\mid x_{n;p}) = \Tr(X_{i,j}^*X_{p,q})/d_n$. We first assume
  $p-i\geq k$ and put $a = \lfloor(p-i-k)/2\rfloor$.  By
  Lemma~\ref{lem_highest_weight} we have
  $\|P_n - (\id_{i+k+a}\otimes P_{j-a})(P_p\otimes\id_{k+q})\| \leq D
  q^{p-(i+k+a)} \leq Dq^a$, where $D>0$ is a constant depending only on $q_0$. This yields
  \begin{align} \nonumber
    (x_{n;i}\mid x_{n;p}) &= d_n^{-1}\Tr_n[P_n (\id_i\otimes X^*\otimes\id_j)P_n(\id_p\otimes X\otimes\id_q)P_n ] \\ \nonumber
                          &\simeq d_n^{-1}\Tr_n[P_n(\id_i\otimes X^*\otimes \id_{a}\otimes P_{j-a})(P_p\otimes X\otimes\id_q)P_n] \\
    \label{eq_gram_decay}
                          &= d_n^{-1}(\Tr_{i+k}\otimes \Tr_{a}\otimes \Tr_{j-a})[(\id_i\otimes X^*\otimes \id_{a}\otimes P_{j-a})(P_p\otimes X\otimes\id_q)P_n].
  \end{align}
  Since $d_n^{-1}\Tr_n(P_n \cdot P_n)$ is a state, the error is bounded by $D q^a\|X\|^2$. In the last expression, the projection $P_p\otimes\id_k\otimes\id_q$
  is absorbed in $P_n$, and since $j-a\geq k+q$ the partial trace $(\id_{i+k}\otimes \Tr_{a}\otimes\id_{j-a})(P_n)$ appears. We know from
  Lemma~\ref{lem_trace_proj} that this partial trace is equal to a multiple $\lambda$ of the identity up to $E d_aq^{\lfloor\beta a\rfloor}$ if
  $q\in\itv]0,{q_0}]$, for some $\beta \in\itv]0,1[$ and $E>0$ depending only on $q_0$. Applying the remaining traces and dividing by $d_n$ the total error is
  controlled by
  \begin{align*}
    Dq^a\|X\|^2 + Eq^{\lfloor\beta a\rfloor} \frac{d_{i+k}d_{a}d_{j-a}}{d_n}\|X\|^2
    &\leq q^{\lfloor\beta a\rfloor} (D + E/(1-q_0^2)^3)\|X\|_2^2 \\
    &\leq C q^{\alpha(p-i-k) - 2-k}\|x\|^2_2,
  \end{align*}
  for $C = [D + E/(1-q_0^2)^3]/(1-q_0)$ and $\alpha = \beta/2$ --- recall that
  $\|X\|_2^2 = d_k \|x\|_2^2 \leq q^{-k} \|x\|_2^2/(1-q)$. But if we replace
  $(\id_{i+k}\otimes \Tr_{a}\otimes\id_{j-a})(P_n)$ by
  $\lambda(P_{i+k}\otimes P_{j-a})$ in~\eqref{eq_gram_decay} we can see the
  trace $\Tr((\id_i\otimes X^*)P_{i+k})$ which vanishes (as well as
  $\Tr((\id_{a'}\otimes X\otimes \id_q)P_{j-a})$, where $a' = \lceil (p-i-k)/2\rceil$).

  This proves the result if $p-i\geq k$. If $i-p\geq k$ we can proceed in the same way ``on the other side'' and the result follows because then
  $q-j = |p-i| \geq k$.
\end{proof}

We give below a graphical version of the above proof, for the convenience of the
reader, in the case $p-i\geq k$. Of course it is still necessary to carry out
the quantitative bookkeeping of approximations, as we did above.  It is
possible to draw similar graphical computations for many lemmata in this section
and the following ones.

\medskip

\begin{center}
\setlength{\unitlength}{1657sp}%
\begingroup\makeatletter\ifx\SetFigFont\undefined%
\gdef\SetFigFont#1#2#3#4#5{%
  \reset@font\fontsize{#1}{#2pt}%
  \fontfamily{#3}\fontseries{#4}\fontshape{#5}%
  \selectfont}%
\fi\endgroup%
\begin{picture}(11280,9204)(-4154,-12313)
\thinlines
{\color[rgb]{0,0,0}\put(3241,-4921){\line( 0,-1){1080}}
}%
{\color[rgb]{0,0,0}\put(4321,-4921){\line( 0,-1){1080}}
}%
{\color[rgb]{0,0,0}\put(6211,-5731){\line( 0,-1){1080}}
}%
{\color[rgb]{0,0,0}\put(4591,-5731){\dashbox{107}(1890,540){}}
}%
{\color[rgb]{0,0,0}\put(5401,-6541){\dashbox{107}(540,540){}}
}%
{\color[rgb]{0,0,0}\put(2971,-6541){\dashbox{107}(2160,540){}}
}%
{\color[rgb]{0,0,0}\put(3781,-4921){\line( 0,-1){270}}
}%
{\color[rgb]{0,0,0}\put(3781,-5731){\line( 0,-1){270}}
}%
{\color[rgb]{0,0,0}\put(4861,-5731){\line( 0,-1){270}}
}%
{\color[rgb]{0,0,0}\put(5671,-5731){\line( 0,-1){270}}
}%
{\color[rgb]{0,0,0}\put(2971,-7351){\dashbox{107}(3510,540){}}
}%
{\color[rgb]{0,0,0}\put(4051,-6541){\line( 0,-1){270}}
}%
{\color[rgb]{0,0,0}\put(5671,-6541){\line( 0,-1){270}}
}%
{\color[rgb]{0,0,0}\put(3781,-7351){\line( 0,-1){540}}
\put(3781,-7891){\line(-1, 0){1350}}
\put(2431,-7891){\line( 0, 1){4050}}
\put(2431,-3841){\line( 1, 0){1350}}
\put(3781,-3841){\line( 0,-1){540}}
}%
{\color[rgb]{0,0,0}\put(5671,-7351){\line( 0,-1){540}}
\put(5671,-7891){\line( 1, 0){1350}}
\put(7021,-7891){\line( 0, 1){4050}}
\put(7021,-3841){\line(-1, 0){1350}}
\put(5671,-3841){\line( 0,-1){540}}
}%
{\color[rgb]{0,0,0}\put(5671,-4921){\line( 0,-1){270}}
}%
{\color[rgb]{0,0,0}\put(3511,-5731){\dashbox{107}(540,540){}}
}%
{\color[rgb]{0,0,0}\put(2971,-4921){\dashbox{107}(3510,540){}}
}%
{\color[rgb]{0,0,0}\put(4321,-7351){\line( 0,-1){810}}
\put(4321,-8161){\line(-1, 0){2160}}
\put(2161,-8161){\line( 0, 1){4590}}
\put(2161,-3571){\line( 1, 0){2160}}
\put(4321,-3571){\line( 0,-1){810}}
}%
{\color[rgb]{0,0,0}\put(-2699,-4471){\line( 0,-1){1080}}
}%
{\color[rgb]{0,0,0}\put(-1619,-4471){\line( 0,-1){1080}}
}%
{\color[rgb]{0,0,0}\put(-2159,-4471){\line( 0,-1){270}}
}%
{\color[rgb]{0,0,0}\put(-2159,-5281){\line( 0,-1){270}}
}%
{\color[rgb]{0,0,0}\put(-269,-7711){\line( 0,-1){540}}
\put(-269,-8251){\line( 1, 0){1350}}
\put(1081,-8251){\line( 0, 1){4860}}
\put(1081,-3391){\line(-1, 0){1350}}
\put(-269,-3391){\line( 0,-1){540}}
}%
{\color[rgb]{0,0,0}\put(-2429,-5281){\dashbox{107}(540,540){}}
}%
{\color[rgb]{0,0,0}\put(-1619,-7711){\line( 0,-1){810}}
\put(-1619,-8521){\line(-1, 0){2160}}
\put(-3779,-8521){\line( 0, 1){5400}}
\put(-3779,-3121){\line( 1, 0){2160}}
\put(-1619,-3121){\line( 0,-1){810}}
}%
{\color[rgb]{0,0,0}\put(-2159,-7711){\line( 0,-1){540}}
\put(-2159,-8251){\line(-1, 0){1350}}
\put(-3509,-8251){\line( 0, 1){4860}}
\put(-3509,-3391){\line( 1, 0){1350}}
\put(-2159,-3391){\line( 0,-1){540}}
}%
{\color[rgb]{0,0,0}\put(-2969,-4471){\dashbox{107}(3510,540){}}
}%
{\color[rgb]{0,0,0}\put(-2969,-7711){\dashbox{107}(3510,540){}}
}%
{\color[rgb]{0,0,0}\put(-2969,-4471){\dashbox{107}(3510,540){}}
}%
{\color[rgb]{0,0,0}\put(-539,-6901){\dashbox{107}(540,540){}}
}%
{\color[rgb]{0,0,0}\put(-2969,-6091){\dashbox{107}(3510,540){}}
}%
{\color[rgb]{0,0,0}\put(-269,-6091){\line( 0,-1){270}}
}%
{\color[rgb]{0,0,0}\put(-2699,-6091){\line( 0,-1){1080}}
}%
{\color[rgb]{0,0,0}\put(-2159,-6091){\line( 0,-1){1080}}
}%
{\color[rgb]{0,0,0}\put(-1619,-6091){\line( 0,-1){1080}}
}%
{\color[rgb]{0,0,0}\put(271,-4471){\line( 0,-1){1080}}
}%
{\color[rgb]{0,0,0}\put(-269,-4471){\line( 0,-1){1080}}
}%
{\color[rgb]{0,0,0}\put(271,-6091){\line( 0,-1){1080}}
}%
{\color[rgb]{0,0,0}\put(-809,-6091){\line( 0,-1){1080}}
}%
{\color[rgb]{0,0,0}\put(-809,-4471){\line( 0,-1){1080}}
}%
{\color[rgb]{0,0,0}\put(-2969,-6091){\dashbox{107}(3510,540){}}
}%
{\color[rgb]{0,0,0}\put(-539,-6901){\dashbox{107}(540,540){}}
}%
{\color[rgb]{0,0,0}\put(-269,-6901){\line( 0,-1){270}}
}%
{\color[rgb]{0,0,0}\put(-2159,-11491){\line( 0,-1){540}}
\put(-2159,-12031){\line(-1, 0){1350}}
\put(-3509,-12031){\line( 0, 1){2430}}
\put(-3509,-9601){\line( 1, 0){1350}}
\put(-2159,-9601){\line( 0,-1){540}}
}%
{\color[rgb]{0,0,0}\put(-1619,-11491){\line( 0,-1){810}}
\put(-1619,-12301){\line(-1, 0){2160}}
\put(-3779,-12301){\line( 0, 1){2970}}
\put(-3779,-9331){\line( 1, 0){2160}}
\put(-1619,-9331){\line( 0,-1){1620}}
}%
{\color[rgb]{0,0,0}\put(-2159,-10681){\line( 0,-1){270}}
}%
{\color[rgb]{0,0,0}\put(-2429,-10681){\dashbox{107}(540,540){}}
}%
{\color[rgb]{0,0,0}\put(-2699,-11491){\line( 0,-1){270}}
\put(-2699,-11761){\line(-1, 0){540}}
\put(-3239,-11761){\line( 0, 1){1890}}
\put(-3239,-9871){\line( 1, 0){540}}
\put(-2699,-9871){\line( 0,-1){1080}}
}%
{\color[rgb]{0,0,0}\put(-1079,-11491){\line( 0,-1){810}}
\put(-1079,-12301){\line( 1, 0){2160}}
\put(1081,-12301){\line( 0, 1){2970}}
\put(1081,-9331){\line(-1, 0){2160}}
\put(-1079,-9331){\line( 0,-1){1620}}
}%
{\color[rgb]{0,0,0}\put(-539,-11491){\line( 0,-1){540}}
\put(-539,-12031){\line( 1, 0){1350}}
\put(811,-12031){\line( 0, 1){2430}}
\put(811,-9601){\line(-1, 0){1350}}
\put(-539,-9601){\line( 0,-1){540}}
}%
{\color[rgb]{0,0,0}\put(  1,-11491){\line( 0,-1){270}}
\put(  1,-11761){\line( 1, 0){540}}
\put(541,-11761){\line( 0, 1){1890}}
\put(541,-9871){\line(-1, 0){540}}
\put(  1,-9871){\line( 0,-1){1080}}
}%
{\color[rgb]{0,0,0}\put(-539,-10726){\line( 0,-1){270}}
}%
{\color[rgb]{0,0,0}\put(-809,-10681){\dashbox{107}(540,540){}}
}%
{\color[rgb]{0,0,0}\put(-2969,-11491){\dashbox{107}(3240,540){}}
}%
{\color[rgb]{0,0,0}\put(3511,-11491){\line( 0,-1){540}}
\put(3511,-12031){\line(-1, 0){1350}}
\put(2161,-12031){\line( 0, 1){2430}}
\put(2161,-9601){\line( 1, 0){1350}}
\put(3511,-9601){\line( 0,-1){540}}
}%
{\color[rgb]{0,0,0}\put(3511,-10681){\line( 0,-1){270}}
}%
{\color[rgb]{0,0,0}\put(3241,-10681){\dashbox{107}(540,540){}}
}%
{\color[rgb]{0,0,0}\put(2971,-11491){\line( 0,-1){270}}
\put(2971,-11761){\line(-1, 0){540}}
\put(2431,-11761){\line( 0, 1){1890}}
\put(2431,-9871){\line( 1, 0){540}}
\put(2971,-9871){\line( 0,-1){1080}}
}%
{\color[rgb]{0,0,0}\put(2701,-11491){\dashbox{107}(1350,540){}}
}%
{\color[rgb]{0,0,0}\put(5131,-11491){\line( 0,-1){540}}
\put(5131,-12031){\line( 1, 0){1350}}
\put(6481,-12031){\line( 0, 1){2430}}
\put(6481,-9601){\line(-1, 0){1350}}
\put(5131,-9601){\line( 0,-1){540}}
}%
{\color[rgb]{0,0,0}\put(5671,-11491){\line( 0,-1){270}}
\put(5671,-11761){\line( 1, 0){540}}
\put(6211,-11761){\line( 0, 1){1890}}
\put(6211,-9871){\line(-1, 0){540}}
\put(5671,-9871){\line( 0,-1){1080}}
}%
{\color[rgb]{0,0,0}\put(5131,-10681){\line( 0,-1){270}}
}%
{\color[rgb]{0,0,0}\put(4861,-10681){\dashbox{107}(540,540){}}
}%
{\color[rgb]{0,0,0}\put(4321,-11491){\dashbox{107}(1620,540){}}
}%
{\color[rgb]{0,0,0}\put(4591,-11491){\line( 0,-1){810}}
\put(4591,-12301){\line( 1, 0){2160}}
\put(6751,-12301){\line( 0, 1){2970}}
\put(6751,-9331){\line(-1, 0){2160}}
\put(4591,-9331){\line( 0,-1){1620}}
}%
\put(3781,-5596){\makebox(0,0)[b]{\smash{{\SetFigFont{10}{12.0}{\rmdefault}{\mddefault}{\updefault}{\color[rgb]{0,0,0}$X^*$}%
}}}}
\put(5491,-5596){\makebox(0,0)[b]{\smash{{\SetFigFont{10}{12.0}{\rmdefault}{\mddefault}{\updefault}{\color[rgb]{0,0,0}$P_{j-a}$}%
}}}}
\put(5671,-6406){\makebox(0,0)[b]{\smash{{\SetFigFont{10}{12.0}{\rmdefault}{\mddefault}{\updefault}{\color[rgb]{0,0,0}$X$}%
}}}}
\put(4051,-6406){\makebox(0,0)[b]{\smash{{\SetFigFont{10}{12.0}{\rmdefault}{\mddefault}{\updefault}{\color[rgb]{0,0,0}$P_p$}%
}}}}
\put(4681,-7216){\makebox(0,0)[b]{\smash{{\SetFigFont{10}{12.0}{\rmdefault}{\mddefault}{\updefault}{\color[rgb]{0,0,0}$P_n$}%
}}}}
\put(4771,-4786){\makebox(0,0)[b]{\smash{{\SetFigFont{10}{12.0}{\rmdefault}{\mddefault}{\updefault}{\color[rgb]{0,0,0}$P_n$}%
}}}}
\put(-2159,-5146){\makebox(0,0)[b]{\smash{{\SetFigFont{10}{12.0}{\rmdefault}{\mddefault}{\updefault}{\color[rgb]{0,0,0}$X^*$}%
}}}}
\put(-269,-6766){\makebox(0,0)[b]{\smash{{\SetFigFont{10}{12.0}{\rmdefault}{\mddefault}{\updefault}{\color[rgb]{0,0,0}$X$}%
}}}}
\put(-1169,-4336){\makebox(0,0)[b]{\smash{{\SetFigFont{10}{12.0}{\rmdefault}{\mddefault}{\updefault}{\color[rgb]{0,0,0}$P_n$}%
}}}}
\put(-1169,-7576){\makebox(0,0)[b]{\smash{{\SetFigFont{10}{12.0}{\rmdefault}{\mddefault}{\updefault}{\color[rgb]{0,0,0}$P_n$}%
}}}}
\put(-1169,-5956){\makebox(0,0)[b]{\smash{{\SetFigFont{10}{12.0}{\rmdefault}{\mddefault}{\updefault}{\color[rgb]{0,0,0}$P_n$}%
}}}}
\put(1621,-6001){\makebox(0,0)[b]{\smash{{\SetFigFont{10}{12.0}{\rmdefault}{\mddefault}{\updefault}{\color[rgb]{0,0,0}$\simeq$}%
}}}}
\put(-4139,-6001){\makebox(0,0)[rb]{\smash{{\SetFigFont{10}{12.0}{\rmdefault}{\mddefault}{\updefault}{\color[rgb]{0,0,0}$\Tr(X_{i,j}^*X_{p,q}) = $}%
}}}}
\put(-4139,-10951){\makebox(0,0)[rb]{\smash{{\SetFigFont{10}{12.0}{\rmdefault}{\mddefault}{\updefault}{\color[rgb]{0,0,0}$ = $}%
}}}}
\put(1621,-10951){\makebox(0,0)[b]{\smash{{\SetFigFont{10}{12.0}{\rmdefault}{\mddefault}{\updefault}{\color[rgb]{0,0,0}$\simeq$}%
}}}}
\put(7111,-10951){\makebox(0,0)[lb]{\smash{{\SetFigFont{10}{12.0}{\rmdefault}{\mddefault}{\updefault}{\color[rgb]{0,0,0}$ = 0 \times 0.$}%
}}}}
\put(-1259,-11356){\makebox(0,0)[b]{\smash{{\SetFigFont{10}{12.0}{\rmdefault}{\mddefault}{\updefault}{\color[rgb]{0,0,0}$P_n$}%
}}}}
\put(-2159,-10546){\makebox(0,0)[b]{\smash{{\SetFigFont{10}{12.0}{\rmdefault}{\mddefault}{\updefault}{\color[rgb]{0,0,0}$X^*$}%
}}}}
\put(-539,-10546){\makebox(0,0)[b]{\smash{{\SetFigFont{10}{12.0}{\rmdefault}{\mddefault}{\updefault}{\color[rgb]{0,0,0}$X$}%
}}}}
\put(3511,-10546){\makebox(0,0)[b]{\smash{{\SetFigFont{10}{12.0}{\rmdefault}{\mddefault}{\updefault}{\color[rgb]{0,0,0}$X^*$}%
}}}}
\put(5131,-10546){\makebox(0,0)[b]{\smash{{\SetFigFont{10}{12.0}{\rmdefault}{\mddefault}{\updefault}{\color[rgb]{0,0,0}$X$}%
}}}}
\put(3376,-11356){\makebox(0,0)[b]{\smash{{\SetFigFont{10}{12.0}{\rmdefault}{\mddefault}{\updefault}{\color[rgb]{0,0,0}$P_{i+k}$}%
}}}}
\put(5131,-11356){\makebox(0,0)[b]{\smash{{\SetFigFont{10}{12.0}{\rmdefault}{\mddefault}{\updefault}{\color[rgb]{0,0,0}$P_{j-a}$}%
}}}}
\end{picture}%
\end{center}

\medskip

\begin{proposition} \label{prop_gram_bounded_0} Fix $q\in\itv]0,1[$ and assume that $q\in\itv]0,{q_0}]$. There exists a constant $C>0$, depending on $k$ and
  $q_0$, such that $\|G_n\|\leq C \|x\|_2^2$ for all $n$. In particular $G(x)$ is bounded.
\end{proposition}

\begin{proof}
  Take the constants $\alpha$, $C$ provided by Lemma~\ref{lem_gram_decay}. Put $l = k+\lceil (2+k)/\alpha\rceil$, so that $\alpha k +2 + k\leq \alpha l$, and
  decompose $G_n = \hat G_n + \check G_n$, where $\hat G_{n;i,p} = \delta_{|i-p|\leq l} G_{n;i,p}$.
  From Lemma~\ref{lem_gram_decay} we have $|\check G_{n;i,p}| \leq Cq^{\alpha(|p-i|-l)}\|x\|_2^2$ and it is then a standard
  fact that $\check G$ is bounded. More precisely for any $\lambda\in\ell^2(\NN)$ we have by Cauchy-Schwarz
  \begin{align*}
    \left|{\ts\sum_{i,p}} \bar\lambda_i\lambda_p \check G_{n;i,p}\right|
    &\leq ({\ts\sum_{i,p}} |\lambda_i|^2 |\check G_{n;i,p}|)^{1/2}({\ts\sum_{i,p}} |\lambda_p|^2 |\check G_{n;i,p}|)^{1/2} \\
    &\leq C\|x\|_2^2~ {\ts\sum_i} |\lambda_i|^2 {\ts\sum_{|p-i|>l}} q^{\alpha(|p-i|-l)}
      \leq \frac{2q^\alpha C\|x\|_2^2}{1-q^\alpha} \|\lambda\|^2.
  \end{align*}
  This shows that $\|\check G_n\|\leq 2q_0^\alpha C\|x\|_2^2/(1-q_0^\alpha)$ for all $n$ and $q\in\itv]0,{q_0}]$.

  On the other hand by Lemma~\ref{lem_gram_entry_bound} we have $|G_{n;i,p}|\leq (1-q_0^2)^{-3}\|x\|_2^2$ for all $n$, $i$, $p$ and it follows easily
  $\|\hat G_n\|\leq (2l+1) (1-q_0^2)^{-3}\|x\|_2^2$.
\end{proof}

\bigskip

Now we want to prove that $G$ has a bounded inverse and obtain uniform estimates with respect to $k$. This requires a finer analysis of the band matrices
$\hat G_n$ of the previous proof. We first show that for $m<n$ the diagonal blocks of size $m-k+1$ of $G_n$ ``resemble'' $G_m$, with a better approximation
order for blocks that are far away from the ``borders'' of $G_n$. This will allow to reduce the analysis of $\hat G_n$ to that of a ``fixed size'' matrix
$G_m$ (in fact $m$ will depend on $k$, but not on $n$).

\begin{lemma}
  \label{lem_gram_block}
  Fix $q_0\in\itv]0,1[$ and assume that $q\in\itv]0,{q_0}]$. Assume that $n = m+a+b$ and $m = i+k+j = p+k+q$. Then there exists a constant $C$ depending only on $q_0$ such that
  \begin{displaymath}
    |G_{m;i,p}-G_{n;i+a,p+a}| \leq \left\{
      \begin{array}{ll}
        C \|x\|_2^2\, q^{\max(j,q)-k} & \text{if $a = 0$,} \\
        C \|x\|_2^2\, q^{\max(i,p)-k} & \text{if $b = 0$.}
      \end{array}\right.
  \end{displaymath}
  We also have $|G_{m;i,p}-G_{n;i+a,p+a}| \leq C \|x\|_2^2\, q^{\min(i,j,p,q)-k}$ for $a$, $b$ arbitrary.
\end{lemma}

\begin{proof}
  By definition we have $n = (i+a) + k + (j+b) = (p+a) + k + (q+b)$.
  The case $b=0$ follows from the case $a=0$ by symmetry. The ``general case'' follows from the first two cases by going first from $m$ to $n' = m+b$ and then
  from $n'$ to $n = n'+a$, and observing that $C q^{\max(j,q)} + C q^{\max(i,p)}\leq 2C q^{\min(i,j,p,q)}$. So we assume $a=0$.

  According to Lemma~\ref{lem_highest_weight} we have $\|P_n - (\id_{i+k}\otimes P_{j+b})(P_m\otimes \id_b)\| \leq C q^j$, which yields
  \begin{align*}
    (x_{n;i}\mid x_{n;p}) &= d_n^{-1} \Tr_n[P_n(\id_i\otimes X^*\otimes \id_{j+b})P_n
                            (\id_p\otimes X\otimes \id_{q+b})P_n] \\
                          &\simeq d_n^{-1} \Tr_n[P_n(\id_i\otimes X^*\otimes P_{j+b})
                            (P_m\otimes \id_b)(\id_p\otimes X\otimes \id_{q+b})P_n] \\
                          &= d_n^{-1}(\Tr_m\otimes\Tr_b)[P_n(\id_i\otimes X^*\otimes \id_{j+b})
                            (P_m\otimes \id_b)(\id_p\otimes X\otimes \id_{q+b})]
  \end{align*}
  up to $C q^j \|X\|^2 \leq Cq^j d_k\|x\|_2^2$, since $P_{j+b}$ is absorbed in $P_n$. Since by \cite[Proposition 1.13]{VaesVergnioux} we have $(\id\otimes\Tr_b)(P_n) = (d_n/d_m) P_m$, this reads
  \begin{displaymath}
    (x_{n;i}\mid x_{n;p}) \simeq d_m^{-1}\Tr_m[P_m(\id_i\otimes X^*\otimes \id_j)
    P_m(\id_p\otimes X\otimes \id_q)] = (x_{m;i}\mid x_{m;j})
  \end{displaymath}
  up to $Cq^j d_k\|x\|_2^2 \leq C q^{j-k} \|x\|_2^2 /(1-q_0)$. If $j\leq q$ we proceed in the same way starting with the estimate $P_n \simeq (P_m\otimes \id_b)(\id_{p+k}\otimes P_{q+b})$ up to
  $C q^q$.
\end{proof}

In the next Theorem we show that the blocks $G_n$ of the Gram matrix $G$ are related by a recursion formula, which allows at Lemma~\ref{lem_band_bound} to obtain
estimates on $G_m$ with a good behavior as $k\to\infty$, improving the ``naive'' Lemma~\ref{lem_gram_entry_bound}.

\begin{theorem}    
  \label{thm_gram_rec}
  Fix $n>k>0$ and $x = u_k(X) \in \W_k$ with $\rho(X) = \mu X$. For $0\leq i<n-k$ and
  $0\leq p\leq n-k$ we have:
  \begin{gather*}
    G_{n;i,p} = \delta_{p<n-k} (1-A^n_p) G_{n-1;i,p} + \delta_{p>0} B^n_p
    G_{n-1;i,p-1} + \delta_{p>1} C^n_p G_{n-1;i,p-2} \text{~~~where}\\
    \nonumber A^n_p = \frac{d_{p+k} d_{p+k-1}}{d_n d_{n-1}}, ~~ B^n_p = 2(-1)^k
    \Re(\mu) \frac{d_{p+k-1} d_{p-1}}{d_n d_{n-1}}, ~~ C^n_p = - \frac{d_{p-1}
      d_{p-2}}{d_n d_{n-1}}.
  \end{gather*}
\end{theorem}
Note that $A^n_p = 1$ if $p=n-k$, $B^n_p = 0$ if $p=0$ and $C^n_p = 0$ if
$p = 0$ or $1$, if one puts $d_{-l} = 0$ for $l>0$. Hence the corresponding
terms vanish ``naturally'' from the recursion equation.

The proof of the theorem will easily follow from the following Lemma, that we
will reuse in Section~\ref{sec_support_loc}, and which relies on two applications of
Wenzl's recursion relation~\eqref{eq_wenzl_recursion}.

\begin{lemma}\label{lem_general_rec}
  For $X\in B(H_k)^\ccirc$, $k\in\NN^*$, such that $\rho(X) = \mu X$, we have
  for all $p$, $q\in\NN$ and $n = p+k+q$:
  \begin{align*}
    \frac{d_{n-1}}{d_n} (\id\otimes\Tr_1)(X_{p,q})
    &= \delta_{q>0}(1-A^n_p)X_{p,q-1} + \delta_{p>0}B^n_p X_{p-1,q} 
      + \delta_{p>1} C^n_p X_{p-2,q+1}.
  \end{align*}
\end{lemma}

\begin{proof}
  \noindent {\bf Step 1.} In this proof we denote $T = X_{p,q}$.  We will use
  the Jones-Wenzl recursion formula for each projection $P_n$ appearing in the
  definition $T = P_n(\id_p\otimes X\otimes\id_q)P_n$, starting with the left
  occurrence. By the adjoint of~\eqref{eq_wenzl_recursion} we have
  $T = \sum_{l=1}^n (-1)^{n-l}(d_{l-1}/d_{n-1}) T_l$ where
  \begin{align*}
    T_n &= (P_{n-1}\otimes\id_1)(\id_p\otimes  X\otimes\id_q)P_n \qquad \text{and} \\
    T_l &= (P_{n-1}\otimes\id_1)(\id_{n-2}\otimes t)
        (\id_{l-1}\otimes t^*\otimes\id_{n-l-1})(\id_p\otimes X\otimes\id_q)P_n
          \qquad \text{for $l<n$.}
  \end{align*}

  \bigskip

  \noindent
  {\bf Step 2.} Denote $M = T_n$. Recall that
  $(\id\otimes\Tr_1)(P_n) = (d_n/d_{n-1}) P_{n-1}$, so that if $q\geq 1$ we have
  $(\id\otimes\Tr_1)(M) = (d_n/d_{n-1}) X_{p,q-1}$. If $q = 0$ we have to
  apply~\eqref{eq_wenzl_recursion} to the second occurrence of $P_n$. This yields
  $M = \sum_{l=1}^n (-1)^{n-l}$ $(d_{l-1}/d_{n-1}) M_l$ where
  \begin{align*}
    M_n &= (P_{n-1}\otimes\id_1)(\id_p\otimes X)(P_{n-1}\otimes\id_1) \qquad \text{and} \\
    M_l &= (P_{n-1}\otimes\id_1)(\id_p\otimes X)
          (\id_{l-1}\otimes t\otimes\id_{n-l-1})(\id_{n-2}\otimes t^*)(P_{n-1}\otimes\id_1)
          \qquad \text{for $l<n$.}    
  \end{align*}
  In $(\id\otimes\Tr_1)(M_n)$ we can factor $(\id\otimes\Tr_1)(X) = 0$
  so this term disappears. Moreover all terms $M_l$ vanish because $t$ hits $X = XP_k$ or
  $P_{n-1}$, except $M_p$. By the conjugate equation we have
  \begin{align*}
    (\id\otimes\Tr_1)(M_p)
    &= (\id_{n-1}\otimes t^*)(M_p\otimes\id_1)(\id_{n-1}\otimes t)  \\
    &= P_{n-1}(\id_{n-1}\otimes t^*)(\id_p\otimes X\otimes\id_1)
      (\id_{p-1}\otimes t\otimes\id_{n-p})P_{n-1}
  \end{align*}
  and we recognize
  $(\id\otimes\Tr_1)(M_p) = P_{n-1} (\id_{p-1}\otimes\rho(X)) P_{n-1}$.
  Altogether we can thus write
  $(\id\otimes\Tr_1)(M) = \delta_{p<n-k}(d_n/d_{n-1}) X_{p,q-1} +
  (-1)^k\delta_{p=n-k} \mu (d_{p-1}/d_{n-1}) X_{p-1,0}$.

  \bigskip

  \noindent
  {\bf Step 3.} Now we come back to the terms $T_l$ with $l<n$. Most of them vanish because
  $t^*$ hits either $X = P_k X$ or $P_n$. The only remaining terms are
  $M' := T_{p+k}$, which appears only if $p<n-k$ (i.e. $q\geq 1$), and $M'' := T_p$, which appears
  if $p\geq 1$. For these terms we apply as well~\eqref{eq_wenzl_recursion} to
  the second occurrence of $P_n$. This yields
  $M' = \sum_{l=1}^n (-1)^{n-l} (d_{l-1}/d_{n-1}) M'_l$ where
  \begin{align*}
    M'_n &= (P_{n-1}\otimes\id_1)(\id_{n-2}\otimes t)
           (\id_{p+k-1}\otimes t^*\otimes\id_{q-1})(\id_p\otimes X\otimes\id_q)
           (P_{n-1}\otimes\id_1) \quad \text{and} \\
    M'_l  &= (P_{n-1}\otimes\id_1)(\id_{n-2}\otimes t)
            (\id_{p+k-1}\otimes t^*\otimes\id_{q-1}) \\
         & \hspace{2cm}(\id_p\otimes X\otimes\id_q)
           (\id_{l-1}\otimes t\otimes\id_{n-l-1})(\id_{n-2}\otimes t^*)(P_{n-1}\otimes\id_1)
           \quad \text{for $l<n$.}
  \end{align*}
  One can simplify
  $(\id\otimes\Tr_1)(M'_n) = (\id_{n-1}\otimes
  t^*)(M'_n\otimes\id_1)(\id_{n-1}\otimes t)$ using the conjugate equation:
  \begin{displaymath}
    (\id\otimes\Tr_1)(M'_n) = P_{n-1}(\id_{p+k-1}\otimes t^*\otimes\id_q)(\id_p\otimes
    X\otimes\id_{q+1})(P_{n-1}\otimes t).
  \end{displaymath}
  This vanishes if $q\geq 2$ because in this case $t$ hits $P_{n-1}$. If $q=1$
  applying once again the conjugate equation we recognize
  $(\id\otimes\Tr_1)(M'_n) = P_{n-1}(\id_p\otimes
  X\otimes\id_{q-1})P_{n-1}$. Finally we have
  $(\id\otimes\Tr_1)(M'_n) = \delta_{p=n-k-1} X_{p,q-1}$.

  Again most of the terms $M'_l$ with $l<n$ vanish because the last $t$ hits
  either $X = X P_k$ or $P_{n-1}$. The first non-vanishing term, if $p\geq 1$,
  is $M'_p$ and we recognize
  $(\id_k\otimes t^*)(\id_1\otimes X\otimes\id_1)(t\otimes\id_k) = \rho(X) = \mu
  X$. By the conjugate equation we have
  $(\id\otimes\Tr_1)(L\otimes tt^*) = L\otimes\id_1$ so that
  \begin{align*}
    (\id\otimes\Tr_1)(M'_p) &= \mu (\id\otimes\Tr_1)[(P_{n-1}\otimes\id_1)(\id_{n-2}\otimes t) \\
    &\hspace{3cm}
         (\id_{p-1}\otimes X\otimes\id_{q-1})(\id_{n-2}\otimes t^*)(P_{n-1}\otimes\id_1)] \\
    &= \mu P_{n-1}(\id_{p-1}\otimes X\otimes\id_q)P_{n-1} = \mu X_{p-1,q}.
  \end{align*}

  The second non-vanishing term is $M'_{p+k}$, but it contains the term
  $(\id_{k-1}\otimes t^*){(X\otimes\id_1)}$
  ${(\id_{k-1}\otimes t)} = (\id\otimes\Tr_1)(X)$ hence it vanishes as
  well. Finally we have $M'_{p+k+1}$ which appears if $q\geq 2$ and by the
  conjugate equation can also be written
  \begin{displaymath}
    M'_{p+k+1} = (P_{n-1}\otimes\id_1)(\id_{n-2}\otimes t)
    (\id_p\otimes X\otimes\id_{q-2}) (\id_{n-2}\otimes t^*)(P_{n-1}\otimes\id_1). 
  \end{displaymath}
  As for $M'_p$ we have the further simplification
  $(\id\otimes\Tr_1)(M'_{p+k+1}) = P_{n-1}(\id_p\otimes X\otimes\id_{q-1})
  P_{n-1} = X_{p,q-1}$.

  \bigskip

  \noindent
  {\bf Step 4.} We proceed similarly for $M''$, writing
  $M'' = \sum_{l=1}^n (-1)^{n-l} (d_{l-1}/d_{n-1}) M''_l$ with
  \begin{align*}
    M''_n &= (P_{n-1}\otimes\id_1)(\id_{n-2}\otimes t)
            (\id_{p-1}\otimes t^*\otimes\id_{k+q-1})(\id_p\otimes X\otimes\id_q)
           (P_{n-1}\otimes\id_1) \quad \text{and} \\
    M''_l  &= (P_{n-1}\otimes\id_1)(\id_{n-2}\otimes t)
            (\id_{p-1}\otimes t^*\otimes\id_{k+q-1}) \\
         & \hspace{2cm}(\id_p\otimes X\otimes\id_q)
           (\id_{l-1}\otimes t\otimes\id_{n-l-1})(\id_{n-2}\otimes t^*)(P_{n-1}\otimes\id_1)
           \quad \text{for $l<n$.}
  \end{align*}
  As in the case of $M'_n$ we find
  \begin{displaymath}
    (\id\otimes\Tr_1)(M''_n) = P_{n-1}(\id_{p-1}\otimes t^*\otimes\id_{q+k})
    (\id_p\otimes X\otimes\id_{q+1})(P_{n-1}\otimes t),
  \end{displaymath}
  which vanishes as soon as $q\geq 1$ because $t$ then hits $P_{n-1}$. If $q = 0$ we recognize
  $(\id\otimes\Tr_1)$ $(M''_n) = P_{n-1}(\id_{p-1}\otimes\rho^*(X))P_{n-1}$. Altogether we have
  $(\id\otimes\Tr_1)(M''_n) = \delta_{p=n-k}\bar\mu X_{p-1,q}$.

  The first non-vanishing term $M''_l$ is $M''_{p-1}$, if $p\geq 2$, which by
  the conjugate equation reads
  $M''_{p-1} = (P_{n-1}\otimes\id_1)(\id_{n-2}\otimes t) (\id_{p-2}\otimes
  X\otimes\id_q) (\id_{n-2}\otimes t^*)(P_{n-1}\otimes\id_1)$. As for $M'$ it
  follows
  $(\id\otimes\Tr_1)(M''_{p-1}) = P_{n-1} (\id_{p-2}\otimes X\otimes\id_{q+1})
  P_{n-1} = X_{p-2,q+1}$. The second non-vanishing term would be $M''_p$ but it
  contains $(\Tr_1\otimes\id)(X)$ hence in fact it vanishes. The last term to
  consider is $M''_{p+k}$ which appears if $q = n-p-k > 0$ and we recognize
  \begin{displaymath}
    M''_{p+k} = (P_{n-1}\otimes\id_1)(\id_{n-2}\otimes t)(\id_{p-1}\otimes \rho^*(X)\otimes\id_{q-1})
                (\id_{n-2}\otimes t^*)(P_{n-1}\otimes\id_1)
  \end{displaymath}
  which yields as before $(\id\otimes\Tr_1)(M''_{p+k}) = \bar\mu X_{p-1,q}$.

  \bigskip

  \noindent {\bf Step 5.}
  Finally we can collect all terms as follows:
  \begin{align*}
    (\id\otimes\Tr_1)(T)
    &= (\id\otimes\Tr_1)\left[ M +
      (-1)^{n-p-k}\delta_{p<n-k}\frac{d_{p+k-1}}{d_{n-1}}M'_n +
      (-1)^k\delta_{n-k>p\geq 1}\frac{d_{p+k-1}d_{p-1}}{d_{n-1}^2}M'_p \right. \\
    &\hspace{2.2cm}  - \delta_{p<n-k-1}\frac{d_{p+k-1}d_{p+k}}{d_{n-1}^2}M'_{p+k+1} +
      (-1)^{n-p}\delta_{p\geq 1}\frac{d_{p-1}}{d_{n-1}}M''_n \\ 
    &\hspace{2.2cm} \left. - \delta_{p\geq 2}\frac{d_{p-1}d_{p-2}}{d_{n-1}^2}M''_{p-1} +
      (-1)^k\delta_{n-k>p\geq 1}\frac{d_{p-1}d_{p+k-1}}{d_{n-1}^2}M''_{p+k} \right].
  \end{align*}
  According to the computations carried above we obtain:
  \begin{align*}
    d_{n-1}^2 (\id\otimes\Tr_1)(T)
    &= \delta_{p<n-k}d_nd_{n-1}X_{p,q-1} +
      (-1)^k\delta_{p=n-k}\mu d_{p+k-1}d_{p-1} X_{p-1,q} \\
    &\hspace{0.4cm} - \delta_{p=n-k-1}d_{p+k-1}d_{p+k} X_{p,q-1}
      + (-1)^k\delta_{n-k>p\geq 1}\mu d_{p+k-1}d_{p-1} X_{p-1,q} \\
    &\hspace{0.4cm}  - \delta_{p<n-k-1} d_{p+k-1}d_{p+k} X_{p,q-1} +
      (-1)^k\delta_{p=n-k}\bar\mu d_{p-1}d_{p+k-1} X_{p-1,q} \\ 
    &\hspace{0.4cm} - \delta_{p\geq 2} d_{p-1}d_{p-2} X_{p-2,q+1} +
      (-1)^k\delta_{n-k>p\geq 1}\bar\mu d_{p-1}d_{p+k-1} X_{p-1,q}.
  \end{align*}
  Merging cases together as appropriate this yields the expression in the statement.
\end{proof}

\begin{proof}[Proof of Theorem~\ref{thm_gram_rec}]
  Recall that by assumption $j\geq 1$, but we allow $q = n-p-k = 0$. Multiplying
  the outcome of Lemma~\ref{lem_general_rec} by
  $(\id_i\otimes X^*\otimes\id_{j-1})$ on the left, we obtain
  \begin{align*}
    &\frac{d_{n-1}}{d_n} (\id\otimes\Tr_1)((\id_i\otimes X^*\otimes\id_j)X_{p,q})
      = \delta_{q>0}(1-A^n_p)(\id_i\otimes X^*\otimes\id_{j-1})X_{p,q-1}+ \\
    &\hspace{3cm}+ \delta_{p>0}B^n_p (\id_i\otimes X^*\otimes\id_{j-1})X_{p-1,q} 
      + \delta_{p>1} C^n_p (\id_i\otimes X^*\otimes\id_{j-1})X_{p-2,q+1}.
  \end{align*}
  We apply $\Tr_1^{\otimes(n-1)}$ to this identity. Since $X_{p,q} = P_n X_{p,q} P_n$ we have e.g.
  \begin{displaymath}
    \Tr_1^{\otimes n}((\id_i\otimes X^*\otimes\id_j)X_{p,q}) = \Tr_1^{\otimes n}(X_{i,j}^*X_{p,q}) = \Tr_n(X_{i,j}^*X_{p,q}),
  \end{displaymath}
  hence we obtain
  \begin{align*}
    & d_n^{-1}\Tr_n(X_{i,j}^*X_{p,q})
      = \delta_{q>0}(1-A^n_p)d_{n-1}^{-1}\Tr_{n-1}(X_{i,j-1}^*X_{p,q-1})+ \\
    &\hspace{2cm}+ \delta_{p>0}B^n_p d_{n-1}^{-1}\Tr_{n-1}(X_{i,j-1}^*X_{p-1,q}) 
      + \delta_{p>1} C^n_p d_{n-1}^{-1}\Tr_{n-1}(X_{i,j-1}^*X_{p-2,q+1}).
  \end{align*}
  This corresponds to the identity in the statement by definition of the Gram matrix $G_n$, using the Peter-Weyl expression of the scalar product.
\end{proof}

\begin{lemma}
  \label{lem_band_bound}
  Fix $q_0 \in \itv]0,1[$ and assume that $q\in\itv]0,{q_0}]$. Then there exists
  a constant $C>0$, depending only on $q_0$, such that
  \begin{align*}
  &|G_{m;i,p}| \leq C (m-k+1)q^{k+1}\|x\|_2^2 \qquad \text{and} \\
  &G_{m;p,p}\geq (C^{-1} - C(m-k)q^{k+1})\|x\|_2^2
  \end{align*}
  if $x\in \W_k$ and
  $0\leq i\neq p\leq m-k$.
\end{lemma}

\begin{proof}
  Since $G_m$ is symmetric we can assume $i<p\leq m-k$. We have $|\Re(\mu)|\leq 1$
  and for $p\leq m-k$ Lemma~\ref{lem_dimensions} shows that we have
  $|B^m_p|\leq 2 q^{k+1}/(1-q^2)^2$, $|C^m_p|\leq q^{2(k+1)}/(1-q^2)^2$. Since
  moreover $A^m_p\in\itv[0,1]$, the recursion formula of
  Theorem~\ref{thm_gram_rec} and Lemma~\ref{lem_gram_entry_bound} imply
  \begin{displaymath}
    |G_{m;i,p}|\leq \delta_{p<m-k} |G_{m-1;i,p}| + 3 q^{k+1}(1-q^2)^{-5}\|x\|_2^2.
  \end{displaymath}
  We iterate this inequality $m-p-k+1$ times, until we reach $G_{p+k-1;i,p}$, in which case the first term disappears. This yields the first estimate with
  $C = 3/(1-q_0^2)^5$.

  For the second one, let us start with $G_{m;0,0}$. In the recursion relation
  of Theorem~\ref{thm_gram_rec} only the first term is non zero when $i = p =
  0$. By an easy induction we have thus
  \begin{displaymath}
    G_{m;0,0} = G_{k;0,0}\prod_{l=k+1}^m (1-A_0^l)
    = \|x\|_2^2 \prod_{l=k+1}^m \left (1-\frac{d_kd_{k-1}}{d_l d_{l-1}}\right).
  \end{displaymath}
  Using the explicit expression of the dimensions $d_i$ and the fact that $1-q^{2k} \leq 1-q^{2l}$ if $l\geq k$ we obtain the following lower bound, which
  depends only on $q_0$:
  \begin{align*}
    G_{m;0,0} &=  \|x\|_2^2  \prod_{l=k+1}^m \left (1-q^{2(l-k)}
      \frac{(1-q^{2k+2})(1-q^{2k})}{(1-q^{2l+2})(1-q^{2l})}\right) \\
      &\geq \|x\|_2^2\prod_{l=k+1}^\infty \left (1-q^{2(l-k)}\right)
      \geq \|x\|_2^2\prod_{i=1}^\infty (1-q_0^{2i}) \geq C^{-1}\|x\|_2^2,
  \end{align*}
  increasing $C$ if necessary. Since $\|x\|_2 = \|x^*\|_2$, the same estimate is true for $G_{m;m-k,m-k}$ by Lemma~\ref{lem_symmetries}.

  For the other diagonal terms we use again the recursion equation, which yields
  for $p<m-k$:
  \begin{align*}
    G_{m;p,p}\geq (1-A_p^m) G_{m-1;p,p} - 3 q^{k+1}(1-q^2)^{-5}\|x\|_2^2.
  \end{align*}
  Again we iterate until $m = p+k+1$, obtaining
  \begin{displaymath}
    G_{m;p,p}
    \geq G_{p+k;p,p} {\ts\prod_{l=p+k+1}^m} (1-A_p^l) - 3(m-p-k)q^{k+1}(1-q^2)^{-5}\|x\|_2^2.
  \end{displaymath}
  The coefficients $1-A_p^l$ do not appear in the second term since they are
  dominated by $1$. We have already proved above that
  $\prod_{l=p+k+1}^m (1-A_p^l) \geq C^{-1}$ (replace $k$ by $k+p$) and so we obtain
  $G_{m;p,p} \geq C^{-2}\|x\|_2^2 - C(m-k)q^{k+1} \|x\|_2^2$.
\end{proof}

\bigskip

The estimates we have obtained about the ``fixed size'' matrix $G_m$ will be sufficient for our purposes only in the $q\to 0$ limit. This corresponds to letting
$q+q^{-1} = N\to\infty$, and apparently we are thus varying the spaces $H_1$, $H_k$. However, let us note that the numbers
\begin{equation*}\label{eq_scalar_product}
  G_{n;i,p} = d_n^{-1} \Tr_1^{\otimes n}[P_n(\id_i\otimes X^*\otimes\id_j)
  P_n(\id_p\otimes X\otimes\id_q)]
\end{equation*}
do not really depend on the precise form of the matrix $X\in B(H_k)^\ccirc$, but only on $k$, $\|x\|_2$ and on the eigenvalue $\mu$ of the rotation operator
$\rho$ corresponding to $X$. Indeed, we can expand the projections $P_n$ into linear combination of Temperley-Lieb diagrams $\pi$, whose coefficients depend on
$n$, $\pi$ and the parameter $q$. Moreover, after this expansion the evaluation of
\begin{equation*}\label{eq_TL_trace}
  \Tr_1^{\otimes n}[T_\pi (\id_i\otimes X^*\otimes\id_j)T_{\pi'}(\id_p\otimes
  X\otimes\id_q)]
\end{equation*}
is given by the evaluation of a Temperley-Lieb tangle at $X^*$ and $X$. Non-vanishing terms necessarily correspond to tangles where strings cannot start and end
on the same internal box, and so they are of the form $\Tr_k[\rho^r(X)^*\rho^s(X)] = \mu^{s-k}\|X\|_2^2 = d_k \mu^{s-k}\|x\|_2^2$. As a result $G_{n;i,p}$ can as
well be considered as a function of $k$, $\mu$, $\|x\|_2$ and $q$. Then it makes sense to take a limit $q\to 0$, and this will allow to prove results ``for
large $N$''.

\medskip

\begin{theorem}
  ~ \label{thm_riesz}
  \begin{enumerate}
  \item For all $q_0 \in\itv]0,1[$ there exists $C>0$ such that, assuming $q\leq q_0$, we have $\|G_n(x)\|\leq C$ for all
    $x\in \W$ and all $n$.
  \item There exists $q_1\in\itv]0,1[$ and $D>0$ such that, assuming $q\leq q_1$, we have $\|G_n(x)^{-1}\|\leq D$ for all
    $x\in \W$ and all $n$.
  \end{enumerate}
\end{theorem}\noindent
This shows in particular that $\{ x_{i,j} \mid x\in \W, i, j\in\NN \}$ is a Riesz basis of $H^\ccirc$ if $q\leq q_1$, and that the map
$\Phi : \ell^2(\NN)\otimes H^\ccirc\otimes\ell^2(\NN) \to H^\circ$, $\delta_i\otimes x\otimes \delta_j \mapsto x_{i,j}$ from Remark~\ref{rk_isomorphism} is an
isomorphism.

\begin{proof}
  Fix $q_0 \in\itv]0,1[$ and assume $q<q_0$. In this proof $C$ denotes a ``generic constant'' depending on $q_0$, that we will only modify a finite number of
  times. We take the constants $C>0$ and $\alpha > 0$ of Lemma~\ref{lem_gram_decay} and we fix the ``cut-off width'' $l = k + \lceil(3+k)/\alpha\rceil$. We will
  distinguish three regimes for the coefficients of our Gram matrix: the diagonal entries, for which we have the trivial estimate of
  Lemma~\ref{lem_gram_entry_bound} and the lower bound of Lemma~\ref{lem_band_bound} ; the entries $G_{n;i,p}$ with $0<|i-p|<2l$ for which we have the uniform
  estimate of Lemma~\ref{lem_band_bound} with a good behavior as $k\to\infty$ ; and the entries such that $|i-p| \geq 2l$ for which we have the off-diagonal
  decay estimate of Lemma~\ref{lem_gram_decay} with a bad behavior as $k\to\infty$.

  Recall that Lemma~\ref{lem_gram_decay} shows that $|G_{n;i,p}|\leq C q^{\alpha(|p-i|-k)-2-k} \|x\|_2^2$ if $|p-i|\geq k$, which by definition of $l$ yields
  $|G_{n;i,p}|\leq C q^{1+\alpha(|p-i|-l)} \|x\|_2^2$. In particular for $|p-i| \geq 2l$ we obtain
  $|G_{n;i,p}|\leq C q^{1+\alpha|p-i|/2} \|x\|_2^2$.

  We then deal with the entries such that $0<|p-i|<2l$. First assuming $n> 2l+5k+1$, we approximate each such entry $G_{n;i,p}$ by a corresponding entry
  $G_{m;i-a,p-a}$ of the smaller matrix $G_m$ with $m-k = 2l + 4k + 1$, using Lemma~\ref{lem_gram_block}. Write $n = i+k+j = p+k+q = m+a+b$. If $i$, $j$, $p$,
  $q > 2k$ we can choose $a$, $b$ such that $i-a$, $p-a$, $j-b$, $q-b \geq 2k+1$ --- we can e.g.\ take $a = \min(i,p)-2k-1$, and since $|i-p|< 2l$ we have
   $i-a < 2l+2k+1$ hence $j-b = 2l+4k - (i - a) + 1 > 2k$ and, similarly, $q-b>2k$. We have then
  \begin{displaymath}
    |G_{n;i,p} - G_{m;i-a,p-a}| \leq C q^{1+k} \|x\|_2^2.
  \end{displaymath}
  If $i\leq 2k$ or $p\leq 2k$ we use the case $a=0$, we have then $j-b\geq 2k+1$ (resp. $q-b\geq 2k+1$) and the estimate still holds. Similarly if $j\leq 2k$ or $q\leq 2k$ we use
  the case $b=0$.
  
  Now if $i\neq p$ Lemma~\ref{lem_band_bound} shows that $|G_{m;i-a,p-a}| \leq C(m-k+1)q^{k+1} \|x\|_2^2$. Altogether we have obtained the
  estimate $|G_{n;i,p}| \leq C (2l+4k+3) q^{1+k}\|x\|_2^2$ if $0 < |p-i| < 2l$. It holds also if $n\leq 2l+5k+1$ by applying directly Lemma~\ref{lem_band_bound} with $m=n$.
  Observe moreover that $l\leq (1+\alpha^{-1})k + 1+3\alpha^{-1} \leq 6\alpha^{-1}k$.  In particular the sequence $v_k = (2l + 4k + 3) q_0^{k/2}$ is bounded, hence we
  can modify $C$ so that $|G_{n;i,p}| \leq C q^{1+k/2} \|x\|_2^2$ for $0<|p-i|<2l$. In that case we have $|p-i| < 12\alpha^{-1}k$, hence we have as
  well $|G_{n;i,p}| \leq C q^{1+\alpha |p-i|/24} \|x\|_2^2$. Merging this with the case $|p-i|\geq 2l$ we have finally
  $|G_{n;i,p}| \leq C q^{1+\alpha |p-i|/24} \|x\|_2^2$ for all $i\neq p$, where $C$ and $\alpha$ depend only on $q_0$.

  Decompose $G_n = \hat G_n + \check G_n$, where $\hat G_n$ is diagonal with the same diagonal entries as $G_n$. The previous estimate shows that $\check G_n$
  is bounded, more precisely for any $\lambda\in\ell^2(\NN)$ we have by Cauchy-Schwarz
  \begin{align*}
    \left|{\ts\sum_{i,p}} \bar\lambda_i\lambda_p \check G_{n;i,p}\right|
    &\leq ({\ts\sum_{i,p}} |\lambda_i|^2 |\check G_{n;i,p}|)^{1/2}({\ts\sum_{i,p}} |\lambda_p|^2 |\check G_{n;i,p}|)^{1/2} \\
    &\leq C q \|x\|_2^2~ {\ts\sum_i} |\lambda_i|^2 {\ts\sum_{|p-i|\geq 1}} q^{\alpha|p-i|/24}
      \leq \frac{2q^{1+\alpha/24} C\|x\|_2^2}{1-q^{\alpha/24}} \|\lambda\|^2.
  \end{align*}
  This shows that $\|\check G_n\|\leq Cq\|x\|_2^2$ for all $n$ and $x$, after dividing $C$ by $2/(1-q_0^{\alpha/24})$. On the other hand we also have
  $\|\hat G_n\|\leq \|x\|_2^2/(1-q_0^2)^3$ by Lemma~\ref{lem_gram_entry_bound} and the first assertion is proved.

  For the inverse of $G$, we need a lower bound on the diagonal entries. We proceed as above, approximating each coefficient $G_{n;p,p}$ by a diagonal
  coefficient $G_{m;p-a,p-a}$ of a smaller matrix $G_m$, with $m - k = 1 + 4k$, and either $a=0$, $b=0$, or $p - a = 2k+1 = q-b$. This yields
  $|G_{n;p,p} - G_{m;p-a,p-a}| \leq C q^{k+1} \|x\|_2^2$. Then we use the lower bound of Lemma~\ref{lem_band_bound}, obtaining
  \begin{displaymath}
    G_{n;p,p} \geq C^{-1}\|x\|_2^2 - C (m-k+1)q^{k+1} \|x\|_2^2
    \geq C^{-1}\|x\|_2^2 - C q (4k+2)q_0^k \|x\|_2^2.
  \end{displaymath}
  Again if $m = 1+5k \geq n$ we obtain this estimate directly from Lemma~\ref{lem_band_bound}, without using $G_m$.
  Since the sequence $v_k = (4k+2)q_0^k$ is bounded, we can modify $C$ so as to obtain $G_{n;p,p} \geq (C^{-1}-Cq)\|x\|_2^2$. For $q$ small enough,
  $C^{-1}-Cq > 0$ and this shows $\hat G_n^{-1}\leq (C^{-1}-Cq)^{-1}\|x\|_2^{-2}\I$.

  Now we write $G_n = (\I +  \check G_n\hat G_n^{-1})\hat G_n$. The estimates obtained above show that we have
  $\|\check G_n \hat G_n^{-1}\| \leq D(q):= Cq/(C^{-1}-Cq)$. For $q_1$ small enough and $q\leq q_1$ we have $D(q)\leq D(q_1) < 1$ so that $G_n$
  is invertible. Moreover we have
  \begin{equation}\label{eq_inverse_diag}
    G_n^{-1} = \hat G_n^{-1}{\ts \sum_{i=0}^\infty} (-1)^i [\check G_n\hat G_n^{-1}]^i
  \end{equation}
  so that $\|G_n^{-1}\|\leq (C^{-1}-Cq_1)^{-1}(1-D(q_1))^{-1}\|x\|_2^{-2}$.
\end{proof}

\begin{remark}
  Using the recursion relation of Theorem~\ref{thm_gram_rec} and the symmetry properties of $G_n$, one can compute $G_n$ by induction on $n$. Numerical
  experiments then show the existence, for all $q\in\itv]0,1[$, of a constant $C>0$ such that $\|G_n\|\leq C\|x\|_2^2$, $\|G_n^{-1}\|\leq C\|x\|_2^{-2}$ for all
  $n$, $k$, $x\in \W_k$. Thus our proof is far from optimal and we strongly believe that the results of Theorem~\ref{thm_riesz} hold for all $q\in\itv]0,1[$
  (with constants depending on $q$).
\end{remark}

\section{An Orthogonality Property}
\label{sec_orthogonality}

Recall from Sections~\ref{sec_bimod} and~\ref{sec_gram} that we have an isomorphism of normed spaces
$\Phi : \ell^2(\NN)\otimes H^\ccirc\otimes \ell^2(\NN) \to H^\circ$. In this section we shall establish a crucial asymptotic orthogonality property of the
following subspaces:
\begin{notation}\label{not_Vm}
  For every $m\in\NN$ we consider the following subspace of $H^\circ$:
  \begin{displaymath}
    V_m = \Phi\left(\ell^2(\NN_{\geq m})\otimes H^\ccirc\otimes \ell^2(\NN_{\geq
        m})\right) = \overline{\Span}\{x_{i,j} \mid x\in H^\ccirc, i, j\geq m\}. 
  \end{displaymath}
\end{notation}
In the rest of this section we will prove that for $y\in p_n(H^\circ) \subset M$
the scalar product $(\zeta y\mid y\zeta)$ becomes small, uniformly on unit
vectors $\zeta\in V_m$, as $m\to\infty$, cf. Theorem~\ref{thm_ortho_global}. We
start by computations in $\Corep(\FO_N)$ which culminate in the ``local estimate'' of
Theorem~\ref{thm_orth_XY_YX}. In these computations $x = u_k(X)$ is a fixed
element of $p_k H^\ccirc$, which is not assumed to be an eigenvector of the
rotation map $\rho$. We then assemble the pieces to come back to $H(k)$ and finally $H^\circ$.

Recall from Section~\ref{sec_prelim} that by Tannaka-Krein duality products
$x_{i,j}y$, $yx_{i,j}$ can be computed from the elements $X_{i,j}*_m Y$,
$Y*_m X_{i,j} \in B(H_m)$ if $x = u_k(X)$ and $y = u_n(Y)$. Recall also that we
use the Hilbert-Schmidt norm $\|X\|_2 := \Tr(X^*X)^{1/2}$ on $B(H_k)$. We have
$\|AXB\|_2 \leq \|A\|\|X\|_2\|B\|$, where $\|A\|$, $\|B\|$ are the operator
norms of $A$, $B\in B(H_k)$. This yields for instance the inequality
$\|X_{i,j}*_m Y\|_2 \leq d_a \|X_{i,j}\|_2 \|Y\|_2$, where $m = i+k+j+n-2a$, and
we recall moreover that $\|X_{i,j}\|_2 \leq \sqrt{d_id_j} \|X\|_2$, see e.g.\ the
proof of Lemma~\ref{lem_gram_entry_bound}.  We still make repeated use of
Lemma~\ref{lem_dimensions} which allows to replace $d_l$ with $q^{-l}$ up to
multiplicative constants.

\begin{lemma}
  \label{lem_orth_YX}
  Fix $k$, $n \in\NN$ and $X \in B(H_k)^\ccirc$, $Y\in B(H_n)^\circ$. Then for
  all $i\geq n$, $r\in\NN$, $j\geq 2r$, $m = n+i+k+j-2a$ with $0\leq a\leq n$,
  there exists $Z\in B(H_{m-2r})$ such that
  \begin{displaymath}
    \|Y *_m X_{i,j} - P_m(Z\otimes\id_{2r})P_m\|_2 \leq C d_a q^{i+k+j-a-2r} \sqrt{d_id_j}\|X\|_2 \|Y\|_2,
  \end{displaymath}
  where $C$ is a constant depending only on $q$, and $\|Z\|_2\leq d_a\sqrt{d_id_{j-2r}}\|X\|_2\|Y\|_2$.
\end{lemma}

\begin{proof}
  We have by definition
  \begin{align*}
    Y*_m X_{i,j} &= P_m (\id_{n-a}\otimes t_a^*\otimes\id_{i+k+j-a})(\id_n\otimes
                   P_{i+k+j})(Y\otimes\id_i\otimes X\otimes\id_j) \\
                 &\hspace{5cm} (\id_n\otimes
                   P_{i+k+j})(\id_{n-a}\otimes t_a\otimes\id_{i+k+j-a})P_m.
  \end{align*}
  We use the estimate from Lemma~\ref{lem_highest_weight} as follows:
  $P_{i+k+j} \simeq (\id_a\otimes P_{i+k+j-a})$ $(P_{i+k+j-2r}\otimes\id_{2r})$
  up to $C q^{i+k+j-a-2r}$ in operator norm. Since
  $(\id_{n-a}\otimes P_{i+k+j-a})$ is absorbed by $P_m$ we can write
  \begin{align*}
    Y*_m X_{i,j} &\simeq P_m (\id_{n-a}\otimes t_a^*\otimes\id_{i+k+j-a})(\id_n\otimes
                   P_{i+k+j-2r}\otimes\id_{2r})(Y\otimes\id_i\otimes X\otimes\id_j) \\
                 &\hspace{4.5cm} (\id_n\otimes
                   P_{i+k+j-2r}\otimes\id_{2r})(\id_{n-a}\otimes t_a\otimes\id_{i+k+j-a})P_m
  \end{align*}
  up to $2C \|t_a\|^2 q^{i+k+j-a-2r} \|Y\otimes\id_i\otimes X\otimes\id_j\|_2 = 2C d_a q^{i+k+j-a-2r} \sqrt{d_id_j}\|X\|_2 \|Y\|_2$ in HS norm.

  This yields the result with
  \begin{align*}
    Z &= P_{m-2r}(\id_{n-a}\otimes t_a^*\otimes\id_{i+k+j-a-2r})(\id_n\otimes
        P_{i+k+j-2r})(Y\otimes\id_i\otimes X\otimes\id_{j-2r})\\
      &\hspace{5cm} (\id_n\otimes
        P_{i+k+j-2r})(\id_{n-a}\otimes t_a^*\otimes\id_{i+k+j-a-2r})P_{m-2r} 
  \end{align*}
  which satisfies the right norm estimate. Note that we have $Z = Y*_{m-2r}X_{i,j-2r}$. 
\end{proof}

\begin{lemma}
  \label{lem_orth_XY}
  Fix $k$, $n \in\NN$ and $X \in B(H_k)^\ccirc$, $Y\in B(H_n)^\circ$. Then for
  all $i\geq n$, $p\in\NN$, $j \geq 2n+3p$, $m = n+i+k+j-2a$ with
  $0\leq a\leq n$, we have
  \begin{displaymath}
    \|(\id\otimes\Tr_{n+2p})(X_{i,j} *_m Y)\|_2\leq C q^{\alpha p}q^{-p}q^{-a}q^{-(i+j+n)/2} \|X\|_2 \|Y\|_2, 
  \end{displaymath}
  where $C > 0$, $\alpha\in\itv]0,1[$ are constants depending only on $q$.
\end{lemma}

\begin{proof}
  In this proof $C$ denotes a generic constant, depending only on $q$ and that we will modify only a finite number of times.
  
  We write $\Tr_{n+2p} = (\Tr_p\otimes\Tr_{p+a}$ $\otimes\Tr_{n-a})$ $(P_{n+2p} \,\cdot\, P_{n+2p})$. Applying this to $X_{i,j} *_m Y$, the projections
  $\id\otimes P_{n+2p}$ are absorbed in $P_m$:
  \begin{align*}
    (\id\otimes\Tr_{n+2p})(X_{i,j} *_m Y) &= (\id_{m-2p-n}\otimes \Tr_p\otimes\Tr_{p+a}\otimes\Tr_{n-a}) [\\
                                          & \qquad P_m (\id_{m-n+a}\otimes t_a^*\otimes\id_{n-a})(X_{i,j}\otimes Y)(\id_{m-n+a}\otimes t_a\otimes\id_{n-a})P_m]
  \end{align*}
  We shall proceed to three successive approximations to show that this quantity is almost zero.

  We first use the estimate
  $P_m \simeq (\id_{m-p-n}\otimes P_{p+n})(P_{m-n+a}\otimes \id_{n-a})$ up to
  $Cq^{p+a}$ in operator norm, from Lemma~\ref{lem_highest_weight}. The projections
  $P_{m-n+a}\otimes\id$ are absorbed by $X_{i,j}$ so that
  \begin{align*}
    &P_m (\id_{m-n+a}\otimes t_a^*\otimes\id_{n-a})(X_{i,j}\otimes Y)(\id_{m-n+a}\otimes t_a\otimes\id_{n-a})P_m \simeq \\
    &\hspace{3cm} \simeq (\id_{m-p-n}\otimes P_{p+n})
      (\id_{m-n+a}\otimes t_a^*\otimes \id_{n-a})\\
    &\hspace{5.5cm}(X_{i,j}\otimes Y) (\id_{m-n+a}\otimes t_a\otimes \id_{n-a}) (\id_{m-p-n}\otimes P_{p+n}),
  \end{align*}
  with an error controlled by $2Cq^{p+a} \|t_a\|^2 \|X_{i,j}\otimes Y\|_2 \leq 2Cq^{p+a} d_a \sqrt{d_id_j}\|X\|_2\| Y\|_2$ in Hilbert-Schmidt norm.
  Observing that $\Tr_p$ hits only $X_{i,j}$, we obtain
  \begin{align*}
    (\id\otimes \Tr_{2p+n})(X_{i,j} *_m Y) &\simeq (\id_{m-2p-n}\otimes\Tr_{p+a}\otimes \Tr_{n-a})[ (\id_{m-2p-n}\otimes P_{p+n})
    \\ & \hspace{-2cm} (\id_{m-n-p+a}\otimes t_a^*\otimes \id_{n-a})(Z\otimes Y) (\id_{m-n-p+a}\otimes t_a\otimes \id_{n-a}) (\id_{m-2p-n}\otimes P_{p+n}) ],
  \end{align*}
  where $Z = (\id_{m-2p-n}\otimes \Tr_p\otimes \id_{p+2a})(X_{i,j})$. We denote the right-hand side by $\Phi(Z)$, with $\Phi : B(H_{m-2p-n}\otimes H_{p+2a}) \to B(H_{m-2p-n})$.
  After
  applying the  trace $\Tr_p\otimes$ $\Tr_{p+a}\otimes\Tr_{n-a}$, see e.g.\ Lemma~\ref{lem_RD}, the error is controlled as follows:
  \begin{align*}
    \|(\id\otimes \Tr_{2p+n})(X_{i,j} *_m Y) - \Phi(Z)\|_2 &\leq 2Cq^{p+a} d_a \sqrt{d_pd_{p+a}d_{n-a}d_id_j} \|X\|_2\|Y\|_2 \\
    & \leq 2C q^{-(i+j+n)/2} \|X\|_2\|Y\|_2,
  \end{align*}
  up to dividing $C$ by the appropriate power of $1/(1-q^2)$, cf. Lemma~\ref{lem_dimensions}. This error is less that the upper bounded in the statement.

  Then we use the estimate $P_{i+j+k} \simeq (P_{i+k+p}\otimes\id_{j-p})(\id_{i+k}\otimes P_j)$, up to $C q^p$ in operator norm, in the expression of $Z$. We have by assumption
  $j\geq 3p+2a$, and in particular we can write
  \begin{align*}
    Z &\simeq (\id_{i+k+j-2p-2a}\otimes \Tr_p\otimes \id_{p+2a})[(P_{i+k+p}\otimes\id_{j-p})
        (\id_i\otimes X\otimes P_j)(P_{i+k+p}\otimes\id_{j-p})] \\
      &= (P_{i+k+p}\otimes\id_{j-2p})
        (\id_i\otimes X\otimes P'_j)(P_{i+k+p}\otimes\id_{j-2p}) =: Z'
  \end{align*}
  where $P'_j = (\id_{j-2p-2a}\otimes\Tr_p\otimes\id_{p+2a})(P_j) \in B(H_{j-2p-2a}\otimes H_{p+2a})$. The error in $Z$ is controlled in HS norm by
  $2Cq^p\sqrt{d_p} \|P_i\otimes X\otimes P_j\|_2 = 2Cq^p\sqrt{d_pd_id_j} \|X\|_2$, so that
  \begin{align*}
    \|\Phi(Z) - \Phi(Z')\|_2 &\leq 2Cq^pd_a\sqrt{d_p d_{p+a} d_{n-a}d_id_j} \|X\|_2\|Y\|_2 \\
    &\leq 2C q^{-a} q^{-(i+j+n)/2} \|X\|_2\|Y\|_2.
  \end{align*}
  Again this is better than the estimate we are trying to prove.

  Now Lemma~\ref{lem_trace_proj} shows that $P'_j \simeq \lambda (\id_{j-2p-2a}\otimes\id_{p+2a})$  in $B(H_{j-2p-2a}\otimes H_{p+2a})$, up to $Cq^{\alpha p}d_p$ in operator norm, for some constant
  $\lambda$ depending on all parameters (and $\alpha > 0$ depending only on $q$). In HS norm we can control this error by
  $Cq^{\alpha p}d_p\sqrt{d_{j-2p-2a}d_{p+2a}}$. This yields
  \begin{displaymath}
    Z'\simeq Z'' := \lambda [(P_{i+k+p}\otimes\id_{j-3p-2a})(\id_i\otimes X\otimes\id_{j-2p-2a})(P_{i+k+p}\otimes\id_{j-3p-2a})]\otimes\id_{p+2a},
  \end{displaymath}
  and we have the control
  \begin{align*}
    \|\Phi(Z') - \Phi(Z'')\|_2 &\leq Cq^{\alpha p}d_ad_p\sqrt{d_id_{j-2p-2a}d_{p+2a}d_{p+a}d_{n-a}}\|X\|_2\|Y\|_2 \\
    &\leq C q^{\alpha p} q^{-a} q^{-p} q^{-(i+j+n)/2}\|X\|_2\|Y\|_2,
  \end{align*}
  which corresponds to the estimate in the statement.

  We finally arrived at
  \begin{align*}
    \Phi(Z'') &= \lambda (P_{i+k+p}\otimes\id_{j-3p-2a})(\id_i\otimes X\otimes\id_{j-2p-2a})(P_{i+k+p}\otimes\id_{j-3p-2a}) \times \\
    & \hspace{1cm} \times
      (\Tr_{p+a}\otimes \Tr_{n-a})[(\id_{p+a}\otimes t^*_a\otimes\id_{n-a})(\id_{p+2a}\otimes Y) (\id_{p+a}\otimes t_a\otimes\id_{n-a}) P_{p+n}].
  \end{align*}
  We claim that the second line above vanishes. Indeed $(\Tr_{p+a}\otimes\id_{n-a})(P_{p+n})$ is a multiple of $\id_{n-a}$, since it is an
  intertwiner of $H_{n-a}$. We are then left with
  \begin{displaymath}
    \Tr_{n-a}[(t^*_a\otimes\id_{n-a})(\id_a\otimes Y)(t_a\otimes\id_{n-a})] = (\Tr_a\otimes\Tr_{n-a})(Y),
  \end{displaymath}
  which vanishes because $y\in p_n H^\circ$. Hence $\Phi(Z'') = 0$ and the result is proved.
\end{proof}

\begin{lemma}
  \label{lem_orth_scalar}
  For $r\leq m/2$, $Z \in B(H_{m-2r})$, $S = P_m(Z\otimes\id_{2r})P_m$ and $T \in B(H_m)$ we have
  \begin{displaymath}
    |(S\mid T)| \leq \sqrt{d_r}\|(\id\otimes\Tr_r)(T)\|_2 \|Z\|_2 + C \|Z\|_2 \|T\|_2, 
  \end{displaymath}
  for some constant $C$ depending only on $q$.
\end{lemma}

\begin{proof}
  Recall once again from Lemma~\ref{lem_highest_weight} that
  $P_m \simeq (P_{m-r}\otimes\id_r)(\id_{m-2r}\otimes P_{2r})$ up to $C q^r$,
  where $C$ is a constant depending only on $q$.  Since
  $T(\id_{m-2r}\otimes P_{2r}) = T = P_mT$ we have
  \begin{align*}
    (S\mid T) & = \Tr_m(P_m(Z^*\otimes\id_{2r})P_m T) \\
              &\simeq (\Tr_{m-2r}\otimes\Tr_{2r})((\id_{m-2r}\otimes P_{2r})(P_{m-r}\otimes\id_r)(Z^*\otimes \id_{2r}) T) \\
              &= (\Tr_{m-2r}\otimes\Tr_r\otimes\Tr_r)((P_{m-r}\otimes\id_r)(Z^*\otimes \id_r\otimes\id_r) T) \\
              &= \Tr_{m-r}[P_{m-r} (Z^*\otimes \id_r) P_{m-r} (\id\otimes\Tr_r)(T)].
  \end{align*}
  By Cauchy-Schwarz the last quantity is dominated by
  $ \sqrt{d_r} \|Z\|_2\|(\id\otimes\Tr_r)(T)\|_2$. Moreover the error term in
  the second line is similarly bounded by
  $C q^r \|Z^*\otimes\id_r\otimes\id_r\|_2 \|T\|_2 = Cq^r \sqrt{d_rd_r} \|Z\|_2$
  $\|T\|_2 \leq C'\|Z\|_2 \|T\|_2$.
\end{proof}

\begin{theorem}
  \label{thm_orth_XY_YX}
  Fix $k$, $k'$, $n \in\NN$ and $X \in B(H_k)^\ccirc$, $X' \in B(H_{k'})^\ccirc$, $Y\in B(H_n)^\circ$. Then for all $i$, $j$, $i'$, $j'\geq 10n$ and
  $m = n+i+k+j-2a = n+i'+k'+j'-2a'$ with $0\leq a, a'\leq n$, we have
  \begin{displaymath}
    |(X_{i,j}*_mY\mid Y*_mX'_{i',j'})| \leq C d_m (q^{\alpha(i'+j')}+q^{\alpha\min(j,j')}) q^{(k+k')/2}
    \|X\|_2 \|X'\|_2 \|Y\|_2^2,
  \end{displaymath}  
  where $\alpha > 0$ is a constant depending only on $q$, and $C$ is a constant depending on $q$ and $n$.
\end{theorem}

\begin{proof}
  We put $p = \lfloor \min(j,j')/10\rfloor - n$ and $r = n+2p$. Thanks to the assumption on $j$, $j'$ we have $m\geq 2r$. We first apply Lemma~\ref{lem_orth_YX}
  to find $Z\in B(H_{m-2r})$ such that $\|Y *_m X'_{i',j'} - S\|_2 \leq C q^{i'+k'+j'-a'-2r}d_{a'} \sqrt{d_{i'}d_{j'}}\|X'\|_2 \|Y\|_2$ with
  $S = P_m(Z\otimes\id_{2r})P_m$. The condition $j'\geq 2r$ is satisfied since $p\leq \frac 1{10}j'-n$. We have then
  \begin{displaymath}
    |(X_{i,j}*_mY\mid Y*_mX'_{i',j'})| \leq |(X_{i,j}*_mY\mid S)|
    + \|X_{i,j}*_m Y\|_2 \|Y *_m X'_{i',j'} - S\|_2.
  \end{displaymath}
  Note that $\|X_{i,j}*_m Y\|_2\leq d_a \sqrt{d_id_j} \|X\|_2 \|Y\|_2$, and since $2r\leq j'/2$ we have
  \begin{align*}
    \|X_{i,j}*_m Y\|_2 \|Y *_m X'_{i',j'} - S\|_2
    &\leq  C q^{i'+k'+j'-a'-2r}d_ad_{a'} \sqrt{d_id_jd_{i'}d_{j'}}\|X\|_2\|X'\|_2 \|Y\|_2^2 \\
    &\leq C q^{i'/2}q^{-(i+j)/2}q^{k'}q^{-a-2a'} \|X\|_2\|X'\|_2 \|Y\|_2^2 \\
    &\leq C_n d_m q^{i'+j'/2}q^{(k+3k')/2} \|X\|_2\|X'\|_2 \|Y\|_2^2,
  \end{align*}
  were $C_n$ is a constant depending on $n$ and $q$.  We apply then
  Lemma~\ref{lem_orth_scalar} to $T = X_{i,j}*_m Y$ and our $S$. This yields
  \begin{displaymath}
    |(X_{i,j}*_mY\mid S)| \leq \sqrt{d_r} \|(\id\otimes\Tr_r)(X_{i,j}*_m Y)\|_2 \|Z\|_2
    + C \|Z\|_2 \|X_{i,j}*_m Y\|_2.
  \end{displaymath}
  Lemma~\ref{lem_orth_YX} also provides a bound on $\|Z\|_2$, in particular
  the second term on the right-hand side above is bounded by
  \begin{align*}
    Cd_{a'}d_a\sqrt{d_id_jd_{i'}d_{j'-2r}}\|X\|_2\|X'\|_2\|Y\|_2^2
    &\leq C  q^{-a-a'} q^{-(i+j+i'+j')/2}q^r\|X\|_2\|X'\|_2\|Y\|_2^2 \\
    &\leq C'_nd_m q^{\frac 15 \min(j,j')}q^{(k+k')/2}\|X\|_2\|X'\|_2\|Y\|_2^2,
  \end{align*}
  since we have $r\geq \frac 15\min(j,j')-n-2$.

  We finally apply Lemma~\ref{lem_orth_XY}. Again the condition $j\geq 2n+3p$ is
  satisfied because $p\leq \frac1{10}j-n$. This yields constants $\alpha_0 \in \itv]0,1[$, $C>0$ depending only on $q$ such that
  \begin{align*}
    \sqrt{d_r} \|(\id\otimes\Tr_r)(T)\|_2 \|Z\|_2
    &\leq C \sqrt{d_r}q^{\alpha_0 p}q^{-p}q^{-a}q^{-(i+j+n)/2} d_{a'}\sqrt{d_{i'}d_{j'-2r}}\|X\|_2 \|X'\|_2 \|Y\|_2^2 \\
    &\leq C q^{\alpha_0 p}q^{-p+r/2}q^{-a-a'}q^{-(i+j+n)/2} q^{-(i'+j')/2}\|X\|_2 \|X'\|_2 \|Y\|_2^2 \\
    &\leq C''_n d_m q^{\alpha_0 p} q^{(k+k')/2} \|X\|_2 \|X'\|_2 \|Y\|_2^2.
  \end{align*}
  Since $p\geq \frac 1{10}\min(j,j')-n-1$, this yields the result, with $\alpha = \min(\alpha_0/10,1/5)$.
\end{proof}

\begin{corollary}
  \label{crl_orth_xy_yx}
  Fix $k$, $k'$, $n\in\NN$ and $x\in p_k H^\ccirc$, $x'\in p_{k'} H^\ccirc$, $y\in p_n H^\circ$. Then for $i$, $i'$, $j$, $j' \geq 10n$ we have
  \begin{displaymath}
    |(x_{i,j}y\mid yx'_{i',j'})| \leq C (q^{\alpha(i+j)} + q^{\alpha(i'+j')} + q^{\alpha \max(\min(i,i'),\min(j,j'))}) \|x\|_2 \|x'\|_2\|y\|_2^2, 
  \end{displaymath}
  where $\alpha > 0$ is a constant depending only on $q$, and $C$ is a constant depending on $q$ and $n$.
\end{corollary}

\begin{proof}
  We have $x = u_k(X)$, $x' = u_{k'}(X')$, $y = u_n(Y)$ with
  $X\in B(H_k)^\ccirc$, $X'\in B(H_{k'})^\ccirc$, $y\in B(H_n)^\circ$.  Recall
  from Remark~\ref{rk_shift_tannaka} that we have then
  $x_{i,j} = u_{i+k+j}(X_{i,j})$. Following the reminder in
  Section~\ref{sec_prelim} --- specifically Equation~\eqref{eq_prod_FO} and
  Notation~\ref{not_convol} --- we obtain
  $x_{i,j}y = \sum_{a=0}^n (\kappa_m^{i+k+j,n})^2 u_m(X_{i,j}*_mY)$,
  where $m = i+k+j+n-2a$ as usual. The same holds for $yx'_{i',j'}$, and 
  the Peter--Weyl--Woronowicz Equation~\eqref{eq_peter_weyl} yields
  \begin{displaymath}
    (x_{i,j}y\mid yx'_{i',j'}) = \sum_{a=0}^n \frac 1{d_m}
    \big(\kappa_m^{i+k+j,n}\kappa_m^{n,i'+k'+j'}\big)^2 (X_{i,j}*_mY \mid Y *_m X'_{i',j'}).
  \end{displaymath}
  According to Lemma~\ref{lem_kappa}, the constants $\kappa$ are uniformly
  bounded by a constant depending only on $q$. Applying
  Theorem~\ref{thm_orth_XY_YX} and noticing that
  $q^{k/2}\|X\|_2 = q^{k/2}\sqrt {d_k} \|x\|_2\leq C \|x\|_2$ we obtain
  \begin{displaymath}
    |(x_{i,j}y\mid yx'_{i',j'})| \leq C (q^{\alpha(i'+j')} + q^{\alpha \min(j,j')}) \|x\|_2 \|x'\|_2\|y\|_2^2,
  \end{displaymath}
  where $C$ is a constant depending only on $q$ and $n$.

  The estimate in the statement follows from this one by symmetry, by switching left and right in Lemmata~\ref{lem_orth_YX}, \ref{lem_orth_XY} and
  \ref{lem_orth_scalar}. More precisely, recall that the antipode $S$ is isometric on $\ell^2(\GGamma)$ in the Kac case, and observe that $S(\chi_i) = \chi_i$,
  so that $S(x_{i,j}) = S(x)_{j,i}$. Applying the first part of this proof we thus get
  \begin{align*}
    |(x_{i,j}y \mid yx'_{i',j'})| &= |(yx'_{i',j'}\mid x_{i,j}y)|
                                  = |(S(x')_{j',i'} S(y)\mid S(y) S(x)_{j,i})| \\
    &\leq C(q^{\alpha(i+j)} + q^{\alpha \min(i',i)}) \|x\|_2 \|x'\|_2\|y\|_2^2.
  \end{align*}
  Taking the best of this estimate and the previous one yields the result.
\end{proof}

To pass from the ``local'' result of Corollary~\ref{crl_orth_xy_yx} to the
``global'' results of Proposition~\ref{prp_orth_local} and
Theorem~\ref{thm_ortho_global} we will need to analyze the kernel appearing on
the right-hand side in Corollary~\ref{crl_orth_xy_yx}. We state separately the
following elementary lemma which will be useful for this purpose.

\begin{lemma}
  \label{lem_max_min_kernel}
  Let $A$, $B\in\ell^2(\NN\times\NN)$ and put $q_{p;i,k} = q^{\max(\min(i,k),\min(p-i,p-k))}$. Then there exists a constant $C>0$ depending only on $q$ such that
  \begin{displaymath}
    {\ts\sum_{i\geq 0}\sum_{k\geq 0}\sum_{p\geq i,k}} ~ q_{p;i,k} |A_{i,p-i} B_{k,p-k}| \leq C \|A\|_2 \|B\|_2.
  \end{displaymath}
\end{lemma}

\begin{proof}
  Denote $T = \{(i,k,p) \in\NN^3 \mid p\geq i, p\geq k\}$.
  We start by applying Cauchy-Schwarz:
  \begin{displaymath}
    \left( {\ts\sum_{T}}~ q_{p;i,k} |A_{i,p-i} B_{k,p-k}|\right)^2 \leq
    {\ts\sum_{T}}~ q_{p;i,k} |A_{i,p-i}|^2 ~\times~ {\ts\sum_{T}}~ q_{p;i,k} |B_{k,p-k}|^2.
  \end{displaymath}
  By the symmetry in $i$ and $k$ it suffices to prove that ${\ts\sum}_{T}$ $q_{p;i,k} |A_{i,p-i}|^2 \leq C \|A\|_2^2$, which we can also write
  $\sum_{i=0}^\infty\sum_{p=i}^\infty s_{p;i} |A_{i,p-i}|^2 \leq C \|A\|_2^2$ with $s_{p;i} := \sum_{k=0}^p q_{p;i,k}$.  This holds for all $A$ if and only if
  $s_{p;i}$ is bounded independently of $i$ and $p$. Since $q_{p;i,k} = q_{p;p-i,p-k}$ we have $s_{p;i} = s_{p;p-i}$, thus we can assume $0\leq i\leq p-i$. We
  write then
  \begin{align*}
    s_{p;i} = \sum_{k=0}^p  q_{p;i,k}
    &= \Big({\ts\sum_{k=0}^{i-1} + \sum_{k=i}^{p-i} + \sum_{k=p-i+1}^p}\Big) q_{p;i,k} \\
    &= {\ts\sum_{k=0}^{i-1}} q^{\max(k,p-i)} +
      \Big({\ts\sum_{k=i}^{p-i} + \sum_{k=p-i+1}^p}\Big) q^{\max(i,p-k)} \\
    &= iq^{p-i} + \Big({\ts\sum_{k=i}^{p-i}} q^{p-k}\Big) + i q^i \leq 2\sup_i(iq^i) + {\ts\frac 1{1-q}}. \qedhere
  \end{align*}
\end{proof}

Recall from Notation~\ref{not_basis} that for $w\in \W$, $k\in\NN^*$ we
denote $H(w)$ resp.\ $H(k)$ the closure of $AwA$ resp.\ $A\W_k A$  in $H^\circ$,
where $\W$ is our privileged basis of $H^\ccirc$. Recall from
Notation~\ref{not_gram} that we denote $G(w)$ the Gram matrix of the family of
vectors $w_{i,j}$, for $w\in \W$.

\begin{proposition}
  \label{prp_orth_local}
  Fix $k$, $k'$, $n\in\NN^*$ and $y\in p_n H^\circ$. Assume that we have a
  common upper bound $\|G(w)^{-1}\| \leq D \|w\|_2^{-2}$,
  $\|G(w)\| \leq D \|w\|_2^2$ for all $w\in \W_k\cup \W_{k'}$.  Then for any $m\geq 10n$ and  
  $\zeta \in V_m\cap H(k)$, $\zeta' \in V_m\cap H(k')$ we have
  $|(\zeta y\mid y\zeta')|\leq C D q^{\alpha (m-|k-k'|)} \|\zeta\| \|\zeta'\|$, where
  $\alpha > 0$ is a constant depending only on $q$, and $C$ is a constant
  depending on $q$, $n$ and $y$.
\end{proposition}

\begin{proof}
  By assumption the map $(w,i,j) \mapsto w_{ij}$ induces a bicontinuous
  isomorphism between $p_kH^\ccirc\otimes\ell^2(\NN\times\NN)$ and $H(k)$.  More
  precisely, since $AwA \bot Aw'A$ for $w\neq w'$ in $\W_k$, the Gram matrix
  $G(k)$ of $(w_{i,j})_{i,j,w}$, with $w\in \W_k$, $i$, $j\in\NN$, is block
  diagonal with $G(w)$, $w\in \W_k$, as diagonal blocks, and thus it is bounded
  with bounded inverse by hypothesis. We can in particular decompose
  $\zeta = \sum_{i,j}\x x{(i,j)}_{i,j}$ with $\x x{(i,j)}\in p_kH^\ccirc$ and,
  denoting $x = (\x x{(i,j)})_{i,j}$, we have
  $\|x\|_2^2 = \sum_{i,j} \|\x x{(i,j)}\|^2 \leq D \|\zeta\|^2$. Similarly we
  write $\zeta' = \sum_{i,j}\x{x'}{(i,j)}_{i,j}$ with
  $\x{x'}{(i,j)} \in p_{k'}H^\ccirc$ and $\|x'\|_2^2 \leq D \|\zeta'\|^2$. We
  have then by Corollary~\ref{crl_orth_xy_yx}:
  \begin{align} \nonumber
    |(\zeta y\mid y\zeta')|
    & \leq {\ts \sum_{i,j} \sum_{i',j'}} |(\x x{(i,j)}_{i,j}y\mid y\x{x'}{(i',j')}_{i',j'})| \\
    \label{eq_three_sums}
    &\leq C~ {\ts \sum_{i,j} \sum_{i',j'}} 
      (q^{\alpha(i+j)} + q^{\alpha(i'+j')} + q^{\alpha \max(\min(i,i'),\min(j,j'))})
      \|\x x{(i,j)}\| \|\x{x'}{(i',j')}\|
  \end{align}
  where $C$ depends on $q$, $n$ and $y$. Since $\zeta$, $\zeta'\in V_m$ we have
  $\x x{(i,j)} = \x{x'}{(i',j')} = 0$ unless $i$, $j$, $i'$, $j'\geq
  m$. Moreover the scalar product
  $(\x x{(i,j)}_{i,j}y\mid y\x{x'}{(i',j')}_{i',j'})$ vanishes unless
  $u_{i+k+j}\otimes u_n$ and $u_n\otimes u_{i'+k'+j'}$ have a common subobject,
  which entails $|i+k+j-i'-k'-j'|\leq 2n$. We remove from~\eqref{eq_three_sums}
  the terms that do not satisfy these conditions. Moreover we regroup the three
  powers of $q$ that appear in~\eqref{eq_three_sums} into three distinct sums
  $S_1$, $S_2$, $S_3$ over $i$, $i'$, $j$, $j'$. 
  
  We start with $S_3$. Denote $p = i+j-2m$, $p'=i'+j'-2m$, $l = p-p'$. This yields a bijection $(i,i',j,j') \mapsto (l,i,i',p)$ in $\ZZ^4$ and we shall compute
  $S_3$ by summing over $(l,i,i',p)$ in the appropriate subset of $\ZZ^4$. If $l\geq 0$, we put further $\ub i = i-m\in\NN$, $\ub i'=i'-m+l\in\NN_{\geq l}$. Note that
  $j-m=p-\ub i$ and $j'- m=p- \ub i'$, so that
  \begin{align*}
    \max(\min(i,i'),\min(j,j')) &= \max(\min(\ub i,\ub i'-l),\min(p-\ub i,p-\ub i')) + m \\
                                & \geq \max(\min(\ub i,\ub i'),\min(p-\ub i,p-\ub i')) + m - |l| \\
    \Rightarrow\quad q^{\alpha \max(\min(i,i'),\min(j,j'))} &\leq q^{\alpha (m - |l|)} q_{p;\ub i,\ub i'}^\alpha,
  \end{align*}
  using the notation of Lemma~\ref{lem_max_min_kernel}.  If $l<0$, we put rather $\ub i = i-m-l\in\NN_{\geq -l}$, $\ub i'=i'-m\in\NN$ so that $j-m=p'-\ub i$ and
  $j'- m=p'- \ub i'$, and we obtain:
  \begin{align*}
    q^{\alpha \max(\min(i,i'),\min(j,j'))} &\leq q^{\alpha (m - |l|)} q_{p';\ub i,\ub i'}^\alpha.
  \end{align*}
  The constraints $j$, $j'\geq m$ translate to $p\geq \ub i,\ub i'$ when $l\geq 0$, and to $p'\geq\ub i,\ub i'$ when $l<0$.
  Re-organizing $S_3$ we thus obtain
  \begin{align*}
    S_3 ~ \leq ~ &\sum_{l\geq 0} Cq^{\alpha (m-|l|)} ~{\sum_{\ub i=0}^\infty \sum_{\ub i'=l}^\infty \sum_{p\geq \ub i,\ub i'}}
         \|\x{x}{(\ub i+m,p-\ub i+m)}\| \|\x{x'}{(\ub i'+m-l    ,p-\ub i'+m)}\|
         ~ q_{p;\ub i,\ub i'}^\alpha\\
            +&\sum_{l< 0} Cq^{\alpha (m-|l|)} ~{\sum_{\ub i= -l}^\infty \sum_{\ub i'=0}^\infty \sum_{p'\geq \ub i,\ub i'}}
            \|\x{x}{(\ub i+m+l,p'-\ub i+m)}\| \|\x{x'}{(\ub i'+m    ,p'-\ub i'+m)}\|
            ~ q_{p';\ub i,\ub i'}^\alpha.
  \end{align*}
  
  For the terms $l\geq 0$ we apply Lemma~\ref{lem_max_min_kernel} with $A_{r,s} = \|\x{x}{(r+m,s+m)}\|$, $B_{r,s} = \|\x{x'}{(r+m-l,s+m)}\|$, which satisfy
  $\|A\|_2 = \|x\|_2$, $\|B\|_2 = \|x'\|_2$. By adding vanishing terms to the sum we can assume that the sum over $\ub i'$ starts at $\ub i'=0$ to apply this
  Lemma. We apply Lemma~\ref{lem_max_min_kernel} similarly to each term $l<0$.  Observe finally that $|l - (k'-k)| = |i+k+j-i'-k'-j'|\leq 2n$, so
  that $l$ takes at most $4n+1$ values and $|l|\leq |k-k'|+2n$. Lemma~\ref{lem_max_min_kernel} thus yields the following upper bound:
  \begin{align*}
    S_3&\leq CC'{(4n+1)}q^{\alpha (m-2n-|k-k'|)}\|x\|_2 \|x'\|_2 
    \leq C''D q^{\alpha(m-|k-k'|)} \|\zeta\| \|\zeta'\|
  \end{align*}
  with $C''$ depending on $q$, $n$ and $y$.

  The case of $S_1$ (and of $S_2$) is similar but the counterpart of
  Lemma~\ref{lem_max_min_kernel} is simpler. We put $\ub i = i-m$,
  $\ub j = j-m$, $\ub i'=i'-m$, $\ub j'=j'-m$. Observe that for non-vanishing
  terms in the sum we have
  $i+j -2m = \frac 12(p+p'+l) = \frac 12(\ub i+\ub j) + \frac 12(\ub i'+\ub
  j')+\frac 12 l$, and still $|l| \leq |k-k'|+2n$. This yields, using again
  Cauchy-Schwarz:
  \begin{align*}
    S_1& = C \sum_{\ub i,\ub j\geq 0}\sum_{\ub i',\ub j'\geq 0} q^{\alpha (2m + l/2)} 
    (q^{\frac\alpha 2(\ub i+\ub j)} \|\x x{(\ub i+m,\ub j+m)}\|)
     (q^{\frac\alpha 2(\ub i'+\ub j')}\|\x{x'}{(\ub i'+m,\ub j'+m)}\|)\\
    &\leq  C q^{\alpha (2m - \frac 12|k-k'|-n)} \sum_{\ub i,\ub j\geq 0}
             q^{\frac\alpha 2(\ub i+\ub j)} \|\x x{(\ub i+m,\ub j+m)}\|
             \sum_{\ub i',\ub j'\geq 0} q^{\frac\alpha 2(\ub i'+\ub j')}\|\x{x'}{(\ub i'+m,\ub j'+m)}\| \\
           &\leq  C q^{\alpha (2m - \frac 12|k-k'|-n)} \|x\|_2 \|x'\|_2 ~
             {\ts\sum_{\ub i,\ub j}} q^{\alpha(\ub i+\ub j)}
             \leq C'''Dq^{\alpha(m-|k-k'|)} \|\zeta\| \|\zeta'\|,
  \end{align*}
  with $C'''$ depending on $q$, $n$ and $y$.
\end{proof}

Taking into account the finite propagation result established at the end of
Section~\ref{sec_bimod} we can finally prove the following global estimate.

\begin{theorem}
  \label{thm_ortho_global}
  Fix $n\in\NN$ and $y\in p_n H^\circ \subset M$. Take the constant $q_1$ given by
  Theorem~\ref{thm_riesz} and assume $q\leq q_1$. Then for any $m\geq 10n$ and $\zeta \in V_m$
  we have $|(\zeta y\mid y\zeta)|\leq C q^{\alpha m} \|\zeta\|^2$,
  where $\alpha > 0$ is a constant depending only on $q$, and $C$ is a constant
  depending on $q$, $n$ and $y$.
\end{theorem}

\begin{proof}
  We have the orthogonal decomposition $\zeta = \sum_{k\in\NN^*} \zeta_k$ with $\zeta_k \in \overline{A\W_kA} = H(k)$. Proposition~\ref{prp_propagation} shows
  that $y\zeta_{k'}$ decomposes into subspaces $H(l)$ with $|k'-l|\leq n$, and similarly $\zeta_k y$ decomposes into subspaces $H(l)$ with $|k-l|\leq n$, so
  that we have $\zeta_k y\bot y\zeta_{k'}$ if $|k-k'|>2n$. Proposition~\ref{prp_orth_local} applies thank to Theorem~\ref{thm_riesz} and the assumption on
  $q$. Thus we can write, using Cauchy-Schwarz:
  \begin{align*}
    |(\zeta y\mid y\zeta)| &\leq {\ts\sum_{|k'-k|\leq 2n}} |(\zeta_k y\mid y\zeta_{k'})|
    \leq C q^{\alpha (m-2n)} {\ts\sum_{|k'-k|\leq 2n}} \|\zeta_k\| \|\zeta_{k'}\| \\
                           &\leq C q^{\alpha (m-2n)} {\ts\sum_{|k'-k|\leq 2n}} \|\zeta_k\|^2
                             \leq C q^{\alpha m}q^{-2\alpha n}(4n+1) \|\zeta\|^2.\qedhere
  \end{align*}
\end{proof}

\section{Support localization}
\label{sec_support_loc}

In this section we will show that for any $m\in\NN$, elements $z\in A^\bot\cap M$ which almost commute to the generator $\chi_1\in A$ have a small component in
the subspace spanned by vectors $w_{i,j}$ with $i \leq m$ or $j\leq m$, and from this we deduce the Asymptotic Orthogonality Property for the MASA $A\subset M$.

Our strategy starts with algebraic arguments, using the constant structures for the left multiplication by $\chi_1\in A$ on the basis $(w_{i,j})$, obtained at
Proposition~\ref{prop_struct_const}, to deduce relations between a vector $z\in H(w)$ and its commutator $[\chi_1,z]$, cf.\ Proposition~\ref{prop_commut_relations}. The main analytical input is then an estimate on coefficients appearing in these relations that we establish at
Lemma~\ref{lem_technical_estimate}, and which allows to establish the main result of this section, Theorem~\ref{thm_support_estimate}. We can then prove
Theorem~\ref{thm_main} by assembling the results of the article, following Popa's classical strategy.

%

\bigskip

The following computation of the structure constants is mainly a reformulation of Lemma~\ref{lem_general_rec} that we already used for the study of the Gram
matrix. Recall Notation~\ref{not_basis} for the orthonormal family $W = \bigsqcup_{k\geq 1} W_k$ which spans the $A,A$-bimodule $H^\circ$. For $w\in W$ we have the
associated vectors $w_{i,j}\in AwA$, where $i$, $j\in\NN$. We agree to denote moreover $w_{i,j} = 0$ if $i<0$ or $j<0$. Finally, let us recall the definition of
the coefficients $A$, $B$, $C$ from the statement of Theorem~\ref{thm_gram_rec}:
\begin{displaymath}
  A^n_p = \frac{d_{p+k} d_{p+k-1}}{d_n d_{n-1}}, \quad
  B^n_p = 2(-1)^k \Re(\mu) \frac{d_{p+k-1} d_{p-1}}{d_n d_{n-1}}, \quad
  C^n_p = - \frac{d_{p-1} d_{p-2}}{d_n d_{n-1}}.
\end{displaymath}
They depend on $p\in\NN$ and $n\in\NN^*$, but also on $k\in\NN$ and $\mu\in\CC$ which will be fixed most of the time. Recall moreover that we are using the
convention $d_p = 0$ for $p<0$.

\begin{proposition}\label{prop_struct_const}
  Let $w\in \W_k$ with associated eigenvalue $\mu$ of the rotation map, and consider the associated coefficients $A$, $B$, $C$. Then for any $i$, $j\in\NN$ we
  have, with $n = i + k + j$:
  \begin{displaymath}
  \chi_1 w_{i,j} = w_{i+1,j} + (1- A^n_j)w_{i-1,j} + B^n_jw_{i,j-1}+C^n_j w_{i+1,j-2}.
\end{displaymath}
\end{proposition}

\begin{proof}
  According to the fusion rules we have $\chi_1 w_{i,j} = p_{n+1}(\chi_1 w_{i,j}) + p_{n-1}(\chi_1 w_{i,j})$, and moreover
  $p_{n+1}(\chi_1 w_{i,j}) = p_{n+1}(\chi_1\chi_iw\chi_j) = w_{i+1,j}$ because $p_{n+1}(\chi_1 p_l(\chi_iw\chi_j)) = 0$ if $l<n$. We compute the second term
  $p_{n-1}(\chi_1 w_{i,j})$ in the Tannaka-Krein picture: putting $w = u_k(X)$ with $X\in B(H_k)^\ccirc$, we have by~\eqref{eq_prod_FO}:
  \begin{displaymath}
    p_{n-1}(\chi_1 w_{i,j}) = (\kappa_{n-1}^{1,n})^2~ u_{n-1}(\id_1*_{n-1}X_{i,j}).
  \end{displaymath}
  Recall the basic intertwiner $V_{n-1}^{1,n} = (P_1\otimes P_n)(t\otimes\id_{n-1})P_{n-1}$. We have
  $V_{n-1}^{1,n*}V_{n-1}^{1,n} = {(t^*\otimes\id_{n-1})}$ $(\id_1\otimes P_n)(t\otimes \id_{n-1}) = (\Tr_1\otimes\id)(P_n) = (d_n/d_{n-1})\id_{n-1}$, so that
  $(\kappa_{n-1}^{1,n})^2 = d_{n-1}/d_n$. Moreover we have by definition
  \begin{displaymath}
    \id_1*_{n-1} X_{i,j} = (t^*\otimes\id_{n-1})(\id_1\otimes X_{i,j})(t\otimes\id_{n-1})
    = (\Tr_1\otimes\id)(X_{i,j}).
  \end{displaymath}
  Switching left and right in Lemma~\ref{lem_general_rec} (or applying the
  antipode as in the proof of Corollary~\ref{crl_orth_xy_yx}) we thus obtain
  \begin{align*}
    (\kappa_{n-1}^{1,n})^2 (\id_1*_{n-1} X_{i,j}) 
    = \delta_{i>0}(1-A^n_j)X_{i-1,j} + \delta_{j>0}B^n_j X_{i,j-1} 
    + \delta_{j>1} C^n_j X_{i+1,j-2}.
  \end{align*}
  This yields the formula in the statement.
\end{proof}

\begin{corollary}\label{cor_commutator}
  Fix $w\in W_k$, assume that $\{w_{i,j} \mid i, j\in\NN\}$ is a Riesz basis, and take an element $z = \sum_{i,j} z_{i,j}w_{i,j}$ in $H(w)$. We put moreover
  $z_{i,j} = 0$ if $i<0$ or $j<0$. Writing similarly $[\chi_1,z] = \sum_{i,j} [\chi_1,z]_{i,j} w_{i,j}$ in $H(w)$, we have
  \begin{displaymath}
    [\chi_1,z]_{i,j} = z_{i-1,j} - z_{i,j-1} + D^{n+1}_j z_{i+1,j} - D^{n+1}_i z_{i,j+1}
    + C^{n+1}_{j+2}z_{i-1,j+2} - C^{n+1}_{i+2} z_{i+2,j-1},
  \end{displaymath}
  where we take $n = i+k+j$ and denote $D^n_j = 1 - A^n_j-B^n_{n-k-j}$. 
\end{corollary}

\begin{proof}
  The proposition gives, by summing in $H(w)$ over $i$, $j\in\NN$:
  \begin{align*}
    \chi_1 z &= \ts\sum_{i,j} z_{i,j} w_{i+1,j} + \sum_{i\geq 1,j}(1- A^{i+k+j}_{j})z_{i,j}w_{i-1,j} \\
    &\hspace{2cm} + \ts\sum_{j\geq 1,i}B^{i+k+j}_{j}z_{i,j}w_{i,j-1} +  \sum_{j\geq 2,i} C^{i+k+j}_{j} z_{i,j} w_{i+1,j-2} \\
             &= \ts\sum_{i\geq 1,j} z_{i-1,j} w_{i,j} + \sum_{i,j}(1- A^{i+k+j+1}_{j})z_{i+1,j}w_{i,j} \\
    &\hspace{2cm} + \ts\sum_{i,j}B^{i+k+j+1}_{j+1}z_{i,j+1}w_{i,j} +  \sum_{i\geq 1,j} C^{i+k+j+1}_{j+2} z_{i-1,j+2} w_{i,j}.
  \end{align*}
  With our convention we can add the terms $i=0$ in the first and last sum, and
  for fixed $i$, $j\in\NN$ this yields
  $(\chi_1 z)_{i,j} = z_{i-1,j} + (1- A^{n+1}_{j})z_{i+1,j} + B^{n+1}_{j+1}z_{i,j+1}
  + C^{n+1}_{j+2} z_{i-1,j+2}$, where $n = i+k+j$. We also have
  $(z\chi_1 )_{j,i} = z_{j,i-1} + (1- A^{n+1}_{j})z_{j,i+1} + B^{n+1}_{j+1}z_{j+1,i}
  + C^{n+1}_{j+2} z_{j+2,i-1}$ by symmetry (or by applying the
  antipode). Switching $i$ and $j$ this reads
  $(z\chi_1 )_{i,j} = z_{i,j-1} + (1 - A^{n+1}_{i})z_{i,j+1} + B^{n+1}_{i+1}z_{i+1,j}
  + C^{n+1}_{i+2} z_{i+2,j-1}$ and a substraction yields the result, since $D_j^{n+1} = 1 - A^{n+1}_j - B^{n+1}_{i+1}$.
\end{proof}

Iterating Corollary~\ref{cor_commutator}, we shall obtain more relations between a vector $z$ and the commutator $[\chi_1,z]$: cf
Equation~\eqref{eq_iterated_move} below where the case $p=1$ corresponds in fact to Corollary~\ref{cor_commutator}. For fixed $m$, $l$, the {\em collection} of
relations~\eqref{eq_iterated_move} for varying $p$ will yield the crucial estimate of Theorem~\ref{thm_support_estimate}. The coefficients appearing
in~\eqref{eq_iterated_move} are introduced inductively as follows.

\begin{notation}\label{not_recursive_coeffs}
  We fix $k\in\NN^*$, $|\mu| = 1$, and $m\in\NN$.  We define families of
  coefficients $f_{i,j}^{l,p}$, $g_{i,j}^{l,p}$ for $i$, $j$, $l$, $p\in\NN$,
  and $\phi^{l,p}_i$ for $l$, $p\in\NN$, $i\in\ZZ$, by induction on $p$, as
  follows. For $p = 0$ we first put $f^{l,0}_{i,j} = \delta_{(i,j) = (m,l)}$ and
  $g^{l,0}_{i,j} = 0$ for all $l$, $i$, $j\in\NN$. Then assuming that
  $f^{l,p}_{i,j}$, $g^{l,p}_{i,j}$ are constructed for a given $p$ and all $l$,
  $i$, $j\in\NN$ we first put
  \begin{displaymath}
    \phi_i^{l,p} = -\sum_{s=-p}^i f^{l,p}_{m+p-s,l+p+s}
  \end{displaymath}
  for $-p\leq i\leq m+p$ and $\phi_i^{l,p} = 0$ for the other values of $i\in\ZZ$. Then we define $g_{i,j}^{l,p+1} = g_{i,j}^{l,p}$ if $i+j\leq m+l+2p-1$,
  $g_{i,j}^{l,p+1} = \phi_{m+p-i}^{l,p}$ if $i+j = m+l+2p+1$ and $g_{i,j}^{l,p+1} = 0$ else. Finally we put $f_{i,j}^{l,p+1} = 0$ if $i+j \neq m+l+2p+2$, and
  \begin{displaymath}
    f_{i,j}^{l,p+1} = D^n_{i} \phi^{l,p}_{m+p-i}  - D^n_{j} \phi^{l,p}_{m+p-i+1}
    + C^n_{i} \phi^{l,p}_{m+p-i+2} - C^n_{j}\phi^{l,p}_{m+p-i-1}
  \end{displaymath}
  if $i+j = m+l+2p+2$, with $n = i+k+j$.
\end{notation}

\begin{remark}
  The last relation in fact implies also $f_{i,j}^{l,p+1} = 0$ if $i+j = m+l+2p+2$ and $i>m+2p+2$, because then $m+p-i < -p-2$. Hence $f^{l,p}_{i,j} = 0$ unless
  $i+j = m+l+2p$ and $i\leq m+2p$. By definition, one can recover the coefficients $f$ from $\phi$ as follows:
  $f_{i,m+l+2p-i}^{l,p} = \phi^{l,p}_{m+p-i-1} - \phi^{l,p}_{m+p-i}$ for $i = 0, \ldots, m+2p$ (which for $i=m+2p$ also reads
  $f_{m+2p,l}^{l,p} = - \phi^{l,p}_{-p}$), and $f_{i,j}^{l,p} = 0$ if $i+j \neq m+l+2p$ or $i> m+2p$. Also, we have $\phi^{l,0}_i = -1$ for $0\leq i\leq m$ and $0$ otherwise. On the
  other hand, we also record the fact that $g_{i,j}^{l,p} = 0$ unless $m+l+1\leq i+j \leq m+l+2p-1$ and $i+j$, $m+l+1$ have the same parity.
\end{remark}

\begin{proposition}\label{prop_commut_relations}
  Fix $w \in W_k$ with associated $\rho$-eigenvalue $\mu$, and $m\in\NN$. Assume that $\{w_{i,j} \mid i, j\in\NN\}$ is a Riesz basis. For any
  $z = \sum_{i,j} z_{i,j} w_{i,j} \in H(w)$ we have, using the coefficients of Notation~\ref{not_recursive_coeffs}:
  \begin{equation}\label{eq_iterated_move}
    \forall p\in\NN, l\in\NN \quad z_{m,l} = \sum_{i,j\in\NN} f_{i,j}^{l,p} z_{i,j}
    + \sum_{i,j\in\NN} g_{i,j}^{l,p} [\chi_1,z]_{i,j}.
  \end{equation}
\end{proposition}

\begin{proof}
  We proceed by induction over $p$, noting that~\eqref{eq_iterated_move} holds trivially for $p=0$. Assume now that it is satisfied for some fixed $p\in\NN$.
  Using Corollary~\ref{cor_commutator} we then write, with the convention $i' = m+l+2p-i$ in each term of the sums:
  \begin{align*}
    z_{m,l} = &\ts\sum_{i=0}^{m+2p} f_{i,i'}^{l,p} z_{i,i'} -\sum_{i=0}^{m+2p}\phi^{l,p}_{m+p-i}(z_{i-1,i'+1} - z_{i,i'}) \\
              &\hspace{1cm}+\ts\sum_{i,j} g_{i,j}^{l,p} [\chi_1,z]_{i,j} + \sum_{i=0}^{m+2p}\phi^{l,p}_{m+p-i}[\chi_1,z]_{i,i'+1}  \\
              &\hspace{1cm}-\ts\sum_{i=0}^{m+2p}\phi^{l,p}_{m+p-i}(D^n_{i'+1} z_{i+1,i'+1} - D^n_{i} z_{i,i'+2}  + C^n_{i'+3}z_{i-1,i'+3} - C^n_{i+2} z_{i+2,i'}),
  \end{align*}
  where $n = k+m+l+2p+2 = i+k+i'+2$.  By definition of the coefficients $\phi$, the first two sums cancel each other: indeed $z_{i,i'}$, for $0\leq i\leq m+2p$,
  appears with the factor $\phi^{l,p}_{m+p-i-1} - \phi^{l,p}_{m+p-i} = f^{l,p}_{i,m+l+2p-i}$ in the second one. The fourth sum contains exactly the terms
  missing to the third one to pass from $g^{l,p}_{i,j}$ to $g^{l,p+1}_{i,j}$, so that we have
  \begin{align*}
    z_{m,l} &= \ts\sum_{i,j} g_{i,j}^{l,p+1} [\chi_1,z]_{i,j}-\sum_{i=0}^{m+2p}\phi^{l,p}_{m+p-i}(D^n_{i'+1} z_{i+1,i'+1} - D^n_{i} z_{i,i'+2}  \\
    &\hspace{8cm} + C^n_{i'+3}z_{i-1,i'+3} - C^n_{i+2} z_{i+2,i'}) \\
            &= \ts\sum_{i,j} g_{i,j}^{l,p+1} [\chi_1,z]_{i,j}+\sum_{i=0}^{m+2p+2} z_{i,i'+2}(\phi^{l,p}_{m+p-i}D^n_{i}-\phi^{l,p}_{m+p-i+1}D^n_{i'+2} \\
    &\hspace{8cm} + \phi^{l,p}_{m+p-i+2}C^n_{i} - \phi^{l,p}_{m+p-i-1} C^n_{i'+2}),
  \end{align*}
  still with $n = k+m+l+2p+2$, and using $\phi^{l,p}_i=0$ for $i<-p$ or $i>m+p$. We recognize in the last sum the definition of $f^{l,p+1}_{i,j}$, with
  $j = m+l+2p-i+2 = i'+2$, so that~\eqref{eq_iterated_move} holds for $p+1$.
\end{proof}

Now the main tool to obtain Theorem~\ref{thm_support_estimate}, together with the relations~\eqref{eq_iterated_move}, is an estimate on the coefficients $\phi$
that we establish at the elementary but technical Lemma~\ref{lem_technical_estimate} below. First we prove the following easy estimates on the coefficients $C$
and $D$.

\begin{lemma} \label{lemma_struct_coeff_estim} For any $a\in\NN$,
  $b\in\NN \cup \{-1\}$ and $k\in\NN^*$ we have, putting $n = a+b+k+2$:
  \begin{gather*}
    |D^n_{b+1}-D^n_{a}| \leq \frac{3q^{a+b+3}q^{-|b-a+1|}}{(1-q^{2n}) (1-q^{2n+2})}, \\
    |D^n_{b+1}| \leq 1 + \frac{2q^{n-a+b+1}}{(1-q^{2n}) (1-q^{2n+2})}, \qquad
    |C^n_{b+1}| \leq \frac{q^{2n-2b}}{(1-q^{2n}) (1-q^{2n+2})}.
  \end{gather*}
\end{lemma}

\begin{proof}
  We start from the identity
  \begin{displaymath}
    d_{n-b-2}d_{n-b-3} - d_{n-a-1}d_{n-a-2} = \pm d_{2n-a-b-3} d_{|b-a+1|-1}.
  \end{displaymath}
  This can be seen by a direct computation using $q$-numbers, or using the fusion rules as follows. Both sides vanish if $a = b+1$, according to our convention
  $d_{-1} = 0$ (and otherwise all indices are in $\NN$). Assume for instance $a < b+1$. Then $d_{n-b-2}d_{n-b-3}$ (resp.\ $d_{n-a-1}d_{n-a-2}$) is the sum of
  the dimensions $d_c$ where $c$ is odd and ranges from $1$ to $2n-2b-5$ (resp.\ $2n-2a-3$). On the other hand the right-hand side is the sum of dimensions
  $d_c$ where $c$ is odd and ranges from $2n-2b-3$ to $2n-2a-3$, so that the relation holds with a negative sign. The other case follows (with a positive sign)
  by exchanging $a$ and $b+1$. Dividing out by $d_nd_{n-1}$ and using the estimate $q^{-c}(1-q^{2c+2}) = (1-q^2)d_c \leq q^{-c}$ we obtain
  \begin{displaymath}
    |(1-A^n_{b+1}) - (1-A^n_{a})| = \frac{d_{2n-a-b-3} d_{|b-a+1|-1}}{d_nd_{n-1}}
    \leq \frac{q^{a+b+3-|b-a+1|}}{(1-q^{2n+2})(1-q^{2n})}.
  \end{displaymath}

  One can proceed in the same way with the constants $B$. By the same reasoning as above, or by a direct computation, we have
  \begin{displaymath}
    d_a d_{n-b-2} - d_{b+1} d_{n-a-1} = \pm d_n d_{|b-a+1|-1},
  \end{displaymath}
  indeed the products on the left are equal to the sum of the dimensions $d_c$ where $c$ has the same parity as $k$ and ranges from $k$ to $n+a-b-2$ (resp. from
  $k$ to $n-a+b$), so that in the difference $c$ ranges from $n+a-b$ to $n-a+b$ if $b>a-1$ (resp. from $n-a+b+2$ to $n+a-b-2$ if $b<a-1$), and we find exactly
  the same terms on the right. Dividing out by $d_nd_{n-1}$ this yields
  \begin{displaymath}
    |B^n_{a+1} - B^n_{b+2}| = 2 |\Re\mu| \frac{d_n d_{|b-a+1|-1}}{d_nd_{n-1}}
    \leq 2 \frac{q^{n-|b-a+1|}}{(1-q^{2n+2})(1-q^{2n})}.
  \end{displaymath}
  Since $n\geq a+b+3$, adding this estimate and the previous one yields the first estimate of the statement.

  The other estimates are easier. We have clearly $0\leq A^n_{b+1}\leq 1$, hence $D^n_{b+1}\leq 1+|B^n_{a+1}|$, and using again the estimates
  $(1-q^2)d_c \leq q^{-c}$ we obtain
  \begin{displaymath}
    |B^n_{a+1}| = 2 |\Re\mu| \frac{d_{a}d_{n-b-2}}{d_nd_{n-1}}
    \leq \frac{2q^{n-a+b+1}}{(1-q^{2n+2})(1-q^{2n})}.
  \end{displaymath}
  Finally the estimate for $C^n_{b+1}$ follows exactly like the one for
  $B^n_{a+1}$ above.
\end{proof}

\begin{lemma}\label{lem_technical_estimate}
  Assume that one can choose $3.4<R<0.995/(2q^2)$. Take $k\in\NN^*$, $|\mu| = 1$, and $m\in\NN$. Then there exists a constant $K$, depending only on $q$ and
  $m$, such that $|\phi^{l,p}_i| \leq K q^{2|i|}(2R)^{i}$ for all $l\geq m$, $p\in\NN$ and $-p\leq i\leq m+p$.
\end{lemma}

\begin{proof}
  Note that the assumption implies the inequalities $6q^2<2Rq^2<1$, and in particular $q^2<1/6$, which we will use frequently in this proof.  The
  recursive construction of the coefficients $f$ yields the following recursion relation over $p$ for the coefficients $\phi$. For $-p-1\leq i\leq m+p+1$ we
  have, putting $n = m+l+2p+k+2$:
  \begin{align*}
    \phi^{l,p+1}_i
    &= -\sum_{s=-p-1}^i f_{m+p+1-s,l+p+1+s}^{l,p+1} \\
    &= \sum_{s=-p-1}^i  (\phi^{l,p}_{s}D^n_{l+p+1+s} - \phi^{l,p}_{s-1}D^n_{m+p+1-s}
      + \phi^{l,p}_{s-2} C^n_{l+p+1+s} - \phi^{l,p}_{s+1}C^n_{m+p+1-s}) \\
    &= \phi^{l,p}_i (D^n_{l+p+1+i} - C^n_{m+p-i+2}) - \phi^{l,p}_{i+1} C^n_{m+p-i+1} 
    \\ & \hspace{2cm} + \sum_{s=-p}^{i-1} \phi^{l,p}_s
         \Bigr(D^n_{l+p+1+s} - D^n_{m+p-s} + \delta_{s\leq i-2}C^n_{l+p+s+3} - C^n_{m+p-s+2}\Bigl),
  \end{align*}
  since $\phi^{l,p}_s = 0$ if $s<-p$.

  We now combine this recursion relation with the estimates of
  Lemma~\ref{lemma_struct_coeff_estim}. We still take $n = k+m+l+2p+2$ and we
  denote $\rho_n = (1-q^{2n})^{-1}(1-q^{2n+2})^{-1}$.  Note that we have
  $\rho_n \leq \rho_t$ if $t\leq n$, and $\rho_3 < 1.005$ (by comparing with the
  value at $q^2 = 0.995/6.8$). Since $k\geq 1$, we have $n \geq 3$, hence the
  estimate $\rho_n \leq 1.005$ that we will use later in the proof. Now
  Lemma~\ref{lemma_struct_coeff_estim} gives, for $-p-1\leq i \leq m+p+1$:
  \begin{align*}
    |\phi^{l,p+1}_i|
    &\leq |\phi^{l,p}_i| \Bigl(1+2\rho_nq^{n-m+l+2i+1}+\rho_nq^{2n-2m-2p+2i-2}\Bigr)
      + |\phi^{l,p}_{i+1}| \rho_n q^{2n-2m-2p+2i}
    \\ & \hspace{1em}
         + \sum_{s=-p}^{i-1} |\phi^{l,p}_s| \times \bigl( 3\rho_n q^{m+l+2p+3-|l-m+2s+1|}
         + \rho_nq^{2n-2l-2p-2s-4}+\rho_n q^{2n-2m-2p+2s-2}\bigr).
  \end{align*}
  Observe that the sum vanishes for $i = -p-1$, as well as $\phi^{l,p}_{-p-1}$,
  $\phi^{l,p}_{m+p+1}$ and $\phi^{l,p}_{m+p+2}$ by convention. Since $k\geq 1$
  and $l\geq m$, and using the value of $n$, we have the following lower bounds for $-p\leq s\leq i-1$:
  \begin{itemize}
  \item[\textbullet] $n-m+l+2i+1\geq 2(p+i+m+2)$,
  \item[\textbullet] $2n-2m-2p+2i \geq 2(p+i+m+3)$,
  \item[\textbullet] $m+l+2p+3-|l-m+2s+1| \geq 2(p-|s|+m+1)$,
  \item[\textbullet] $2n-2l-2p-2s-4 \geq 2(p-|s|+m+1)$,
  \item[\textbullet] $2n-2m-2p+2s-2 \geq 2 (p-|s|+m+2)$.
  \end{itemize}
  Applying these bounds and factoring $3+1+q^2 \leq \frac{25}6$ in the sum we arrive at the slightly simpler estimate:
  \begin{align*}
    |\phi^{l,p+1}_i|
    &\leq |\phi^{l,p}_i| \bigl(1+ 3\rho_nq^{2(p+i+m+2)}\bigr)
      + |\phi^{l,p}_{i+1}| \rho_n q^{2(p+i+m+3)}
      + {\ts\frac{25}6} \rho_n\sum_{s=-p}^{i-1} |\phi^{l,p}_s|  q^{2(p-|s|+m+1)}.
  \end{align*}
  
  Then we denote $\psi^{l,p}_i = |\phi^{l,p}_i| q^{-2|i|}(2R)^{-i}$, so that our
  aim is now to find $K$ such that $\psi^{l,p}_i \leq K$ for all $l$, $p\in\NN$,
  $-p\leq i\leq m+p$.  For $\psi$ the previous estimate becomes
  \begin{align*}
    \psi^{l,p+1}_i
    &\leq \psi^{l,p}_i \bigl(1 + 3\rho_nq^{2(p+i+m+2)}\bigr)
      + \psi^{l,p}_{i+1} \rho_n 2R q^{2(p+i+m+2)}
      + {\ts\frac{25}6} \rho_n q^{2(p-|i|+m+1)} \sum_{s=-p}^{i-1} (2R)^{s-i}\psi^{l,p}_s ,
  \end{align*}
  where we have used $|i+1|-|i| \geq -1$ in the second term.  Let us denote
  $K_r = (6q^2)^{-m} \times \prod_{t=0}^{r-1} {(1 + q^{2t})}$, which is increasing with $r$ and starts with
  $K_0 = (6q^2)^{-m}$. We will prove, for each $l\geq m$ by induction on $p$, the
  following estimate:
  \begin{displaymath}
    (H_p) \quad \forall i\in\{-p,\ldots,p+m\} \quad \psi^{l,p}_i \leq K_{p+m-|i|}.
  \end{displaymath}
  For $p=0$ and $0\leq i\leq m$ we have indeed $\phi^{l,p}_i = -1$, hence
  $\psi^{l,p}_i = (2Rq^2)^{-i} \leq (6q^2)^{-m} \leq K_{m-|i|}$.  Assume now that the
  estimates hold for a fixed $p\in\NN$, and let us establish them at
  $p+1$. Together with $(H_p)$, our recursive estimate on $\psi$ yields, for
  $-p-1\leq i\leq m+p+1$:
  \begin{align} \nonumber
    \psi^{l,p+1}_i
    &\leq \delta_{-p\leq i\leq p+m} K_{p+m-|i|} \bigl(1 + 3\rho_nq^{2(p+i+m+2)}\bigr)
        + \delta_{i\leq p+m-1}K_{p+m-|i+1|} \rho_n 2R q^{2(p+i+m+2)} \\
    &\hspace{6cm} \label{eq_rec_estimate} +
       {\ts\frac{25}6} \rho_n q^{2(p-|i|+m+1)} \sum_{s=-p}^{i-1} (2R)^{s-i} K_{p+m-|s|}.
  \end{align}

  Consider first $i = -p-1$. The above estimate reads in this case
  $\psi^{l,p+1}_{-p-1} \leq K_m \rho_n 2Rq^{2(m+1)}$. Since $2Rq^2<0.995$
   and $\rho_n\leq 1.005$, we obtain $\psi^{l,p+1}_{-p-1} \leq K_m$
  as needed.

  Then we consider the case $-p\leq i\leq -1$. For $s\leq i$ we have
  $K_{p+m-|s|} = K_{p+m+s} \leq K_{p+m-|i|}$. Moreover we have
  $K_{p+m-|i+1|} 2R q^{4} = K_{p+m-|i|} (1+q^{2(p+m+i)})2Rq^4 \leq \frac 13
  K_{p+m-|i|}$ because $2Rq^2<1$ and
  $q^2(1+q^{2(p+m+i)}) \leq 2q^2\leq\frac 13$.  Hence~\eqref{eq_rec_estimate}
  yields
  \begin{align*}
    \psi^{l,p+1}_i
    &\leq K_{p+m-|i|} \bigl( 1+{\ts\frac 3{36}}\rho_nq^{2(p+i+m)} + {\ts\frac 13}\rho_n q^{2(p+i+m)}
      + {\ts\frac{25}{36}}\rho_n q^{2(p-|i|+m)} \ts\sum_{s=-\infty}^{i-1} (2R)^{s-i}\bigr) \\
    &\leq K_{p+m-|i|} \bigl( 1+{\ts\frac 5{12}}\rho_nq^{2(p+i+m)}
      + {\ts\frac{25}{36}} \rho_n q^{2(p+m-|i|)}/(2R-1) \bigr) \\
    &\leq K_{p+m-|i|} \bigl( 1+{\ts\frac 59}\rho_nq^{2(p+m-|i|)} \bigr)
      \leq K_{p+m-|i|} \bigl( 1 + q^{2(p+m-|i|)} \bigr) = K_{p+1+m-|i|},
  \end{align*}
  where we used $2R-1\geq 5$ and $\rho_n \leq \frac 95$.

  Now we consider the case $0\leq i\leq p+m$. Let us observe that for any
  $t\in \NN$ we have $(1+q^{2t})/2 \leq 1$, hence
  $2^{-t}\prod_{r=0}^{t-1}(1+q^{2r}) \leq 1$. Then for $|s|\leq i$ we can write
  $K_{p+m-|s|} 2^{s-i} = K_{p+m-i} 2^{s-i}
  \prod_{t=p+m-i}^{p+m-|s|-1}(1+q^{2t}) \leq K_{p+m-i} 2^{|s|-i}
  \prod_{r=0}^{i-|s|-1}(1+q^{2r}) \leq K_{p+m-i}$. On the other hand for
  $s\leq -i$ we clearly have $K_{p+m-|s|} \leq K_{p+m-i}$ (and $2^{s-i}\leq
  1$). Using this, our estimate~\eqref{eq_rec_estimate} thus yields
  \begin{align*}
    \psi^{l,p+1}_i
      &\leq K_{p+m-i}\bigl(1+3\rho_nq^{2(p+i+m+2)}
        +  \rho_n 2R q^{2(p+i+m+2)} 
        + {\ts\frac{25}6} \rho_n q^{2(p-i+m+1)} \ts\sum_{s=-p}^{i-1} R^{s-i}\bigr) \\
      &\leq K_{p+m-i}\bigl(1+q^{2(p-i+m)}\bigl({\ts\frac 3{36}}\rho_n
        +  {\ts\frac 1{6}}\rho_n 
        + {\ts\frac{25}{36}} \rho_n/ (R-1)\bigr)\bigr) \\
      &\leq K_{p+m-i} (1+{\ts\frac{43}{72}}\rho_n q^{2(p+m-i)})
        \leq K_{p+m-i} (1+q^{2(p+m-i)}) = K_{p+1+m-i},
  \end{align*}
  where we used $2Rq^2 < 1$, $q^2\leq \frac 16$, then $R-1\geq 2$
  and $\rho_n\leq\frac{72}{43}$.
  
  Finally when $i = p+m+1$ the first two terms in the
  estimate~\eqref{eq_rec_estimate} vanish and we are left with the sum which can
  be dealt with as before:
  \begin{align*}
    \psi_{p+m+1}^{l,p+1}
    &\leq  {\ts\frac{25}6} \rho_n  \ts\sum_{s=-p}^{p+m} (2R)^{s-p-m-1}K_{p+m-|s|} \\
    &\leq  {\ts\frac{25}6}K_0 \rho_n  \ts\sum_{s=-p}^{p+m} R^{s-p-m-1}2^{-1}
    \leq {\ts\frac{25}6} K_0 \rho_n /(2(R-1)) \leq K_0,
  \end{align*}
  since $R\geq 3.4$ and $\rho_n\leq 1.1$. This is the required estimate to conclude
  the proof of $(H_{p+1})$.
  
  We have now proved by induction that $(H_p)$ holds for all $p\in\NN$. Moreover
  we have $K_r \leq K := \lim_{s\to\infty} K_s$ for all $r\in\NN$, with
  $K<+\infty$ because $q<1$. Hence the lemma is proved.
\end{proof}

\begin{notation}
  We consider the following {\em non-orthogonal} projections $E_m$, $Q_m$  defined as follows: for all $w\in \W$, $i$, $j\in\NN$
  \begin{align*}
    E_m(w_{i,j}) &= w_{i,j} \quad \text{if $i\geq m$ and $j\geq m$,}  \quad 0 \quad\text{otherwise ;} \\
    Q_m(w_{i,j}) &= w_{i,j} \quad \text{if $j\geq i = m$,}  \quad 0 \quad\text{otherwise.}
  \end{align*}
  Observe that if $\{w_{i,j} \mid w\in W, i, j\in\NN\}$ is a Riesz basis, these projections extend to idempotents in $B(H^\circ)$, and the range of $E_m$ is the subspace $V_m$ from Notation~\ref{not_Vm}.
\end{notation}

\begin{theorem} \label{thm_support_estimate} Assume that $\{w_{i,j} \mid w\in W, i, j\in\NN\}$ is a Riesz basis and that $N\geq 3$.  Then there exist 
  constants $L_m$ such that we have, for any $p\in\NN^*$, $m\in\NN$ and $z\in H^\circ$:
  \begin{displaymath}
    \|(1-E_m)z\|_2^2 \leq L_m p^{-1}\|z\|_2^2 + L_m p\|[\chi_1,z]\|_2^2 .
  \end{displaymath}
\end{theorem}


\begin{proof}
  Note that for $N=3$ we have $q^{-2} = \frac 12(N^2+N\sqrt{N^2-4}-2) \simeq 6.854$, so that for $N\geq 3$ we have $q^{-2}> 6.85$ and $0.995/(2q^2) > 3.4$. As a
  result we can find a number $R$ such as in the hypothesis of Lemma~\ref{lem_technical_estimate}.  To start with, we deduce from that lemma estimates on
  various sums of coefficients $f$ and $g$. For each $m$ we denote $K_m$ the constant provided by the lemma.

  Recall that $f_{i,j}^{l,p}$ vanishes except for the following entries: $f^{l,p}_{m+p-i,l+p+i} = \phi^{l,p}_{i-1}-\phi^{l,p}_i$ with $-p\leq i\leq p+m$, and
  the convention $\phi^{l,p}_{-p-1} = 0$. Thus for fixed $l\geq m$ and $p$ we have $\sum_{i,j} |f^{l,p}_{i,j}| \leq 2 \sum_{i=-p}^{p+m} |\phi^{l,p}_i|$. As a
  result Lemma~\ref{lem_technical_estimate} yields $\sum_{i,j} |f^{l,p}_{i,j}| \leq 2K_m S$ with $S = \sum_{i=-\infty}^{+\infty} q^{2|i|}(2R)^i$, which is
  finite because $2Rq^2<1$ and $q^2/2R<q^2/6<1$ by choice of $R$. In the same way for fixed $i$, $j$ we have $f^{l,p}_{i,j} = 0$ unless $i+j\geq m$ and $l\geq m$,
  $p\in\NN$ satisfy $l = i+j-m-2p$. Thus we obtain, putting $r = p+m-i$, the same upper bound for the sum over $l$ and $p$:
  \begin{align*}
    \sum_{p,l\geq m}|f^{l,p}_{i,j}|
    &= \sum_{p=0}^{\lfloor (i+j)/2-m\rfloor}
      \bigl|\phi^{i+j-m-2p,p}_{m+p-i} - \phi^{i+j-m-2p,p}_{m+p-i-1}\bigr| \\
    &= \sum_{r=m-i}^{\lfloor (j-i)/2\rfloor}
      \bigl|\phi^{m+j-i-2r,r+i-m}_{r} - \phi^{m+j-i-2r,r+i-m}_{r-1}\bigl|
      \leq 2K_m S.
  \end{align*}

  On the other hand, recall that if $g^{l,p}_{i,j} \neq 0$ then there exists
  $0\leq r\leq p-1$ such that $i+j = m+l+2r+1$, and then
  $g^{l,p}_{i,j} = \phi^{l,r}_{r+m-i}$. Hence for fixed $l\geq m$ and $p$ we have
  \begin{align*}
  \sum_{i,j}|g^{l,p}_{i,j}| \leq \sum_{r=0}^{p-1}
  \sum_{i=0}^{m+l+2r+1}|\phi^{l,r}_{r+m-i}| \leq
  \sum_{r=0}^{p-1} \sum_{i=-\infty}^{+\infty}|\phi^{l,r}_i| \leq
  pK_mS.
  \end{align*}
  Similarly, for fixed $i$, $j$, $p$ we have $\sum_{l\geq m}|g^{l,p}_{i,j}| = \sum_{r=0}^\rho |\phi^{i+j-m-2r-1,r}_{r+m-i}|$ where
  $\rho = \min(p-1,$ ${\lfloor(i+j-1)/2\rfloor}-m)$, hence once again $\sum_{l\geq m}|g^{l,p}_{i,j}| \leq K_mS$.

  Now we can proceed to the main part of the proof. We start by $Q_m$ instead of $1-E_m$. By decomposing $H^\circ$ into the pairwise orthogonal sub-bimodules
  $H(w)$ we can assume that $z$ belongs to $H(w)$ for some $w$ --- indeed $Q_m$ and the commutator with $\chi_1$ commute with the projections onto these
  submodules.  Since $(w_{i,j})$ is a Riesz basis we can replace $\|Q_mz\|_2^2$ with $\sum_{l\geq m} |z_{m,l}|^2$, $\|z\|_2^2$ with $\sum_{i,j}|z_{i,j}|^2$ and
  $\|[\chi_1,z]\|_2^2$ with $\sum_{i,j}|[\chi_1,z]_{i,j}|^2$. Then for any fixed $p$ we have the following estimates, using
  Proposition~\ref{prop_commut_relations} and Cauchy-Schwartz:
  \begin{align*}
    \sum_{l\geq m} |z_{m,l}|^2
    &= \sum_{l\geq m} \Bigl|\sum_{i,j} f_{i,j}^{l,p} z_{i,j}
      + \sum_{i,j} g_{i,j}^{l,p} [\chi_1,z]_{i,j}\Bigr|^2 \\
    &\leq 2\sum_{l\geq m} \Bigl|\sum_{i,j} f_{i,j}^{l,p} z_{i,j}\Bigr|^2
      + 2\sum_{l\geq m} \Bigl|\sum_{i,j}  g_{i,j}^{l,p} [\chi_1,z]_{i,j}\Bigr|^2 \\
    &\leq 2\sum_{l\geq m} \sum_{i,j} |f_{i,j}^{l,p}|\sum_{i,j} |f_{i,j}^{l,p}z_{i,j}^2| +
      2\sum_{l\geq m}\sum_{i,j} |g_{i,j}^{l,p}|\sum_{i,j} |g_{i,j}^{l,p}[\chi_1,z]_{i,j}^2| \\
    &\leq 4K_mS\sum_{l\geq m,i,j} |f_{i,j}^{l,p}z_{i,j}^2|
      + 2pK_mS\sum_{i,j} |[\chi_1,z]_{i,j}|^2 \sum_{l\geq m}|g_{i,j}^{l,p}| \\
    &\leq 4K_mS\sum_{l\geq m,i,j} |f_{i,j}^{l,p}z_{i,j}^2| + 2pK_m^2S^2\sum_{i,j} |[\chi_1,z]_{i,j}|^2.
  \end{align*}
  Then we take the average of these inequalities over $p = 0, \ldots, r-1$:
  \begin{align*}
    \sum_{l\geq m} |z_{m,l}|^2
    &\leq \frac{4K_mS}{r}\sum_{i,j}|z_{i,j}^2|\sum_{p<r,l\geq m} |f_{i,j}^{l,p}|
      + \frac{2K_m^2S^2}r \sum_{p<r}p\sum_{i,j} |[\chi_1,z]_{i,j}|^2, \\
    &\leq \frac{8K_m^2S^2}{r}\sum_{i,j} |z_{i,j}|^2
      + rK_m^2S^2\sum_{i,j} |[\chi_1,z]_{i,j}|^2.
  \end{align*}
  It is then easy to upgrade this estimate from $Q_m$ to $1-E_n$. First of all by symmetry we have the same estimate for the sum of $|z_{l,m}|^2$ over $l>m$,
  for fixed $m$. Then for fixed $n$ we put $L_n = \sum_{m=0}^{n-1}K_m^2$, and by summing over $0\leq m\leq n-1$ we obtain for all $r$:
  \begin{displaymath}
    \sum_{l<n ~\text{or}~ m<n} |z_{m,l}|^2
    \leq 16L_nS^2 r^{-1} \sum_{i,j} |z_{i,j}|^2
      + 2L_nS^2 r \sum_{i,j} |[\chi_1,z]_{i,j}|^2.
    \end{displaymath}
    This gives the required estimate by the Riesz basis property.
\end{proof}

\begin{proof}[\bf Proof of Theorem~A]
  Take the constant $q_1$ provided by Theorem~\ref{thm_riesz}, and $N_0\in\NN$, $N_0\geq 3$, such that the associated constant $q_0$ satisfies $q_0<q_1$. Then
  for $N\geq N_0$ we have $q<q_1$, so that $\{w_{i,j} \mid w\in W, i, j\in\NN\}$ is a Riesz basis.

  To prove the AOP, take elements $z_r\in A^\bot\cap M$ such that $\|z_r\|\leq 1$ and $\|[\chi_1,z_r]\|_2\to_\omega 0$. We want to prove that
  $(yz_r\mid z_ry) \to_\omega 0$ for any $y\in A^\bot\cap M$. By Kaplansky's density theorem and linearity, we can assume that $y\in p_n H^\circ\cap M$ for some
  fixed $n\in\NN^*$, with $\|y\|\leq 1$. Now for any $m\in\NN$ we can write
  \begin{displaymath}
    |(yz_r\mid z_ry)| \leq |(y E_m(z_r)\mid E_m(z_r)y)| + \|(1-E_m)(z_r)\|_2(\|z_r\|_2+\|E_m(z_r)\|_2).
  \end{displaymath}
  We apply our Theorem~\ref{thm_ortho_global} to $\zeta = E_m(z_r)$, obtaining $|(y E_m(z_r)\mid E_m(z_r)y)|\leq C q^{\alpha m} \|E_m(z_r)\|_2^2$ for
  $m\geq 10 n$. The projections $E_m$ are not orthogonal, but since $(w_{i,j})_{w,i,j}$ is a Riesz basis they are bounded independently of $m$. Thus we get
  \begin{displaymath}
    |(yz_r\mid z_ry)| \leq C q^{\alpha m} + C \|(1-E_m)(z_r)\|_2.
  \end{displaymath}
  for some new constant $C$ independent of $m$ and $r$.
  We now take $\epsilon>0$ and choose a fixed $m\geq 10 n$ such that $C q^{\alpha m} \leq \epsilon/2$. Then we apply Theorem~\ref{thm_support_estimate}: for all $p\in\NN^*$ and $r$ we have
  \begin{displaymath}
    \|(1-E_m)(z_r)\|_2^2 \leq L_mp^{-1} + L_mp \|[\chi_1,z_r]\|_2^2.
  \end{displaymath}
  We choose $p$ such that
  $L_mp^{-1} \leq \epsilon^2/8C^2$. Finally by assumption for $\omega$-almost all $r$ we have $\|[\chi_1,z_r]\|_2^2 \leq \epsilon^2 /8C^2L_mp$. The above estimates then
  show that for the same $r$'s we have ${|(yz_r\mid z_ry)|}\leq\epsilon$, and this concludes the proof of the AOP.  
\end{proof}

\subsection*{Acknowledgments}

The authors wish to thank the anonymous referee for the careful reading of this article.  R.V.\ and X.W.\ were partially supported by the french Agence
Nationale de la Recherche (grant ANR-19-CE40-0002). X.W.\ was partially supported by the CEFIPRA project 6101-1, the National Research Foundation of Korea
(grants NRF-2022R1A2C1092320 and NRF-2020R1C1C1A01009681), and the National Natural Science Foundation of China (grants No.~12031004 and No.~W2441002).

\bibliographystyle{alpha}

\begin{thebibliography}{CFRW10}

\bibitem[Ban96]{Banica_LibreOrtho}
Teodor Banica.
\newblock Th\'{e}orie des repr\'{e}sentations du groupe quantique compact libre
  {${\rm O}(n)$}.
\newblock {\em C. R. Acad. Sci. Paris S\'{e}r. I Math.}, 322(3):241--244, 1996.

\bibitem[Ban97]{Banica_LibreUnitaire}
Teodor Banica.
\newblock Le groupe quantique compact libre {${\rm U}(n)$}.
\newblock {\em Comm. Math. Phys.}, 190(1):143--172, 1997.

\bibitem[BC18]{BrannanCollins_Entangled}
Michael Brannan and Beno\^{\i}t Collins.
\newblock Highly entangled, non-random subspaces of tensor products from
  quantum groups.
\newblock {\em Comm. Math. Phys.}, 358(3):1007--1025, 2018.

\bibitem[BV18]{BrannanVergnioux}
Michael Brannan and Roland Vergnioux.
\newblock Orthogonal free quantum group factors are strongly 1-bounded.
\newblock {\em Adv. Math.}, 329:133--156, 2018.

\bibitem[BVY21]{BrannaVergniouxYoun}
Michael Brannan, Roland Vergnioux, and Sang-Gyun Youn.
\newblock Property {RD} and hypercontractivity for orthogonal free quantum
  groups.
\newblock {\em Int. Math. Res. Not. IMRN}, (2):1573--1601, 2021.

\bibitem[CFRW10]{CFRW_RadialMaxAmenable}
Jan Cameron, Junsheng Fang, Mohan Ravichandran, and Stuart White.
\newblock The radial masa in a free group factor is maximal injective.
\newblock {\em J. Lond. Math. Soc. (2)}, 82(3):787--809, 2010.

\bibitem[DL84]{DoplicherLongo}
Sergio Doplicher and Roberto Longo.
\newblock Standard and split inclusions of von {N}eumann algebras.
\newblock {\em Invent. Math.}, 75(3):493--536, 1984.

\bibitem[FK97]{FrenkelKhovanov_Graphical}
Igor Frenkel and Mikhail Khovanov.
\newblock Canonical bases in tensor products and graphical calculus for
  {$U_q({\mathfrak s}{\mathfrak l}_2)$}.
\newblock {\em Duke Math. J.}, 87(3):409--480, 1997.

\bibitem[FV16]{FreslonVergnioux}
Amaury Freslon and Roland Vergnioux.
\newblock The radial {MASA} in free orthogonal quantum groups.
\newblock {\em J. Funct. Anal.}, 271(10):2776--2807, 2016.

\bibitem[Iso15]{Isono_NoCartan}
Yusuke Isono.
\newblock Examples of factors which have no {C}artan subalgebras.
\newblock {\em Trans. Amer. Math. Soc.}, 367(11):7917--7937, 2015.

\bibitem[KW22]{KrajczokWasilewsi}
Jacek Krajczok and Mateusz Wasilewski.
\newblock On the von {N}eumann algebra of class functions on a compact quantum
  group.
\newblock {\em J. Funct. Anal.}, 283(5):Paper No. 109549, 29, 2022.

\bibitem[NT13]{NeshveyevTuset_Book}
Sergey Neshveyev and Lars Tuset.
\newblock {\em Compact quantum groups and their representation categories},
  volume~20 of {\em Cours Sp\'{e}cialis\'{e}s [Specialized Courses]}.
\newblock Soci\'{e}t\'{e} Math\'{e}matique de France, Paris, 2013.

\bibitem[Pop83]{Popa_MaxInjective}
Sorin Popa.
\newblock Maximal injective subalgebras in factors associated with free groups.
\newblock {\em Adv. in Math.}, 50(1):27--48, 1983.

\bibitem[Pyt81]{Pytlik_Radial}
Tadeusz Pytlik.
\newblock Radial functions on free groups and a decomposition of the regular
  representation into irreducible components.
\newblock {\em J. Reine Angew. Math.}, 326:124--135, 1981.

\bibitem[R{\u a}d91]{Radulescu_Radial}
Florin R{\u a}dulescu.
\newblock Singularity of the radial subalgebra of {${\mathcal L}(F_N)$} and the
  {P}uk\'{a}nszky invariant.
\newblock {\em Pacific J. Math.}, 151(2):297--306, 1991.

\bibitem[VD95]{VanDaele_Haar}
Alfons Van~Daele.
\newblock The {H}aar measure on a compact quantum group.
\newblock {\em Proc. Amer. Math. Soc.}, 123(10):3125--3128, 1995.

\bibitem[Ver07]{Vergnioux_RD}
Roland Vergnioux.
\newblock The property of rapid decay for discrete quantum groups.
\newblock {\em J. Operator Theory}, 57(2):303--324, 2007.

\bibitem[VV07]{VaesVergnioux}
Stefaan Vaes and Roland Vergnioux.
\newblock The boundary of universal discrete quantum groups, exactness, and
  factoriality.
\newblock {\em Duke Math. J.}, 140(1):35--84, 2007.

\bibitem[Wan95]{Wang_FreeProd}
Shuzhou Wang.
\newblock Free products of compact quantum groups.
\newblock {\em Comm. Math. Phys.}, 167(3):671--692, 1995.

\bibitem[Wen87]{Wenzl_Projections}
Hans Wenzl.
\newblock On sequences of projections.
\newblock {\em C. R. Math. Rep. Acad. Sci. Canada}, 9(1):5--9, 1987.

\bibitem[Wor98]{Woronowicz_CQG}
Stanis\l aw~L. Woronowicz.
\newblock Compact quantum groups.
\newblock In {\em Sym\'{e}tries quantiques ({L}es {H}ouches, 1995)}, pages
  845--884. North-Holland, Amsterdam, 1998.

\end{thebibliography}
\renewcommand{\bibliofont}{\normalsize}


\end{document}